\documentclass{amsart}

\usepackage{amssymb,amsthm,amsmath}
\usepackage{mathtools}
\usepackage{enumitem}
\usepackage{graphicx}
\usepackage{float}
\usepackage{microtype}
\usepackage{hyperref}

\theoremstyle{plain}
\newtheorem{prop}{Proposition}[section]
\newtheorem{lemma}[prop]{Lemma}
\newtheorem{thm}[prop]{Theorem}

\theoremstyle{definition}
\newtheorem{defi}[prop]{Definition}

\theoremstyle{remark}
\newtheorem{remark}[prop]{Remark}

\newcommand{\pa}{\partial} 
\newcommand{\iu}{\mathrm{i}\mkern1mu}
\newcommand{\odd}{\mathrm{odd}}
\newcommand{\Log}{\operatorname{Log}}
\newcommand{\fu}{\mathbf I}
\newcommand{\futoo}{\mathbf K}
\newcommand{\fuPsi}{\mathbf\Psi}
\newcommand{\PI}{\mathbf P_I}
\newcommand{\PK}{\mathbf P_K}
\newcommand{\mcJ}{\mc J\!}

\makeatletter
\@namedef{subjclassname@2020}{%
  \textup{2020} Mathematics Subject Classification}
\makeatother

\makeatletter

\renewcommand\subsubsection{\@secnumfont}{\bfseries}%
\renewcommand\subsubsection{\@startsection{subsubsection}{3}
  \z@{.5\linespacing\@plus.7\linespacing}{-.5em}%
  {\normalfont\bfseries}}
  
\makeatother

\newcommand{\stirmat}{U}
\newcommand{\stirelem}{u}

\newcommand{\cone}{\mf p}
\newcommand{\ctwo}{\mf q}
\newcommand{\cthree}{\mf r}

\newcommand{\setmid}{\;:\;}
\DeclareMathOperator{\Coeff}{Coeff}

\DeclareMathOperator{\End}{End}
\DeclareMathOperator{\Mat}{Mat}
\DeclareMathOperator{\GL}{GL}
\DeclareMathOperator{\SL}{SL}
\DeclareMathOperator{\SU}{SU}
\DeclareMathOperator{\PSL}{PSL}
\DeclareMathOperator{\AdS}{AdS}

\DeclareMathOperator{\Ima}{Im}
\DeclareMathOperator{\Rea}{Re}
\DeclareMathOperator{\sgn}{sgn}

\newcommand\N{\mathbb{N}}

\newcommand\R{\mathbb{R}}
\newcommand\Z{\mathbb{Z}}
\newcommand\C{\mathbb{C}}
\newcommand{\h}{\mathbb{H}}

\newcommand{\mc}[1]{\mathcal #1}
\newcommand{\mf}[1]{\mathfrak #1}
\newcommand{\wh}{\widehat}
\newcommand{\eps}{\varepsilon}

\DeclareMathOperator{\id}{id}
\newcommand{\mat}[4]{\begin{pmatrix} #1&#2\\#3&#4\end{pmatrix}}
\newcommand{\textmat}[4]{\left(\begin{smallmatrix} #1&#2 \\ #3&#4 \end{smallmatrix}\right)}

\begin{document}

\title[Fourier expansions for non-unitary twists]{Fourier expansions of 
vector-valued automorphic functions with non-unitary twists}
\author[K.~Fedosova]{Ksenia Fedosova}
\address{Ksenia Fedosova, Albert-Ludwigs-Universit\"at Freiburg, Mathematisches 
Institut, Ernst-Zermelo-Str.~1, 79104 Freiburg im Breisgau, Germany}
\email{ksenia.fedosova@math.uni-freiburg.de}
\author[A.~Pohl]{Anke Pohl}
\address{Anke Pohl, University of Bremen, Department~3 -- Mathematics, Institute for Dynamical Systems, Bibliothekstr.~5,  28359 Bremen, Germany}
\email{apohl@uni-bremen.de}
\author[J.~Rowlett]{Julie Rowlett}
\address{Julie Rowlett, Mathematical Sciences, Chalmers University of 
Technology, 41296 Gothenburg, Sweden}
\email{julie.rowlett@chalmers.se}

\begin{abstract} 
We provide Fourier expansions of vector-valued eigenfunctions of the hyperbolic Laplacian that are twist-periodic in a horocycle direction.
The twist may be given by any endomorphism of a finite-dimensional vector space; no assumptions on invertibility or unitarity are made.
Examples of such eigenfunctions include vector-valued twisted automorphic forms of Fuchsian groups. 
We further provide a detailed description of the Fourier coefficients and explicitly identify each of their constituents, which intimately depend on the eigenvalues of the twisting endomorphism and the size of its Jordan blocks. 
In addition, we determine the growth properties of the Fourier coefficients.
\end{abstract}

\maketitle
\tableofcontents

\section{Introduction}\label{s:intro}

For complex-valued functions on the hyperbolic plane that are simultaneously eigenfunctions of the hyperbolic Laplacian and periodic in a horocycle 
direction\footnote{We refer to Section~\ref{sec:setting} for precise 
definitions. Using the upper half plane model~\eqref{eq:intro_uhp}, we may 
restrict to Laplace eigenfunctions $f\colon \h\to\C$ that are periodic in the 
real direction with period~$1$, thus $f(z) = f(z+1)$ for all $z\in\h$.} it is 
not only known that they admit a Fourier expansion in the direction of 
periodicity but also rather precisely how the Fourier coefficients behave in 
the orthogonal direction. This deep knowledge about the Fourier coefficients is 
indispensable in many studies of the properties of automorphic functions and 
forms and their applications within mathematics and physics. However, not only 
\emph{periodic} Laplace eigenfunctions but also Laplace eigenfunctions with 
certain \emph{controlled non-periodicities} occur naturally in many areas, most 
notably in spectral theory, number theory and mathematical physics. 
Further below, in Section~\ref{sec:motivation}, we provide a few instances of such occurances.

These and further applications demand detailed information about the fine structures of Laplace eigenfunctions with the type of controlled 
non-periodicities (for short: \emph{twists}) as considered in the present 
article. These demands and, in addition, the idea---as promoted already by 
Selberg~\cite{Selberg_est_fourier}---to investigate \emph{twisted} automorphic 
functions alongside their \emph{untwisted} (periodic) relatives and to exploit their interactions in order to understand their properties inspires the quest for \mbox{Fourier(-type)} expansions of vector-valued eigenfunctions of the Laplacian that are not necessarily periodic but rather \emph{twist-periodic} with a twist given by a, potentially non-unitary, endomorphism of a vector space.
With this paper we develop such Fourier expansions for \emph{all} twist-periodic vector-valued Laplace eigenfunctions and \emph{all} finite-dimensional complex vector spaces.
We give a detailed description of their Fourier coefficients, explicitly identifying each of their constituents and precisely determining their relevant growth properties.  
In what follows we provide a more detailed account of our main results, our motivation as well as an indication of the structural differences from the classical, untwisted case and the unitarily twisted case, including a brief discussion of the relation of our results to those in these more classical situations.

\subsection{Main results}
We use the upper half plane model
\begin{equation}\label{eq:intro_uhp}
\h \coloneqq \{ z\in\C \setmid \Ima z>0\}\,,\quad ds_z^2 \coloneqq
\frac{dz\,d\overline z}{(\Ima z)^2}
\end{equation}
of the hyperbolic plane, and we fix a finite-dimensional complex vector 
space~$V$ as well as an endomorphism $A\in\End(V)$ of $V$. We emphasize 
that~$A$ 
is \emph{not} required to be unitary. We let 
\[
\Delta \coloneqq -y^2 \bigl( \partial_x^2 + \partial_y^2\bigr)\,,\qquad 
(z=x+\iu y\in\h\,, x\in\R\,, y>0)\,,
\]
denote the hyperbolic Laplacian in the upper half plane model, and we 
consider~$\Delta$ as an operator on smooth $V$-valued functions on~$\h$.
We are interested in the Laplace eigenfunctions $f\colon\h\to V$, thus
\begin{equation}\label{eq:Delta_eigen_intro}
\Delta f = s(1-s)f\,,
\end{equation}
that satisfy the \emph{twist-periodicity} condition 
\begin{equation}\label{eq:twist_intro}
f(z+1) = A f(z)
\end{equation}
for all~$z\in\h$. We note that we could define the Laplacian~$\Delta$ on a much 
larger domain, e.g., $V$-valued eigendistributions. The elliptic regularity 
of~$\Delta$, however, immediately implies that all its eigendistributions are 
smooth (even real-analytic) 
eigenfunctions~\cite[Theorem~6.33]{folland1995introduction}. For that reason we 
will refrain from introducing this additional freedom which will not add any 
additional generality to our results, and we will discuss smooth functions 
only.
We emphasize that we do not assume any growth properties of these 
functions, and hence the \emph{spectral parameter}~$s$ of the Laplace 
eigenfunction~$f$ in~\eqref{eq:Delta_eigen_intro} can \emph{a priori} take any value in~$\C$.

We fix a basis of the complex vector space~$V$ with respect to which $A$ is represented by a matrix $J$ in (complex) \emph{Jordan-like normal form}:
\begin{equation}\label{eq:Jordan_intro}
J = 
\begin{pmatrix}
J_1 
\\
& \ddots 
\\
&  & J_p
\end{pmatrix}\,.
\end{equation}
Here, each block~$J_j$ is a square matrix whose only non-zero entries are on 
the diagonal and the superdiagonal and it is either of the form 
\begin{equation}\label{eq:intro_zero}
\begin{pmatrix} 
0 & 1 
\\
& \ddots & \ddots 
\\
& &  0 & 1
\\
& & &  0
\end{pmatrix}
\end{equation}
or of the form
\begin{equation}\label{eq:intro_lambda}
\begin{pmatrix}
\lambda_j & \lambda_j
\\
& \ddots & \ddots 
\\
& & \lambda_j & \lambda_j 
\\
& & & \lambda_j
\end{pmatrix}
\end{equation}
for some $\lambda_j\in\C$, $\lambda_j\not=0$. 
We remark that such a basis of~$V$ always exists as follows immediately from the existence of the standard Jordan normal form and elementary linear algebra.
By splitting~$V$ into a direct sum of subspaces  
\begin{equation}\label{eq:intro_splitV}
V = V_1 \oplus \ldots \oplus V_p
\end{equation}
corresponding to the generalized eigenspaces of the Jordan-like blocks 
in~\eqref{eq:Jordan_intro}, and also splitting the Laplace eigenfunction~$f$ analogously into a direct sum of functions 
\begin{equation}\label{eq:intro_splitf}
f = f_1 \oplus \ldots \oplus f_p
\end{equation}
with $f_j\colon \h\to V_j$, $j\in\{1,\ldots, p\}$, we may study the Fourier 
expansion of~$f$ separately for each Jordan-like block, or equivalently, 
separately for each~$f_j$. 

We will first observe that for the Jordan-like blocks with vanishing eigenvalue 
(thus, if we consider a block of the form~\eqref{eq:intro_zero}) the 
corresponding components of~$f$ vanish identically.

\begin{prop}[Vanishing eigenvalue]\label{prop:intro_vanish}
For each $j\in\{1,\ldots, p\}$ for which the eigenvalue of~$J_j$ is $0$, the 
component function~$f_j$ is the zero function. 
\end{prop}

For $j\in\{1,\ldots, p\}$ for which the Jordan-like block~$J_j$ is of the 
form~\eqref{eq:intro_lambda}, the main constituents of the Fourier expansion 
of~$f_j$ are extensions and adaptations of those from the untwisted case (i.e., $A$ is the identity), which we briefly survey in Section~\ref{sec:motivation}.
First of all, to be able to allow any nonzero value for the eigenvalue~$\lambda_j$ of the Jordan-like block~$J_j$, we may need to choose a nonstandard branch cut for 
the holomorphic logarithm, say $\omega_j\R_{\geq 0}$, and adapt the modified 
Bessel functions~$I$ and~$K$ accordingly. For the sake of readability of this 
introduction, we invite all readers to ignore this subtlety for the moment and 
refer to Sections~\ref{sec:cut} and~\ref{sec:Bessel} for all details. However, 
we will reflect this issue in the notation of the modified Bessel functions of 
the first and second kind, by writing $I_\nu(\,\cdot\,; \omega_j)$ and 
$K_\nu(\,\cdot\,;\omega_j)$ instead of~$I_\nu$ and~$K_\nu$, respectively. 

After this adaptation, all constituents can be built up from either the modified 
Bessel functions~$I_\nu(\,\cdot\,;\omega_j)$ and~$K_\nu(\,\cdot\,;\omega_j)$, 
or as linear combinations of~$y^a$ and~$y^a\log y$ for suitable values 
of~$a\in\pm s + \tfrac12\Z$, depending on~$s$. More precisely, for any 
$\alpha\in\C\smallsetminus\omega_j\R_{\geq 0}$, $m\in\N_0$ and $y>0$ we set 
\begin{align}
\fu_j(m,\alpha,y,s) & \coloneqq 
\frac{y^{\frac12}}{\iu^m}\partial^m_\alpha\left( I_{s-\frac12}(\alpha 
y;\omega_j)\right) 
\label{eq:intro_fu}
\intertext{and}
\futoo_j(m,\alpha,y,s) & \coloneqq  
\frac{y^{\frac12}}{\iu^m}\partial^m_\alpha\left( K_{s-\frac12}(\alpha 
y;\omega_j)\right)\,.
\label{eq:intro_futoo}
\end{align}
As we will see, these maps are the constituents for almost all Fourier 
coefficients. We emphasize that in~\eqref{eq:intro_fu} 
and~\eqref{eq:intro_futoo}, the differential operator~$\partial_\alpha^m$ is 
not the plain $m$-th derivative of~$\Psi_{s-\frac12}(\,\cdot\,;\omega_j)$ for 
$\Psi\in\{I,K\}$, but acts on the map
\[
\alpha \mapsto \Psi_{s-\frac12}(\alpha y; \omega_j)\,,
\]
as shall be indicated by the additional brackets around 
$\Psi_{s-\frac12}(\alpha y;\omega_j)$. Further, we remark that it is always 
possible to choose the branch cut, and hence the $\omega_j$, uniformly for all 
$j$ so that the definitions of $\fu_j$ and $\futoo_j$ in~\eqref{eq:intro_fu} 
and~\eqref{eq:intro_futoo} only depend on the eigenvalues of~$A$ but not on the 
sizes of the Jordan-like blocks. 

The zeroth Fourier coefficient, however, may correspond to the case that $\alpha=0$ 
and, as in the untwisted case, is then formed in a different way. In the 
untwisted case the value~$\frac12$ for the spectral parameter~$s$ is special 
because of a resonance situation in a certain differential 
equation (see~\eqref{eq:Bessel_intro} below). In our twisted case, analogous 
resonance situations may occur, and a higher size of the Jordan-like block~$J_j$ 
causes a larger set of special values for the spectral parameter, namely
\begin{equation}\label{eq:intro_special}
E_j \coloneqq \left\{ \frac12\pm k \setmid  k=0,\ldots, 
\left\lfloor\frac{d_j-1}{2}\right\rfloor\right\}\,, 
\end{equation}
where $d_j$ is the (complex) dimension of the vector space~$V_j$. For any 
$m\in\N_0$ and $y>0$ we define 
\[
\fu_j(m,0,y,s)
\]
to be a certain precise linear combination of 
\[
y^{s+2\left\lfloor\frac{m}{2}\right\rfloor}\quad\text{and}\quad 
y^{s+2\left\lfloor\frac{m}{2}\right\rfloor}\log y
\]
if 
\[
s\in E_j \cap \left[ \frac12-\frac{m}2, \frac12 \right]\,,
\]
and to be a certain precise scalar multiple of 
\[
y^{s+2\left\lfloor\frac{m}{2}\right\rfloor}
\]
for all other values of~$s$. We refer to Section~\ref{sec:IandKpart2} for the 
precise expressions, and remark here only that the definition of 
$\fu_j(m,0,y,s)$ depends on the size of the Jordan-like block~$J_j$ and hence 
cannot be made uniform for all Jordan-like blocks of~$A$, in contrast to the 
definition in~\eqref{eq:intro_fu}. Analogously we define 
\[
\futoo_j(m,0,y,s)
\]
to be a certain precise linear combination of 
\[
y^{1-s+2\left\lfloor\frac{m}{2}\right\rfloor}\quad\text{and}\quad 
y^{1-s+2\left\lfloor\frac{m}{2}\right\rfloor}\log y
\]
if
\[
s\in E_j\cap \left[ \frac32, \frac12 + \frac{m}{2}\right]\,,
\]
and to be a certain scalar multiple of 
\[
y^{1-s+2\left\lfloor \frac{m}{2} \right\rfloor}
\]
for all other values of~$s$. As for~$\fu_j(\cdot,0,\cdot,\cdot)$, the 
definition of~$\futoo_j(\cdot,0,\cdot,\cdot)$ cannot be made independent of~$j$. 
For each $n\in\Z$ we set 
\begin{equation}\label{eq:defalphan}
\alpha_{j,n} \coloneqq 2\pi n - \iu\log(\lambda_j;\omega_j)\,,
\end{equation}
where $\log(\;\cdot\;;\omega_j)$ is the complex logarithm with branch 
cut~$\omega_j\R_{\geq 0}$ (see Section~\ref{sec:cut} for details). Further we 
pick $\eps_{j,n}\in\{0,1\}$ such that 
\begin{equation}\label{eq:pickepsn}
(-1)^{\eps_{j,n}}\alpha_{j,n} \notin \omega_j\R_{>0}\, ,  
\end{equation}
and set 
\begin{equation}\label{eq:deftildealphan}
\tilde\alpha_{j,n} \coloneqq (-1)^{\eps_{j,n}}\alpha_{j,n}\,.
\end{equation}

\begin{thm}[Single Jordan block; coarse statement]\label{thm:intro_single}
Let $j\in\{1,\ldots,p\}$ be such that the eigenvalue~$\lambda_j$ of $J_j$ is 
not $0$. Then the Fourier expansion of $f_j$ is 
\[
f_j(z) = \sum_{n\in\Z} J_j^x \wh f_{j,n}(y,s) e^{2\pi \iu n x} \qquad (z=x+\iu 
y\in\h)\,,
\]
where the Fourier coefficient function~$\wh f_{j,n}(\cdot, s)$ for $n\in\Z$ is 
of the form 
\begin{align*}
\wh f_{j,n}(y,s) & = \tilde C_n 
\begin{pmatrix} 
\fu_j(d_j-1, \tilde\alpha_{j,n}, y, s)
\\
\vdots
\\
\fu_j(1, \tilde\alpha_{j,n}, y, s)
\\
\fu_j(0, \tilde\alpha_{j,n}, y, s)
\end{pmatrix}
+ \tilde D_n 
\begin{pmatrix}
\futoo_j(d_j-1, \tilde\alpha_{j,n}, y, s)
\\
\vdots 
\\
\futoo_j(1, \tilde\alpha_{j,n}, y, s)
\\
\futoo_j(0, \tilde\alpha_{j,n}, y, s)
\end{pmatrix}
\end{align*}
with $\tilde C_n, \tilde D_n$ being appropriate matrices in~$\C^{d_j\times 
d_j}$.
\end{thm}

We refer to Sections~\ref{sec:coeffmatrices} and~\ref{sec:refined} for precise 
statements and the refined statement of Theorem~\ref{thm:intro_single}, in 
particular to Theorem~\ref{Fourier:main_theorem}. 
Moreover, we will see that the matrices $\tilde C_n$ and $\tilde D_n$ in 
Theorem~\ref{thm:intro_single} belong to a certain precise subspace of 
$\C^{d_j\times d_j}$ of dimension~$2d_j$ if $\lambda_j=1$, $n=0$ (i.e., if 
$\alpha_{j,n}=0$) and $d_j>1$, and of dimension~$d_j$ in all other cases. See 
Section~\ref{sec:coeffmatrices} and Theorem~\ref{Fourier:main_theorem}. 

We emphasize that in the case that the endomorphism~$A$ is the identity, we have $\lambda_j=1$ and hence $\alpha_{j,n}=2\pi n$. We then may pick~$\omega_j=-1$ so 
that $\tilde\alpha_{j,n} = 2\pi|n|$ and hence Theorem~\ref{thm:intro_single} reduces to the classical result stated in~\eqref{eq:untw_coeff_intro}. 
Our desire to be able to recover the classical result from Theorem~\ref{thm:intro_single} is indeed the main motivation to introduce the $\eps_{j,n}$-family in~\eqref{eq:pickepsn}.
See Section~\ref{sec:motivation} for more details.
In particular, in this case, the eigenvalue~$\lambda_j=1$ makes no visible appearance (as the relevant term~$\log(\lambda_j;\omega_j)$ vanishes). However, if $A$ is not the identity and hence the eigenvalue~$\lambda_j$ is not~$1$ (for some~$j\in\{1,\ldots, p\}$), then the modified Bessel functions have to account for a non-periodicity or a twist-periodicity in the~$y$-direction of the function~$f_j$, which results in the dependence of the Bessel functions on~$\lambda_j$ (and hence~$\tilde\alpha_{j,n}$) as stated in~\eqref{eq:intro_fu}--\eqref{eq:intro_futoo}.

The Fourier expansion of~$f$ can now be deduced by combining 
Proposition~\ref{prop:intro_vanish} and Theorem~\ref{thm:intro_single}.

\begin{thm}[Several Jordan blocks]\label{thm:intro_general}
The Fourier expansion of~$f$ is 
\[
f(z) = \sum_{n\in\Z} J^x \hat f_n(y,s) e^{2\pi \iu n x}\qquad (z=x+\iu 
y\in\h)\,,
\]
where the $n$-th Fourier coefficient is
\[
\hat f_n(y,s) = \begin{pmatrix} 
\wh f_{1,n}(y,s)
\\
\vdots 
\\
\wh f_{p,n}(y,s)
\end{pmatrix}
\]
with 
\[
\wh f_{j,n}(y,s) = 
\begin{cases} 
0 & \text{if $\lambda_j=0$}\,,
\\
\text{as in Theorem~\ref{thm:intro_single}} & \text{if $\lambda_j\not=0$}
\end{cases}
\]
for all $j\in\{1,\ldots, p\}$.
\end{thm}

As in the untwisted case, we can distinguish the functions~$\fu_j$ 
and~$\futoo_j$ by their growth properties. This result is certainly of 
independent interest and will also be used in the proof of 
Theorem~\ref{thm:intro_single}. Due to the sometimes necessary non-standard 
choice of a branch cut for the holomorphic logarithm and the related adaptations 
of the modified Bessel functions (depending on the endomorphism~$A$), the 
growth behavior of the functions~$\fu_j$ and~$\futoo_j$ may be qualitatively 
different depending on the different domains for~$\alpha$ 
in~\eqref{eq:intro_fu} and~\eqref{eq:intro_futoo}. In the ``principal sector'' 
(roughly, the sector in which the adapted and the classical modified Bessel 
functions are qualitatively equal) their growth behavior is as in the classical 
untwisted case, however outside of this sector it may change. We refer to 
Section~\ref{sec:growthIK} for the precise and detailed description of the 
situation and state here only a partial result.

\begin{thm}[Growth behavior, partial statement]\label{thm:intro_growth}
Let $s\in\C$, $m\in\N_0$, $j\in\{1,\ldots, p\}$, and suppose that 
$\alpha$ is an element of~$\C\nobreak\smallsetminus\nobreak\omega_j\R_{\geq0}$ 
in the ``principal sector.'' Then:
\begin{enumerate}[label=$\mathrm{(\roman*)}$, ref=$\mathrm{\roman*}$]
\item The absolute value of~$\fu_j(m,\alpha,y,s)$ increases exponentially 
as $y\to\infty$.
\item The absolute value of~$\futoo_j(m,\alpha,y,s)$ decreases exponentially 
as $y\to\infty$.
\end{enumerate}
\end{thm}

In what follows we will first discuss, in Section~\ref{sec:motivation}, the mathematical context and the motivation of our investigations. After that, in Section~\ref{sec:survey} we will provide a survey of the proofs of the main theorems. Of course, the following two sections can also be read in converse order.

\subsection{Motivation and mathematical context}\label{sec:motivation}

\subsubsection{Brief review of the classical, untwisted case}
If the endomorphism~$A$ is the identity on~$V$, then the 
system~\eqref{eq:Delta_eigen_intro}-\eqref{eq:twist_intro} asks for the Laplace 
eigenfunctions~$f\colon\h\to\nobreak V$ with spectral parameter~$s\in\C$ that are 
periodic with period~$1$ in the $x$-direction. Thus, 
\eqref{eq:Delta_eigen_intro}-\eqref{eq:twist_intro} becomes 
\begin{equation}\label{eq:setup_untw_intro}
\Delta f = s(1-s)f\qquad\text{and}\qquad f(z+1) = f(z)\quad\text{for all $z\in\h$\,.}
\end{equation}
In this classical, untwisted case the existence and structure of the Fourier 
expansion of~$f$ is well-known (see, e.g.,~\cite[Sections~1.2 and~8.1]{BLZm}).  
The periodicity of~$f$ causes the Fourier expansion to take the form
\begin{equation}\label{eq:untw_exp_intro}
f(z) = f(x+\iu y) = \sum_{n\in\Z} \hat f_n(y,s)e^{2\pi \iu n x}\qquad 
(z\in\h)\,. 
\end{equation}
As $f$ is a Laplace eigenfunction, the $n$-th Fourier coefficient 
function~$\hat f_n(\cdot,s)$ necessarily satisfies the modified Bessel 
differential equation 
\begin{equation}\label{eq:Bessel_intro}
\left( y^2 \partial_y^2 + s(1-s) - (2\pi n y)^2 \right) \hat f_n(y,s) = 0\,, 
\end{equation}
which depends on~$n$ and $s$. If $\dim V = 1$ (i.e., the scalar case), and 
$n\not=0$, then a fundamental set of solutions of the differential 
equation~\eqref{eq:Bessel_intro} is given by the two functions
\begin{equation}\label{eq:base_scalar_intro}
y\mapsto y^{\frac12} I_{s-\frac12}(2\pi |n| y) \qquad\text{and}\qquad y\mapsto 
y^{\frac12} K_{s-\frac12}(2\pi |n| y)\,,
\end{equation}
which are both defined on the interval~$(0,\infty)$. Here $I_{s-\frac12}$ 
and~$K_{s-\frac12}$ are two specific linearly independent solutions of the 
modified Bessel differential equation, commonly known as the \emph{modified 
Bessel functions of the first and second kind with index~$s-\tfrac12$}, 
respectively. Among their important features are their growth behaviors 
as~$y\to\infty$:  the function $I_{s-\frac12}$ is exponentially increasing, and 
$K_{s-\frac12}$ is exponentially decreasing. These growth properties are passed 
on to the functions in~\eqref{eq:base_scalar_intro}. Theorem~\ref{thm:intro_growth} and its refined variants in Section~\ref{sec:growthIK} state analogous properties in the twisted case. For $\dim V = 1$ and $n=0$, two independent 
solutions of~\eqref{eq:Bessel_intro} are $y^s$ and $y^{1-s}$ if 
$s\not=\tfrac12$, and $y^\frac12$ and $y^\frac12\log y$ if $s=\tfrac12$.
For a vector space~$V$ of arbitrary finite dimension, the Fourier coefficient 
functions in~\eqref{eq:untw_exp_intro} are therefore of the form 
\begin{equation}\label{eq:untw_coeff_intro}
\hat f_n(y,s) =
\begin{cases}
c_ny^{\frac12}I_{s-\frac12}(2\pi |n| y)  + d_ny^{\frac12}K_{s-\frac12}(2\pi 
|n| y)\,, & \text{$n\not=0$,}
\\
c_0y^s + d_0 y^{1-s}\,, & \text{$n=0$, $s\not=\frac12$,}
\\
c_0y^{\frac12} + d_0 y^{\frac12}\log y\,, & \text{$n=0$, $s=\frac12$}
\end{cases}
\end{equation}
for any~$n\in\Z$, where the coefficients $c_n,d_n$ are suitable elements 
of~$V$, depending on~$f$. 

\subsubsection{Compatibility of Theorem~\ref{thm:intro_general} with the classical, untwisted result}

Theorem~\ref{thm:intro_general} and its full version,  Theorem~\ref{Fourier:main_theorem}, are compatible with the 
classical result~\eqref{eq:untw_coeff_intro} in the untwisted setting, as we now indicate.
Since the endomorphism~$A$ is the identity and hence all its eigenvalues equal~$1$, we may choose $\omega=\omega_j=-1$ (see right after Proposition~\ref{prop:intro_vanish}), which retrieves the principal logarithm. Then all adapted modified Bessel functions are identical to the classical modified Bessel functions and $\alpha_n=2\pi n$ for all~$n\in\Z$ (see~\eqref{eq:defalphan}).
For $n>0$ we necessarily pick $\eps_n=0$ in~\eqref{eq:pickepsn}, 
and for $n<0$ we necessarily pick $\eps_n=1$. This results in 
\[
 \tilde\alpha_n=(-1)^{\eps_n}\alpha_n = 2\pi |n|
\]
for all $n\in\Z$. See~\eqref{eq:deftildealphan}.
From this it immediately follows that Theorem~\ref{Fourier:main_theorem} provides the classical result. 

At this point we can explain more precisely our main motivation for working with the $\eps_n$-family instead of using right away the values~$\alpha_n$. (See also the comment right after Theorem~\ref{thm:intro_single}.) In the introduction of Section~\ref{sec:proof_nonzero} we discuss an alternative method for handling negative~$\alpha_n$. If we would implement this method, then we would 
need to pick a value for~$\omega$ that is not in~$\R_{<0}$, but then we 
would be unable to reproduce the classical result.

\subsubsection{Motivation I, untwisted case}
To illustrate the importance of such precise knowledge of the structure of the 
Fourier expansion and the Fourier coefficients in the untwisted case as in~\eqref{eq:untw_coeff_intro}, we now briefly describe one of 
its many consequences: the growth thresholds of automorphic Laplace 
eigenfunctions at cusps. To that end we first note 
that~\eqref{eq:untw_coeff_intro} implies that the function~$f$ 
in~\eqref{eq:setup_untw_intro} and~\eqref{eq:untw_exp_intro} decomposes as
\begin{equation}\label{eq:decomp_f_intro}
f = f_I + f_0 + f_K
\end{equation}
with 
\begin{equation}\label{eq:fI_intro}
f_I(z) \coloneqq \sum_{\substack{n\in\Z\\ n\not=0}} 
c_ny^{\frac12}I_{s-\frac12}(2\pi |n|y)e^{2\pi \iu n x}\,,\qquad z=x+\iu 
y\in\h\,,
\end{equation}
growing exponentially in absolute value as $y\to\infty$ as soon as not all 
coefficients~$c_n$ vanish, 
\begin{equation}\label{eq:fK_intro}
f_K(z) \coloneqq \sum_{\substack{n\in\Z\\ n\not=0}} 
d_ny^{\frac12}K_{s-\frac12}(2\pi |n|y)e^{2\pi \iu n x}\,, \qquad z=x+\iu 
y\in\h\,,
\end{equation}
decaying exponentially as $y\to\infty$ as soon as not all coefficients~$d_n$ 
vanish, and
\begin{equation}\label{eq:f0_intro}
f_0(z) \coloneqq 
\begin{cases}
c_0y^s + d_0y^{1-s} & \text{for $s\not=\frac12$}\,,
\\
c_0y^{\frac12} + d_0y^{\frac12}\log y & \text{for $s=\frac12$}
\end{cases}
\end{equation}
with a growth rate depending on~$s$ but being at most polynomial as 
$y\to\infty$.

We now let $\Gamma$ be a discrete subgroup of~$\SL_2(\R)$ (or~$\PSL_2(\R)$, 
if we want to consider Fuchsian groups); examples include $\Gamma=\SL_2(\Z)$ (or 
$\Gamma=\PSL_2(\Z)$). Then $\Gamma$ acts isometrically on the upper half 
plane~$\h$ by fractional linear transformations, 
\[
\mat{a}{b}{c}{d}.z = \frac{az + b}{cz + d}
\]
for~$\textmat{a}{b}{c}{d}\in\Gamma$, $z\in\h$. We suppose further that the 
hyperbolic surface\footnote{For simplicity, we refer to all of the orbit 
spaces~$\Gamma\backslash\h$ as hyperbolic surfaces, even though some of them 
are genuine orbifolds and not smooth manifolds. In this case, the Riemannian 
metric is to be understood as a Riemannian metric on an 
orbifold.}~$\Gamma\backslash\h$, that is the orbit space of the action 
of~$\Gamma$ on~$\h$ with the canonical Riemannian metric, has at least one 
infinite end of finite area, a so-called \emph{cusp}. In rough terms (and fully 
sufficient for our purposes),  a cusp can be described as follows: we take a 
(very large) compact subset $K$ of~$\Gamma\backslash\h$ such that for all 
compact subsets~$\tilde K$ with $K\subseteq\tilde K$ the 
spaces~$(\Gamma\backslash\h)\smallsetminus K$ 
and~$(\Gamma\backslash\h)\smallsetminus\tilde K$ have the same number of 
connected components. These connected components are the \emph{ends} 
of~$\Gamma\backslash\h$, and those of finite area are the \emph{cusps}. (It is 
irrelevant for our purposes that these notions of end and cusp depend on the 
choice of ~$K$ and should be more correctly called \emph{end area} and 
\emph{cusp area}, respectively.)  For example, the \emph{modular 
surface}~$\SL_2(\Z)\backslash\h$ has one cusp and no further ends. In many 
applications in harmonic analysis, number theory, and physics, one needs to 
understand the growth behavior at cusps of \emph{\mbox{$\Gamma$-}auto\-morphic 
functions}, that is of the functions~$u\colon\h\to V$ satisfying
\begin{equation}\label{eq:auto_intro}
u(g.z) = u(z) \qquad\text{for all $g\in\Gamma$, $z\in\h$}\,,
\end{equation}
that are simultaneously Laplace eigenfunctions, thus
\begin{equation}\label{eq:Delta_u_intro}
\Delta u = s(1-s)u\,.
\end{equation}
In particular, Maass wave forms for~$\Gamma$, real-analytic Eisenstein series, 
and $\Gamma$-auto\-morphic forms of weight~$0$ are examples of such automorphic 
Laplace eigenfunctions. We remark that such functions  
satisfying~\eqref{eq:auto_intro} and~\eqref{eq:Delta_u_intro} are sometimes 
called \mbox{$\Gamma$-}auto\-morphic \emph{forms}. However, since we do not 
require any growth properties or regularity properties at cusps or other ends 
of the hyperbolic surfaces, we will refer to these objects as automorphic 
\emph{functions} despite the standing hypothesis of them being Laplace 
eigenfunctions.

In the case that $\Gamma$ is generated by the element 
$\textmat{1}{1}{0}{1}\in\SL_2(\R)$, the 
system~\eqref{eq:auto_intro}--\eqref{eq:Delta_u_intro} reduces 
to~\eqref{eq:setup_untw_intro}. Also for any other choice of~$\Gamma$, 
the setting~\eqref{eq:setup_untw_intro} is vital to understanding the solutions 
to~\eqref{eq:auto_intro}--\eqref{eq:Delta_u_intro}: locally at each cusp 
of~$\Gamma\backslash\h$, after an appropriate change of coordinates (which we 
will suppress in what follows), the invariance~\eqref{eq:auto_intro} descends 
to $1$-periodicity in the $x$-direction.  Studying the local Fourier expansion 
of $u$ in this cusp is therefore the same as studying the Fourier expansion 
of~$f$ as in~\eqref{eq:setup_untw_intro}. In particular, in this cusp, the 
function~$u$ admits a decomposition as $f$ in~\eqref{eq:decomp_f_intro}. 
Therefore, if we know that $u$ has at most polynomial growth in the cusp, then 
the term corresponding to~$f_I$ in the Fourier expansion of~$u$ 
in~\eqref{eq:decomp_f_intro} vanishes and the expansion becomes
\[
u = u_0 + u_K
\]
with $u_0$ and $u_K$ as in~\eqref{eq:decomp_f_intro} (with $u$ in place 
of~$f$). 
This in turn provides us with the additional knowledge on the polynomial growth 
rate of~$u$. If $u$ is bounded and $s\not\in\{0,1\}$, then in addition the 
term~$u_0$ vanishes, and $u=u_K$ is exponentially decreasing as 
$y\nobreak\to\nobreak\infty$. Results of this type are crucial for several 
important applications, including obtaining dimension bounds for subspaces of 
solutions to~\eqref{eq:auto_intro} and~\eqref{eq:Delta_u_intro} as 
in~\cite{Lewis_Zagier}, as well as characterizing period functions for 
different types of automorphic functions as in~\cite{Lewis_Zagier, BP-A}.

\subsubsection{Motivation II, twisted case}
In the untwisted case, the endomorphism~$A$ is the identity on~$V$, and 
studying 
the Fourier expansions of $\Gamma$-automorphic Laplace eigenfunctions in cusps 
is equivalent to studying the Fourier expansions of solutions 
of~\eqref{eq:setup_untw_intro}. In the \emph{twisted case}, i.e., when the 
endomorphism~$A$ in \eqref{eq:twist_intro} is not the identity on~$V$, the 
following similar relation exists:
we consider functions $u\colon\h\to V$ which satisfy the \emph{twisted} 
invariance 
\begin{equation}\label{eq:auto_twist}
u(g.z) = \chi(g) u(z) \qquad\text{for all $g\in\Gamma$, $z\in\h$}\,,
\end{equation}
where $\chi\colon\Gamma\to\GL(V)$ is a representation of~$\Gamma$ on~$V$. 
Functions satisfying~\eqref{eq:auto_twist} are often called 
\emph{$(\Gamma,\chi)$-automorphic functions} or \emph{vector-valued automorphic 
functions with multiplier~$\chi$}. If we suppose for simplicity (to avoid 
any changes of coordinates) that 
\[
T\coloneqq\mat{1}{1}{0}{1}\in\Gamma,
\]
then \eqref{eq:auto_twist} applied for $g\coloneqq T$ becomes $u(T.z)=\chi(T)u(z)$, or 
equivalently,
\begin{equation}
u(z+1) = Au(z) \qquad\text{for all $z\in\h$}
\end{equation}
with $A\coloneqq \chi(T)$, which is precisely~\eqref{eq:twist_intro} for~$f=u$. 
An even stronger relation between~\eqref{eq:twist_intro} 
and~\eqref{eq:auto_twist} is valid if we suppose in addition that 
$\textmat{1}{x}{0}{1}\notin\Gamma$ for $x\in (0,1)$ or, in other words, if $T$ is a generator of the stabilizer group of~$\infty$ in~$\Gamma$.
Then $\Gamma\backslash\h$ has a cusp of the form 
\[
\langle T\rangle\backslash \{ x+\iu y \in\h \setmid y > y_0\}
\]
for some $y_0>0$. Restricted to this cusp, the only twist-invariance of~$u$ 
that is preserved from~\eqref{eq:auto_twist} is the one by~$T$ (and its 
iterates). 
Thus, locally in cusps, twisted automorphic functions are exactly as 
in~\eqref{eq:twist_intro}, and hence their local properties can be studied via 
investigating functions satisfying~\eqref{eq:twist_intro}.

Those twisted automorphic functions that satisfy in addition the eigenfunction 
property~\eqref{eq:Delta_u_intro}, thus, the functions $u\colon\h\to V$ that 
satisfy
\begin{equation}\label{eq:intro_autom_twist}
\Delta u = s(1-s)u \qquad\text{and}\qquad u(g.z) = \chi(g)u(z)\quad 
(g\in\Gamma,\ z\in\h)\,,
\end{equation}
are the \emph{$(\Gamma,\chi)$-automorphic Laplace eigenfunctions}. Prominent 
subclasses include various kinds of \emph{$V$-valued automorphic forms of 
weight~$0$ with multiplier~$\chi$}, which are distinguished by regularity 
conditions or growth properties at the ends of $\Gamma\backslash\h$.

Twisted automorphic Laplace eigenfunctions have received considerable attention 
for several decades already as they arise naturally in mathematics and physics. 
For example, they help us to understand the fine structure of the 
\emph{untwisted} automorphic Laplace eigenfunctions~\cite{Selberg_est_fourier}, 
and they play a crucial role in some areas of modern physics. 

More precisely, several physical systems can only be adequately addressed with 
vector-valued functions rather than scalar functions, due to the necessity to 
describe several characteristics simultaneously, like a field of arbitrary spin. 
 For example, vector-valued eigenfunctions of Laplace operators appear in the 
study of Casimir interaction between two conducting obstacles~\cite{MR3584189}.  
Vector-valued operators of Laplace-type help to investigate the behavior of a 
charged particle under the influence of a  magnetic field~\cite{MR3692925}.
They also play an  important role in attacking tunneling problems~\cite{MR4127499}.

In mathematical physics, both the thermal $\AdS_3$ space-times and the BTZ 
black hole are modeled by a geometric setting closely related to that 
considered in the present article.  Specifically, they are modeled by~$\Gamma 
\backslash \mathbb{H}^3$ with $\Gamma \subseteq \SL_2(\C)$ being a discrete 
infinitely cyclic subgroup~\cite[p.~4]{heatkernelonAdS3}.  For 
the BTZ black hole, the group $\Gamma$ is explicitly given by  
\[
\Gamma = \left\langle \left( \begin{smallmatrix}
e^{a + i b} & 0 \\ 0 & e^{-a-ib}
\end{smallmatrix} \right) \right\rangle 
\]
with~$a,b \in \R$ depending on the black hole's mass and angular momentum. See~\cite[(6-7)]{bytsenko2006spectral}.  

In the geometric setting of~$\mathbb{H}^3$, fields of arbitrary spin-$s$ 
correspond to sections of homogeneous vector bundles \cite[p.~2]{heatkernelonAdS3}. 
If the field should be defined not only on~$\mathbb{H}^3$, but also on~$\Gamma \backslash \mathbb{H}^3$, as used to model the aforementioned space times and black holes, then an additional $\Gamma$-invariance is required.
Consequently, a homogeneous vector bundle with such a $\Gamma$-invariance becomes a locally homogeneous vector bundle.  

To be more precise, let $G$ be a semisimple Lie group and $K$ its maximal 
compact subgroup.
For BTZ black holes, $G=\SL_2(\C)$ and $K=\SU(2)$.
Let $\rho\colon G \mapsto \GL(V)$ be a finite-dimensional representation of~$G$, and let $\rho_K$ and $\rho_\Gamma$ be the restrictions of~$\rho$ to~$K$ and $\Gamma$, respectively.
We let $E_{\rho_\Gamma}$ be a flat bundle associated to~$\rho_{\Gamma}$ and we let $E_{\rho_K}$ be the locally homogeneous vector bundle associated to~$\rho_K$.
In order to define the latter, we first define a homogeneous vector bundle 
\[
\widetilde{E}_{\rho_K} \coloneqq (G \times V) / K \to G/K\,,
\]
where $K$ acts on the right as 
\[
(g,v)k \coloneqq (gk, \rho_K(k^{-1}) v)\,, \qquad \text{for $g \in G$, $k \in K$, $v \in V$}
\]
and let $E_{\rho_K} \coloneqq \Gamma \backslash \widetilde{E}_{\rho_K}$. In~\cite[Proposition 3.1]{MatsushimaMurakami} it is stated that 
\[
E_{\rho_K} \cong E_{\rho_\Gamma}\,.
\]  
If we consider a spin-$s$ representation of~$\SU(2)$ for an integer~$s$, then it has a lift to a representation of~$\SL_2(\mathbb{C})$.
In this situation, we can construct a flat vector bundle associated to the restriction of such a representation to~$\Gamma$, which will be in one-to-one correspondence with the original locally homogeneous vector bundle.
Thus, fields of spin-$s$ for integer~$s$ are in this way isomorphic to sections of a flat vector bundle, analogous to the vector-valued functions in our present study.  
Moreover, non-unitary twists are essential:
at least for~$\SL_2(\R)$, restrictions of representations of~$\SL_2(\R)$ to a subgroup of~$\SL_2(\R)$, are in general \emph{not} unitary. 
One may compare the explicit expression for such representations in the proof 
of~\cite[Proposition~5.1]{FP_szf}. In the $3$-dimensional case, these correspond 
to~$\rho_\Gamma$ defined above which are in general non-unitary as well.   
Thus, spin-$s$ fields are isomorphic to sections of certain flat vector bundles 
corresponding to non-unitary representations. If these sections are Laplace 
eigenfunctions, then they are expected to correspond to zeros of the Selberg zeta 
function.  In fact, for finite volume surfaces and unitary representations this 
is a direct consequence of the classical Selberg trace formula.  These zeros 
encode the normal frequencies of spin-$s$ fields on thermal $\AdS_{3}$ 
space-times. Consequently, detailed information on vector-valued Laplace 
eigenfunctions twisted by not-necessarily unitary representations should lead to 
a deeper understanding of Vasiliev theories~\cite{Vasiliev}. Of course these applications also show that it would be of great interest to study Fourier expansions of twisted automorphic Laplace eigenfunctions for much more general spaces. However, such a generalization of our results is beyond the scope of this article. 

In the applications we just discussed, the twist is provided by a genuine representation and hence the twisting endomorphism is \emph{invertible}. 
However, with our investigations in this article, we can easily cover \emph{non-invertible} endomorphisms as well. In view of understanding the twist as a perturbation, we are certain that also non-invertible twists will have 
far-reaching important applications and the transition from non-invertible twists to invertible ones will lead to interesting insights. 

Due to the great significance of twisted automorphic Laplace eigenfunctions, much effort has already been spent on their study and several important 
results have been obtained, for various types of such functions and in various 
generalities in regard to the twist. A non-exhaustive list regarding 
\emph{non-unitary twists} includes~\cite{Eholzer_Skoruppa, Daughton2016, 
Raji_cohom, Muehlenbruch_Raji, Deitmar_Monheim_traceformula, 
Deitmar_Monheim_eisenstein, Deitmar_locally_compact, FP_szf, FP_eisenstein, 
Saber_Sebbar} and \cite{Knopp_Mason_definition, Knopp20032, 
Knopp_Mason_illinois, Knopp2011, Knopp_Mason2012, Kohnen_Mason, Kohnen_Martin, 
Kohnen, Kohnen_Mason2, Bantay_Gannon, Creutzig_Gannon, Gannon_theory, DGMS}. 
Analogous questions can be asked for higher-dimensional spaces as well; the 
first results in this direction for non-unitary twists were recently established 
in~\cite{Mueller_STF, Spilioti2015, Fedosova_nonunitary, 
Deitmar_Monheim_traceformula, Deitmar_locally_compact, Spilioti2020}.

Further examples of functions closely related to this area of research are
vector-valued modular forms. Typically they are not Laplace 
eigenfunctions, but they obey the same type of controlled non-periodicities as 
supposed in the present article. Vector-valued modular forms play a crucial role 
in the generalized moonshine conjecture \cite{Dong_Li_Mason, 
gannon2006monstrous, MR3859972} being ultimately related to generating series 
for characters of rational vertex operator algebras 
\cite{Zhu_characters_of_vertex_operator_algebras, franc2021character}.  
Vector-valued modular forms are also important objects in 2-dimensional 
conformal field theory~\cite{MR3871267, francesco2012conformal}.  Deformations 
of vector-valued modular forms arise in the study of gravitational waves 
\cite{massive_deformations_of_maass_forms}. Additionally, vector-valued modular 
forms are used in the Standard Model of particle physics to construct example 
models for lepton masses \cite{liu2021modular}. Vector-valued mock modular forms 
appear in string theory while studying {D}3-instantons \cite{MR3638318}.
Interestingly, some components of vector-valued automorphic functions satisfy  
differential equations similar to those used to model the behavior of the 4-loop 
supergraviton \cite{Differential_equations_in_automorphic_forms, 
Eisenstein_series_for_higher-rank_groups_and_string_theory_amplitudes}. 
Considering the importance of vector-valued modular forms, it seems reasonable 
to expect that also a kind of non-holomorphic vector-valued 
``modular forms'' that are Laplace eigenfunctions will have similar physical 
applications, for which the results of our present article are applicable.

\subsubsection{Previous results on Fourier expansions in the presence of twists, and their relation to Theorem~\ref{thm:intro_general}}
Results regarding Fourier expansions of twisted automorphic Laplace 
eigenfunctions are---despite their importance---rather sparse. To date Fourier 
expansions have been established mostly for twists by unitary 
representations; for example, in~\cite{Phillips_perturb_twist, PhillipsSarnak94}, they appear in the study of the behavior of counting functions of cusp forms. In particular, the authors employ the space of $L^2$-functions with zero Fourier coefficients in the cusp at~$\infty$ vanishing close enough to the mentioned cusp. Other references include~\cite{Hejhal2} and~\cite{Venkov_book}.

In order to show that our results are compatible with the mentioned results, we compare the structure of the Eisenstein series in \cite{Venkov_book} with the Fourier expansion we have obtained. We denote by $E_\chi(z,s) \coloneqq E_{\Gamma, \chi}(z,s)$ the Eisenstein series twisted by a unitary representation~$\chi$; without loss of generality we let $\dim_\C \chi = 1$. If we assume that the eigenvalue of $\chi$ is equal to one as in \cite[Theorem 3.1.2]{Venkov_book}, then the problem reduces to the classical case that has been considered above. In the case when the eigenvalue,~$e^{2 \pi i \theta}$, is not equal to one, \cite[Theorem 3.1.3]{Venkov_book} reads 
\[
E_\chi(z,s) = \sum_{j\in \mathbb{Z}} a_{n, \theta, s} e^{2 \pi i (j + \theta) x} K_{s-1/2}(2 \pi |j + \theta| y)
\]
for some $a_{n, \theta, s}$ depending only on $n$, $s$, $\Gamma$ and $\theta$. The result in \cite{Deitmar_Monheim_eisenstein} is structurally similar; however, the coefficients, $a_{n, \theta, s}$, are different.
 
We note that the series above are, strictly speaking, not Fourier series, because $E_\chi(z,s)$ is not periodic in the $x$-direction. However, $e^{-2 \pi i  \theta x} E_\chi(z,s) $ is periodic. We will see later in Section \ref{sec:periodization} that a similar procedure of \textit{periodization}, albeit more technically involved, plays a similar role for the Fourier expansions for non-unitary representations.

For non-unitary twists, Fourier expansions are known only for Eisenstein series, 
however with increasing complexity of twists, namely twists that are unitary in 
cusps~\cite{Deitmar_Monheim_eisenstein} and twists with non-expanding cusp 
monodromy~\cite{FP_eisenstein}. For the sake of completeness we remark that 
related results exist for (weighted) modular forms and cusp forms that are 
twisted with one-dimensional representations that are unitary in cusps, however 
asking for Fourier expansions in~$z$ not in~$x$, which are structurally easier. 
See, e.g., \cite{Selberg_est_fourier, Rademacher_Zuckerman, Raji_fourier}.

\subsection{Survey of proofs and organization of this article}\label{sec:survey}

This article is rather technical and notation-heavy due to the very nature of the 
main results. For the convenience of the reader we provide now an informal 
survey of the proofs. Simultaneously, we present the organization of this 
article.

We collect the necessary background material and introduce the major part of 
the 
notation in Section~\ref{sec:prelimsnotation}. We start the investigations by 
reducing their complexity in the manner to which we have alluded above. Namely, 
we pick a basis of the vector space~$V$ with respect to which the 
endomorphism~$A$ is represented by a matrix in Jordan-like normal form. 
See~\eqref{eq:Jordan_intro}--\eqref{eq:intro_lambda}. The associated 
decompositions of~$V$ and~$f$ as in~\eqref{eq:intro_splitV} 
and~\eqref{eq:intro_splitf} are preserved under the action of~$A$ due to its 
block structure. Thus they are also compatible with Fourier expansions, by 
linearity. We may therefore \emph{reduce the investigations to the case that 
$A$ is represented by a single Jordan-like block}, and we can easily deduce 
Theorem~\ref{thm:intro_general} from Proposition~\ref{prop:intro_vanish} and 
Theorem~\ref{thm:intro_single}. This is discussed in further detail in 
Section~\ref{sec:reduction}.

We then suppose that $A$ is represented by a single Jordan-like block and 
\emph{split the investigations into the two cases whether $A$ has 
eigenvalue~$0$ or not}. In the first case, which is rather straightforward, $A$ 
is represented by a Jordan-like block of the form~\eqref{eq:intro_zero} and hence 
nilpotent. An iteration of $f(z+1)=Af(z)$ yields 
\[
f(z+d) = A^d f(z) = 0
\]
for $d=\dim V$. From this, Proposition~\ref{prop:intro_vanish} follows 
immediately. We dispatch this case in Section~\ref{sec:reduction}.

The other case, in which the eigenvalue of~$A$ is not~$0$, constitutes the main 
bulk of this article. To develop a Fourier expansion in this twisted case, we 
proceed as analogously to the classical approach for the untwisted case as 
possible. However, the presence of the twist not only causes some major 
differences to which we need to adapt but also makes these investigations much 
more involved than in the classical case. 

The first main difference from the classical, untwisted case is that the 
map~$f$ is not periodic in the $x$-direction. For periodic maps we can separate 
the $x$- and the \mbox{$y$-}component of the argument~$z=x+\iu y$ of~$f$ by 
representing~$f$ with respect to the Fourier basis~$\{e^{2\pi \iu 
n x}\}_{n\in\Z}$. The representation coefficients in this basis depend then 
on~$y$ only. In our situation these coefficients are necessarily dependent on 
both~$x$ and~$y$. To overcome this issue we use the following 
\emph{periodization of~$f$}, an approach which is also known from the study of monodromy of ordinary differential equations, vector bundles and modular forms. We represent the endomorphism~$A$ by the 
Jordan-like block~$J$ with eigenvalue~$\lambda\not=0$ of the 
form~\eqref{eq:intro_lambda}. 
The map 
\[
F\colon \h \to V\,,\quad z=x+\iu y\ \mapsto\ J^{-x}f(z)\,,
\]
is $1$-periodic in the $x$-direction. By abstract Fourier theory we therefore 
know that the map~$F$ has a Fourier expansion of the form 
\[
F(z) = \sum_{n\in\Z} \hat f_n(y) e^{2\pi \iu n x}\,,
\]
where the coefficients $\hat f_n(y)\in V$ depend on~$y$ only. Therefore, the 
original map~$f$ admits the expansion
\[
f(z) = \sum_{n\in\Z} J^x\hat f_n(y) e^{2\pi \iu n x}\,.
\]
We observe that although the coefficients in this representation depend on 
both~$x$ and~$y$, the contributions of~$x$ and~$y$ are clearly separated. The 
details of this step are contained in Section~\ref{sec:periodization}.

We then turn to investigating the coefficient functions $\hat f_n\colon 
\R_{>0}\to V$. Exploring the property that $f$ is a Laplace eigenfunction, we 
find that the coefficient functions satisfy the \emph{($n$-dependent) 
differential equation}
\begin{equation}\label{eq:ODE_twist_intro}
\left( y^2\partial^2_y + s(1-s) + y^2\partial^2_x\vert_{x=0}\left(e^{2\pi \iu 
n x}J^x\right) \right)\hat f_n(y) = 0\,,\qquad y\in\R_{>0}\,,
\end{equation}
which is reminiscent of the modified Bessel differential 
equation~\eqref{eq:Bessel_intro}. At this point we encounter the second main 
difference to the untwisted case. If the size of the Jordan-like block~$J$ is 
larger than~$1$, then the differential equation~\eqref{eq:ODE_twist_intro} is 
vector-valued or, in other words, a system of differential equations. Since $J$ 
is not diagonalizable, this system does not decompose into several independent 
differential equations; instead it is heavily interrelated. The deduction 
of~\eqref{eq:ODE_twist_intro} is presented in Section~\ref{s:system}.

After a base change the system~\eqref{eq:ODE_twist_intro} becomes a 
\emph{cascade of differential equations} that we can solve iteratively. As in 
the classical case, we may encounter resonance situations for the zeroth 
coefficient function $\hat f_0$, depending on the value of the spectral 
parameter~$s$ and the eigenvalue~$\lambda$ of~$A$. In Section~\ref{s:system} we 
present this base change, and in 
Sections~\ref{sec:proof_zero}--\ref{sec:proof_nonzero} we provide a detailed 
discussion of the \emph{solution} of~\eqref{eq:ODE_twist_intro}.

The asymptotic behavior of the functions~$\fu$ and $\futoo$ as stated in 
Theorem~\ref{thm:intro_growth} is essentially a consequence of the growth 
properties of the modified Bessel functions. Section~\ref{sec:growthIK} is 
devoted to the proof of the extended version of Theorem~\ref{thm:intro_growth}.

\subsection*{Acknowledgements} 
AP's research is funded by the Deutsche Forschungsgemeinschaft (DFG, German 
Research Foundation) -- project no.~264148330 and no.~441868048 (Priority 
Program~2026 ``Geometry at Infinity''). In addition, she wishes to thank the 
Hausdorff Institute for Mathematics in Bonn for excellent working conditions 
during the HIM trimester program ``Dynamics: Topology and Numbers,'' where part 
of this manuscript was prepared. JR is supported by Swedish Research Council 
Grant 2018-03873. The National Science Foundation award DMS-1440140 supported 
JR at the Mathematical Sciences Research Institute in Berkeley, California 
during the fall semester of 2019. KF and AP wish to thank the Institut 
Mittag-Leffler in Djursholm, Sweden, supported by the Swedish Research Council 
under grant no. 2016-06596, for hospitality during the program ``Thermodynamic 
Formalism---Applications to Geometry, Number Theory, and Stochastics,'' where 
some part of the research for this manuscript was done. The authors wish to thank the referees for valuable comments that helped to improve the presentation of this article.

\section{Essential objects and the refined statement of 
Theorem~\ref{thm:intro_single}}\label{sec:prelimsnotation}

In this section we introduce the majority of the objects used throughout this 
article. In particular, we provide the precise definitions of all constituents 
of the Fourier coefficient functions. In addition, with all definitions in 
place, we will present the refined statement of Theorem~\ref{thm:intro_single}, 
namely as Theorem~\ref{Fourier:main_theorem} below.

\subsection{General notation}\label{s:notations} 
We denote by~$\N$, $\N_0$, $\Z$, $\R$ and~$\C$ the set of the positive natural 
numbers, the non-negative integers, all integers, the real numbers and the 
complex numbers, respectively. We use 
\[
\R_{\geq x_0} \coloneqq \{ x\in\R \mid x\geq x_0\}
\]
with $x_0\in\R$ to denote the closed right half line, starting at~$x_0$. We  
assign analogous meanings to~$\R_{>x_0}$, $\R_{\leq x_0}$ and $\R_{<x_0}$. We 
write $\iu=\sqrt{-1}$ for the imaginary unit in~$\C$, in contrast to~$i$, which 
we will  use as an index. For a complex number, $z\in\C$, we denote its real 
and imaginary part by~$\Rea z$ and~$\Ima z$, respectively. We will often 
write~$x$ for~$\Rea z$ and~$y$ for~$\Rea z$, with implicit understanding that 
$x$ and~$y$ depend on~$z$. Further, for any complex number~$z\in \C$ we 
write~$\bar{z}$ for its complex conjugate. We denote the set~$\C\setminus\{0\}$ 
of invertible complex numbers by~$\C^\times$. For any $n\in\N$, we use 
$\C^{n\times n}$ and $\Mat(n\times n;\C)$ interchangeably for the set of 
$(n\times n)$-matrices with complex entries. Further, for real numbers~$r\in\R$ 
we let $\sgn(r)$ denote the sign of~$r$. We will use the floor function, 
$\lfloor \cdot \rfloor\colon\R \to \Z$, 
\[
 \left\lfloor x \right\rfloor \coloneqq \max\{ m\in\Z \setmid m\leq x\}\,.
\]
Further, from now on, $\Gamma$ always denotes the Gamma functions (contrary to the previous sections, where $\Gamma$ was used for a discrete subgroup of~$\SL_2(\C)$ or~$\SL_2(\R)$ or~$\PSL_2(\R)$, which will not make any further appearances).

\subsection{Setting}\label{sec:setting}
We use the upper half plane model for the hyperbolic plane, with the standard 
hyperbolic Riemannian metric:
\[
\h = \{ z=x+\iu y \in \C \setmid \Ima z = y > 0\}\,, \quad ds^2_z = \frac{dx^2 
+ 
dy^2}{y^2}\,.
\] 
In this model, the Laplace operator is
\[
 \Delta = -y^2 \left( \pa_x ^2 + \pa_y ^2 \right)\,.
\]
We fix a finite-dimensional complex vector space, $V$, and consider $\Delta$ as 
an operator on the smooth functions on~$\h$ with values in~$V$. We further fix 
an endomorphism~$A$ of~$V$. This endomorphism may well be non-unitary or 
non-invertible. We will develop Fourier expansions for the smooth 
functions~$f\colon\h\to V$ that are twist-periodic with respect to~$A$,
\begin{equation}\label{aut_f_c_1}
 f(z+1) = Af(z) \qquad\text{for all $z\in\h$}\,,
\end{equation}
and eigenfunctions of the Laplacian  
\begin{equation}\label{aut_f_c_2}
 \Delta f = s(1-s)f
\end{equation}
for some~$s\in\C$. We remark that each eigenvalue of~$\Delta$ can be 
represented 
in the form $s(1-s)$ with some~$s\in\C$. Unless $s(1-s)=1/4$, in which case 
$s=1/2$, there are always two choices of values for~$s$. We will see that it is 
more convenient to associate to the function~$f$ a \emph{spectral 
parameter}~$s$ instead of its eigenvalue~$s(1-s)$.

We further remark that this setting is more general than it may seem at first 
glance. To elaborate on this, we recall that any horocycle on~$\h$ is 
isomorphic to a maximal \emph{unipotent} subgroup of orientation-preserving 
Riemannian isometries of~$\h$. Vice versa, any such subgroup generates a family 
of concentric horocycles \cite{Eberlein, ImHof}. For this reason we call such a 
subgroup a \emph{horocycle direction}. Each of these subgroups is conjugate 
within the orientation-preserving Riemannian isometries to the one-parameter 
group 
\[
 \{ \h\to\h,\ z\mapsto z+r\}_{r\in\R}\,,
\]
and the associated family of horocycles in~$\h$ consists of the ``horizontal'' 
lines 
\[
 \{\iu y+\R\}_{y>0}\,.
\]
If a map is (twist-)periodic on such a horocycle, then we may always assume 
that 
the (twist-)period is~$1$ after suitable rescaling (which is indeed another 
conjugation within the orientation-preserving Riemannian isometries). 
Therefore, any map that is twist-periodic in any horocycle direction is 
conjugate to a map satisfying~\eqref{aut_f_c_1}.

In the remainder of this section we will introduce several objects most of 
which 
have to be defined separately for each Jordan block of~$A$ and which depend 
intimately on the size of a Jordan block and its eigenvalue. 
\begin{center}
\framebox{
\begin{minipage}{.55\textwidth}
For that reason we suppose in what follows that $A$ acts irreducibly on~$V$ and 
is invertible. 
\end{minipage}
}
\end{center}
In other words, we suppose that $A$ has a single Jordan block, and that its 
eigenvalue,~$\lambda$, is not equal to~$0$. Further we let $d\coloneqq \dim V$ 
denote the dimension of~$V$ (as a complex vector space). All definitions that 
we will present in what follows easily generalize to the generic setting by 
applying them separately to each Jordan block. We will use this approach in 
Section~\ref{sec:reduction}.

Several of the objects that we will define further below will depend on the 
value of the eigenvalue~$\lambda$ of~$A$. However, to simplify notation, we 
will usually suppress this dependence. 

\subsection{Complex logarithm and branch cuts}\label{sec:cut}

We denote the real logarithm by~$\log$, thus
\[
 \log\colon\R_{>0} \to \R\,.
\]
For any complex number~$z\in\C^\times$  we choose the standard 
branch for its argument~$\arg z$, thus $\arg z  \in (-\pi, \pi]$. Further, we 
denote the standard choice of the complex logarithm, the \emph{principal 
logarithm}, by~$\Log$. Thus, 
\begin{equation}\label{eq:Log}
 \Log z = \log|z| + \iu \arg z
\end{equation}
for all~$z\in\C^\times$. As is well-known, the domain of holomorphy 
of~$\Log$ is~$\C\setminus\R_{\leq 0}$.  

For several constructions below we will require a complex logarithm that is 
holomorphic at several points related to the eigenvalue~$\lambda$ of~$A$. 
Because we allow here any (nonzero) complex number as the value of~$\lambda$, we 
will also choose, if necessary or convenient, a non-standard complex logarithm. 
However, we will restrict the admissible choices to those complex logarithms 
that extend the real logarithm~$\log$. 

We pick~$\omega\in\C$, $\omega\notin [0,\infty)$, such that for each $n\in\Z$, 
at least one of 
\begin{equation}\label{eq:pm_omega}
 -\iu\Log\lambda + 2\pi n \quad\text{and}\quad \iu\Log\lambda - 2\pi n
\end{equation}
is \emph{not} in~$\omega\R_{>0}$. We note that $\pm(\iu\Log\lambda -2\pi n) = 
0$ if and only if $\lambda=1$ and $n=0$. In all other cases, at least one of 
the values in~\eqref{eq:pm_omega} is \emph{not} in~$\omega\R_{\geq0}$. Clearly, 
the possible values for~$\omega$ depend on~$\lambda$ (which we do not reflect in 
the notation in favor of simplification), but many choices are possible. The 
definitions and constructions that we will present in what follows intimately 
depend on the choice of~$\omega$. However, a different choice will result only 
in phase changes that do not qualitatively affect the results of 
Theorems~\ref{thm:intro_single}, \ref{thm:intro_general} 
and~\ref{thm:intro_growth} and their refinements and extensions further below.  

We use the set~$\omega\R_{\geq0}$ as a branch cut for a complex logarithm. In 
other words, we will use throughout the realization of the complex logarithm 
that is holomorphic on~$\C\setminus\omega\R_{\geq0}$ and coincides with the 
real logarithm on~$\R_{>0}$. We denote this logarithm 
by~$\log(\;\cdot\;;\omega)$. 
To be more specific, we first note that choosing $\omega=-1$ leads to the 
standard choice of the complex logarithm, thus 
\[
\log(\;\cdot\;;-1) = \Log \colon \C^\times\to\C\,\,,
\]
with $\Log$ being defined in~\eqref{eq:Log}. For generic 
$\omega\in\C\setminus[0,\infty)$ we set 
\begin{equation}\label{Omega_definition}
\Omega_\omega \coloneqq 
\begin{cases}
(-\pi , \arg\omega) \quad &\text{for $\arg\omega > 0$}\,, 
\\
(\arg\omega, \pi] \quad &\text{for $\arg\omega < 0$}\,, 
\end{cases}
\end{equation}
and define the logarithm map
\[
 \log(\;\cdot\;;\omega)\colon\C^\times\to \C
\]
by
\begin{equation}
\log(z; \omega) \coloneqq  
\begin{cases}
\log(z; -1) & \text{for $\arg z \in \Omega_\omega$}\,,
\\
\log(z; -1) - 2 \pi \iu \sgn( \arg\omega) \quad & \text{otherwise}
\end{cases}
\end{equation}
for any~$z\in\C^\times$. The set~$\Omega_\omega$ is chosen such that 
$\log(\;\cdot\;;\omega)$ coincides with $\log$ on~$\R_{>0}$. The logarithm 
map~$\log(\;\cdot\;;\omega)$ is holomorphic on~$\C\setminus\omega\R_{\geq0}$.
The relation between~$\log(\;\cdot\;;\omega)$ and $\Log = \log(\;\cdot\;;-1)$
is shown in Figure~\ref{fig:log_cut_pm}.

\begin{figure}
\centering
\includegraphics[width=300pt]{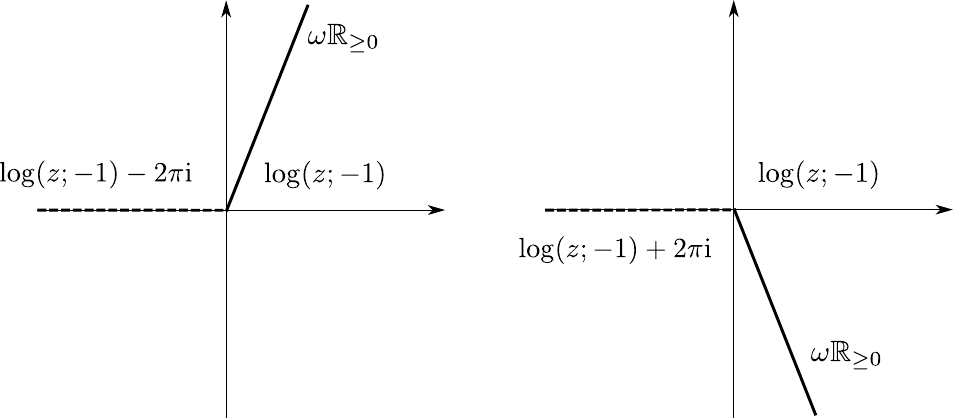}
\caption{The left figure indicates the definition of~$\log(\;\cdot\;; \omega)$ 
for $\arg\omega > 0$ in relation to~$\log(\;\cdot\;;-1)$. The right figure 
shows the relation between~$\log(\;\cdot\;; \omega)$ and~$\log(\;\cdot\;;-1)$ 
if~$\arg\omega < 0$.}
\label{fig:log_cut_pm}
\end{figure}

\subsection{Jordan-like blocks}\label{sec:jordanblocks}

For any $\mu \in \C$, $\mu\not=0$, we let $J(\mu)$ be the $(d\times d)$-matrix 
all of whose entries on the diagonal and superdiagonal equal~$\mu$ and all of 
whose other entries are~$0$:
\begin{equation}\label{J_def}
J(\mu) = 
\begin{pmatrix}
\mu & \mu
\\
& \mu & \mu
\\
& & \ddots & \ddots 
\\
& & &  \mu & \mu
\\
& & & & \mu
\end{pmatrix}\,.
\end{equation}
We call $J(\mu)$ the \emph{Jordan-like matrix} with \emph{eigenvalue}~$\mu$. 
Using the \emph{Kronecker delta function}
\begin{equation}\label{eq:kronecker}
 \delta_0\colon \C\to \{0,1\}\,,\quad \delta_0(z)\coloneqq \begin{cases} 1 & 
\text{if $z=0$} \\ 0 & \text{otherwise} \end{cases}\,,
\end{equation}
we may describe this Jordan-like matrix by
\[
 J(\mu) =  \bigl( j_{mn} \bigr)_{m,n=1}^d = \bigl(\mu \, \delta_0(m-n) + \mu \, 
\delta_0(m+1-n)\bigr)_{m,n=1}^d\,.
\]
As is well-known, the map 
\begin{equation}\label{eq:homJintegers}
 \Z\to\Mat(d\times d;\C)\,,\quad k\mapsto J(\mu)^k\,,
\end{equation}
is a group homomorphism. To understand the effect of the twisting 
endomorphism~$A$ on the components of the Fourier coefficients in the real 
direction (i.e., the~\mbox{$x$-}di\-rec\-tion) we will need an extension 
of~\eqref{eq:homJintegers} to a continuous homomorphism from~$\R$ 
to~$\Mat(d\times d;\C)$, applied to~$\mu$ being the eigenvalue~$\lambda$ 
of~$A$. See Theorems~\ref{thm:intro_single} and~\ref{Fourier:main_theorem}. The existence of such a continuous extension of~\eqref{eq:homJintegers} and its degree of (non-)uniqueness is well-known.
However, for the convenience of the reader we now provide details. This also allows us---as a by-product---to explicitly track down the influence of the choice of the logarithm in this extension, for which reason we are here sticking to a Jordan-like matrix representation of the considered endomorphism. 

The extension we will use throughout is provided in the following lemma, for which we recall that $\lambda$ denotes the eigenvalue of~$A$, and that $\log(\;\cdot\;;\omega)$ is the complex logarithm chosen in Section~\ref{sec:cut}.
For any $x\in\R$, we set 
\begin{equation}\label{eq:def_lambdax}
 \lambda^x \coloneqq e^{x\log(\lambda;\omega)}\,.
\end{equation}
We use the standard notation~$\binom{x}{p}$ for the generalized binomial coefficients of~$x\in\R$ and~$p\in\N_0$.

\begin{lemma}\label{J_power_lemma} Let $J\coloneqq J(\lambda)$. Then the map 
\begin{equation}\label{eq:HomJreals}
 \R\to\Mat(d\times d;\C)\,,\quad x\mapsto J^x\,,
\end{equation}
with 
\begin{equation}\label{Jx1}
J^x \coloneqq \lambda^x  
\begin{pmatrix}
\binom{x}{0} & \binom{x}{1} & \binom{x}{2} & \ldots
& \binom{x}{d-1}
\\[1mm]
& \binom{x}{0} & \binom{x}{1} & \ldots & \binom{x}{d-2}
\\
& & \ddots & \ddots &   \vdots
\\[-2mm]
& & & \ddots & \binom{x}{1}
\\[1mm]
& & & & \binom{x}{0}
\end{pmatrix}\,,
\end{equation}
is a continuous homomorphism that extends the homomorphism 
in~\eqref{eq:homJintegers} for $\mu=\lambda$.
\end{lemma}

\begin{remark}
We emphasize two facts about Lemma~\ref{J_power_lemma}.
\begin{enumerate}[label=$\mathrm{(\roman*)}$, ref=$\mathrm{\roman*}$]
\item Formula~\eqref{Jx1} shows that if the considered endomorphism is represented as a Jordan-like matrix, then the influence of the choice of the logarithm is separated from the inner structure of the matrix.
\item The map in Lemma~\ref{J_power_lemma} and its properties can be derived from the matrix logarithm and the matrix exponential map as follows. Even though this fact is well-known from, e.g., Lie theory, we provide here some details for the convenience of the reader.
The matrix~$J$ decomposes as
\[
 J = \lambda( I + N )\,,
\]
where $I$ is the $(d\times d)$-identity matrix and $N$ is the nilpotent matrix 
with ones ($1$'s) on the superdiagonal and zeros ($0$'s) elsewhere. Then 
\[
 \log J = \log(\lambda;\omega)\,I + \log(I+N)\,,
\]
where
\begin{align*}
 \log(I+N) & = \sum_{k=1}^\infty \frac{(-1)^{k+1}}{k} N^k = 
\sum_{k=1}^{d-1}\frac{(-1)^{k+1}}{k}N^k
 \\
 & = \begin{pmatrix}
      0 & 1 & -\frac12 & \ldots & \ldots  & \frac{(-1)^d}{d-1}
      \\
       & 0 & 1 & -\frac12 & \ldots & \frac{(-1)^{d-1}}{d-2}
       \\[1mm]
       & & \ddots & \ddots & \ddots & \vdots
       \\[1mm]
       & & & 0 & 1 & -\frac12 
       \\
       & & & & 0 & 1 
       \\
       & & & & & 0
     \end{pmatrix}\,.
\end{align*}
Here, we suppressed in the notation that the matrix logarithm~$\log J$ depends 
on the choice of~$\omega$. The matrix logarithm~$\log(I+N)$, however, is 
independent of~$\omega$. Further we have
\begin{align*}
 J^x & = \exp(x\log J) = \exp\bigl(x\log(\lambda;\omega)\,I\bigr) 
\exp\bigl(x\log(I+N)\bigr)
 \\
 & = \exp\bigl(x\log(\lambda;\omega)\bigr) \cdot 
\prod_{k=1}^{d-1}\exp\left(\frac{(-1)^{k+1}x}{k}N^k \right)
 \\
 & = \exp\bigl(x\log(\lambda;\omega)\bigr)
 \begin{pmatrix}
\binom{x}{0} & \binom{x}{1} & \binom{x}{2} & \ldots
& \binom{x}{d-1}
\\[1mm]
& \binom{x}{0} & \binom{x}{1} & \ldots & \binom{x}{d-2}
\\
& & \ddots & \ddots &   \vdots
\\[-2mm]
& & & \ddots & \binom{x}{1}
\\[1mm]
& & & & \binom{x}{0}
\end{pmatrix}
\\
& = \lambda^x 
\begin{pmatrix}
\binom{x}{0} & \binom{x}{1} & \binom{x}{2} & \ldots
& \binom{x}{d-1}
\\[1mm]
& \binom{x}{0} & \binom{x}{1} & \ldots & \binom{x}{d-2}
\\
& & \ddots & \ddots &   \vdots
\\[-2mm]
& & & \ddots & \binom{x}{1}
\\[1mm]
& & & & \binom{x}{0}
\end{pmatrix}
\,.
\end{align*}
Since the complex logarithm, understood in its initial form as the inverse 
relation to the exponential function, is not univalent, neither is the matrix 
logarithm. 
In turn, there are several real analytic  extensions of the homomorphism 
in~\eqref{eq:homJintegers}. The map in Lemma~\ref{J_power_lemma} is the 
unique real analytic extension that stays within real matrices for all real 
values of~$\lambda$ if $\omega\in\R_{<0}$. We remark that for 
Lemma~\ref{J_power_lemma} it is not important that the complex logarithm is 
holomorphic at~$\lambda$ (i.e., in some neighborhood of~$\lambda$). Below we 
will provide an alternative, rather hands-on proof of Lemma~\ref{J_power_lemma}.

Picking another branch of the complex logarithm and hence another real analytic extension of the homomorphism in~\eqref{eq:homJintegers} would cause a re-indexing of the Fourier coefficients in Theorems~\ref{thm:intro_single}, 
\ref{Fourier:main_theorem} and~\ref{thm:intro_general}, but not a qualitative 
change in the results.
\end{enumerate}

\end{remark}

\begin{proof}[Proof of Lemma~\ref{J_power_lemma}]
We first show that the map in~\eqref{eq:HomJreals} is a group homomorphism. To 
simplify notation we set 
\[
 \mcJ(x) \coloneqq J^x
\]
for all~$x\in\R$, and recall that the matrix entries of~$\mcJ(x)$ are
\begin{equation}\label{eq:mcJx_entries}
 \mcJ(x)_{ij} = 
 \begin{cases} 
  \lambda^x \binom{x}{j-i} & \text{if $i\leq j$}\,,
  \\
  0 & \text{otherwise,}
 \end{cases}
\end{equation}
for~$i,j\in\{1,\ldots, d\}$. Let $x_1,x_2\in\R$ and set $B \coloneqq 
\mcJ(x_1)\mcJ(x_2)$. In order to show that $B=\mcJ(x_1+x_2)$ we first note that 
the matrix~$B$ is upper triangular as a product of such matrices. For the 
matrix 
entries of~$B$ in the upper triangle we proceed by direct calculation, taking 
advantage of~\eqref{eq:mcJx_entries}. For the matrix entry~$B_{ij}$ with 
$i,j\in\{1,\ldots,d\}$, $i\leq j$, we find 
\begin{align*}
B_{ij} & = \sum_{k=1}^d \mcJ(x_1)_{ik}\mcJ(x_2)_{kj} = \sum_{k=i}^j 
\mcJ(x_1)_{ik}\mcJ(x_2)_{kj}
\\
& = \lambda^{x_1+x_2}\sum_{k=i}^j \binom{x_1}{k-i}\binom{x_2}{j-k}
\\
& = \lambda^{x_1+x_2}\binom{x_1+x_2}{j-i}\,,
\end{align*}
where we used the Chu--Vandermonde identity for the last equality. This shows 
that $B=\mcJ(x_1+x_2)$ and hence the map in~\eqref{eq:HomJreals}, called~$\mc 
J$ 
above, is a group homomorphism. 

It remains to show that the map~$\mc J$ coincides with the homomorphism 
in~\eqref{eq:homJintegers}. Since~$\mc J$ is already known to be a group 
homomorphism, it suffices to show the equality of both maps on~$\N$. A 
straightforward proof by induction confirms that for all~$m\in\N$, the $m$-fold 
product of the matrix~$J$ equals~$\mcJ(m)$, and hence the claimed equality of 
maps.
\end{proof}

\subsection{The $y$-dependent Fourier coefficient functions}

The main constituents of the $y$-dependent part of the Fourier coefficient 
functions in the Fourier expansion of an $A$-twisted Laplace eigenfunction are 
the modified Bessel functions and their derivatives, or certain specific linear 
combinations of~$y^a$ and $y^a\log y$ with $a\in\pm s +\tfrac12\Z$, all 
depending on the spectral parameter~$s$. 

The definition of the modified Bessel functions admits some flexibility as to 
their domain of holomorphy. Classically, they are defined as functions that are 
analytic on $\C\smallsetminus(-\infty, 0]$. See~\cite[p. 45]{Watson44}. Because 
we allow here any complex numbers as eigenvalues of the twisting 
endomorphism~$A$ in~\eqref{aut_f_c_1}, we need to slightly adapt this classical 
definition in order to match our needs (more precisely, to allow for more 
compact statements) and change the domain of holomorphy of the modified Bessel 
functions accordingly. The choice of the complex logarithm was already discussed 
in Section~\ref{sec:cut}. We will now discuss the construction of the modified 
Bessel functions.

We recall that the endomorphism~$A$ is invertible and has a single Jordan block 
by hypothesis. We further recall that the eigenvalue of~$A$ is 
denoted~$\lambda$.

\subsubsection{Modified Bessel functions with suitable domain of 
holomorphy}\label{sec:Bessel}

The modified Bessel functions are classically defined as functions that are 
holomorphic on the cut plane~$\C\setminus(-\infty,0]$. More precisely, the 
\emph{modified Bessel function of the first kind with index~$\eta\in\C$} is the 
map
\[
 I_\eta\colon\C^\times \to \C
\]
given by 
\begin{equation}\label{eq:defIclass}
I_\eta(z) \coloneqq e^{\eta \log\left(\frac{z}{2}; -1\right)} 
\sum_{m=0}^{\infty} \frac{1}{m! \Gamma(\eta+m+1)} 
\left(\frac{z}{2}\right)^{2m}\,.
\end{equation}
See~\cite[p. 77]{Watson44}. Since the principal logarithm~$\log(\;\cdot\;;-1)$ 
is holomorphic on the cut plane~$\C\setminus(-\infty,0]$, so is~$I_\eta$. The 
\emph{modified Bessel function of the second kind with index~$\eta\in\C$} is the 
map 
\[
 K_\eta\colon\C^\times \to\C\,,
\]
which, for $\eta\notin\Z$, is defined as 
\begin{equation}\label{eq:defKclassnonint}
 K_\eta(z) \coloneqq \frac{\pi}{2} \frac{I_{-\eta}(z)-I_\eta(z)}{\sin(\eta 
\pi)}\,.
\end{equation}
For $\eta\in\Z$, it is 
\begin{equation}\label{eq:defKclassint}
 K_\eta(z) \coloneqq \lim_{\xi\to\eta} K_\xi(z)\,.
\end{equation}
See~\cite[p. 78]{Watson44}. For any value of~$\eta$, the map~$K_\eta$ is 
holomorphic on~$\C\setminus(-\infty,0]$. For our applications we need a 
realization of each modified Bessel function that 
is holomorphic on~$\C\setminus\omega\R_{\geq0}$, where~$\omega\R_{\geq0}$ is 
the branch cut chosen in Section~\ref{sec:cut}. For that reason we will 
slightly 
redefine the maps in~\eqref{eq:defIclass}, \eqref{eq:defKclassnonint} 
and~\eqref{eq:defKclassint} using the complex logarithm 
map~$\log(\;\cdot\;;\omega)$ from Section~\ref{sec:cut} as follows. 
 
We define our variant of the modified Bessel function of the first kind with 
index~$\eta\in\C$ to be the map 
\[
 I_\eta(\;\cdot\;;\omega)\colon\C^\times \to\C
\]
given by
\begin{equation}\label{eq:defI}
I_\eta(z;\omega) \coloneqq e^{\eta \log\left(z/2;\omega \right)} 
\sum_{m=0}^{\infty} \frac{1}{m! 
\Gamma(\eta+m+1)}\left(\frac{z}{2}\right)^{2m}\,.
\end{equation}
For our variant of the modified Bessel function of the second kind with 
index~$\eta\in\C$ we use the map 
\[
 K_\eta(\;\cdot\;;\omega)\colon\C^\times \to\C\,,
\]
which for $\eta\notin\Z$ is given by
\begin{equation}\label{eq:defKnonint}
K_\eta(z;\omega) \coloneqq \frac{\pi}{2} 
\frac{I_{-\eta}(z;\omega)-I_\eta(z;\omega)}{\sin(\eta \pi)}\,.
\end{equation}
For $\eta\in\Z$ we set 
\begin{equation}\label{eq:defKint}
 K_\eta(z;\omega) \coloneqq \lim_{\xi\to\eta} K_\xi(z;\omega)\,.
\end{equation}
For any value of~$\eta$, the maps~$I_\eta(\;\cdot\;;\omega)$ 
and~$K_\eta(\;\cdot\;;\omega)$ are holomorphic 
on~$\C\setminus\omega\R_{\geq0}$. 

They are closely related to the classical modified Bessel functions. For the 
modified Bessel function of the first kind we find 
\begin{equation}\label{I_prop_def}
I_\eta(z; \omega) = 
\begin{cases}
I_\eta(z) \quad &\text{for $\arg z \in \Omega_\omega$}\,,
\\
I_\eta(z) \cdot e^{- 2 \pi \iu \eta \sgn( \arg\omega)} & \text{otherwise}\,,
\end{cases}
\end{equation}
where~$\Omega_\omega$ is the set defined in~\eqref{Omega_definition}. See 
Figure~\ref{fig:log_cut_I}.
\begin{figure}
\centering
\includegraphics[width=270pt]{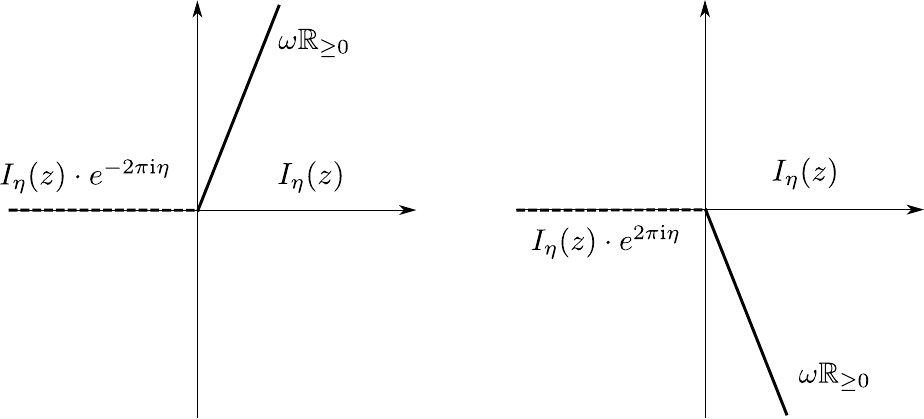}
\caption{The left figure indicates the relation 
between~$I_{\eta}(\;\cdot\;;\omega)$ and~$I_\eta$ for $\arg\omega>0$; the right 
figure for $\arg\omega<0$.} 
\label{fig:log_cut_I}
\end{figure}
For the modified Bessel function of the second kind with a non-integer 
index, say~$\eta\in\C\setminus\Z$, we have 
\begin{equation}\label{K_prop_def1}
 K_\eta(z;\omega) = \frac{\pi}{2} \frac{I_{-\eta}(z) - 
I_{\eta}(z)}{\sin(\eta\pi)} = K_\eta(z)
\end{equation}
if $\arg z\in\Omega_\omega$, and
\begin{align}\label{K_prop_def2}
K_\eta(z;\omega) 
& = \frac{\pi}{2} \frac{e^{2 \pi \iu \eta \sgn(\arg\omega)}I_{-\eta}(z)-e^{-2 
\pi \iu \eta \sgn (\arg\omega)}I_\eta(z)}{\sin(\eta \pi)} 
\\  \nonumber
&= \frac{\pi}{2} \cdot I_{-\eta}(z) \cdot  \frac{e^{2 \pi \iu \eta \sgn( 
\arg\omega)}  - e^{- 2 \pi \iu \eta \sgn(\arg\omega)}}{\sin(\eta \pi)} 
\\ \nonumber
& \quad + e^{- 2 \pi \iu \eta \sgn(\arg\omega)} K_\eta(z) 
\\ \nonumber
&=2 \pi \iu \sgn   (\arg\omega) \cdot  \cos(\pi \eta) \cdot I_{-\eta}(z) +  
e^{- 2 \pi \iu \eta \sgn (\arg\omega)} K_\eta(z)
\end{align} 
if $\arg z\notin\Omega_\omega$. See Figure~\ref{fig:log_cut_K}. For any integer 
index~$\eta\in\Z$, the relation simplifies to
\begin{equation}\label{K_prop_def3}
K_\eta(z;\omega) =
\begin{cases}
K_\eta(z) \quad & \text{for $\arg z \in \Omega_\omega$}\,,
\\
K_\eta(z)  + 2\pi\iu\, s(\omega)\cdot(-1)^\eta\cdot  I_{-\eta}(z) &  
\text{otherwise}\,,
\end{cases}
\end{equation}
where $s(\omega)\coloneqq \sgn(\arg\omega)$.
We note that for any value of~$\eta$, the maps~$I_\eta(\;\cdot\;;\omega)$ 
and~$K_\eta(\;\cdot\;;\omega)$ coincide with~$I_\eta$ and~$K_\eta$, 
respectively, on~$\R_{>0}$.

\begin{figure}
\centering
\includegraphics[width=350pt]{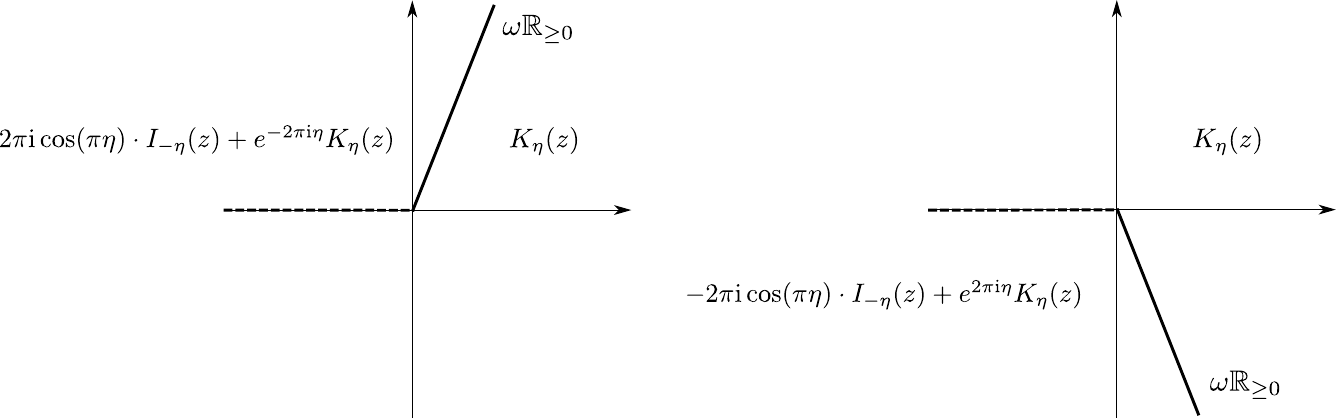}
\caption{The left figure shows the relation between~$K_{\eta}(\;\cdot\;; 
\omega)$ and~$K_\eta$ for $\arg\omega>0$; the right figure for $\arg\omega<0$.}  
\label{fig:log_cut_K}
\end{figure}

\subsubsection{The functions~$\fu$ and~$\futoo$ for 
$\alpha\not=0$}\label{sec:IandKpart1} 

With the complex logarithm and the modified Bessel function chosen adapted to 
the value of the eigenvalue~$\lambda$ of~$A$ we can now  define the principal 
constituents for the $y$-dependent Fourier coefficient functions, namely the 
two functions 
\[
 \fu,\futoo\colon \N_0 \times \left(\C\setminus\omega\R_{>0}\right) \times 
\R_{>0}\times \C \to \C\,,
\]
which provide the necessary generalization of the Fourier coefficient functions 
from the untwisted case. 

If the second argument is not zero, then these functions are essentially given 
by derivatives of the modified Bessel functions. More precisely, for any 
$m\in\N_0$, $\alpha\in\C\setminus\omega\R_{\geq0}$, $y\in\R_{>0}$ and $s\in\C$ 
we set 
\begin{align}
\label{eq:def_fu_general}
\fu(m,\alpha,y,s) & \coloneqq 
\frac{y^{\frac 1 2}}{\iu^m} \pa_\alpha^m  \left( I_{s-\frac 1 2} ( \alpha y; 
\omega) \right),
\intertext{and}
\label{eq:def_futoo_general}
\futoo(m,\alpha,y,s) &\coloneqq  \frac{y^{\frac 1 2}}{\iu^m} \pa_\alpha^m  
\left( K_{s-\frac 1 2} (\alpha y; \omega) \right)\,.
\end{align}
We caution that the differential operator~$\partial_\alpha^m$  is not the plain 
$m$-th derivative of~$\Psi_{s-\frac12}(\,\cdot\,;\omega)$ for 
$\Psi\in\{I,K\}$ but acts on the map
\[
 \alpha \mapsto \Psi_{s-\frac12}(\alpha y; \omega)\,,
\]
as shall be indicated by the additional brackets around $\Psi_{s-\frac12}(\alpha 
y;\omega)$. We further remark that the functions~$\fu$ and~$\futoo$ depend on 
the choice of~$\omega$. To keep the notational complexity to a minimum, we do 
not reflect the dependency in the symbols for these functions. However, it is 
visible from their domains.

If the second argument is zero, then the definition of the functions~$\fu$ and 
$\futoo$ is qualitatively different. We will provide it in the next section.

As we can see already in the statement of Theorem~\ref{thm:intro_single}, the 
Fourier coefficients with nonzero indices are linear combinations of the two 
functions~$\fu$ and~$\futoo$ for appropriate values of~$\alpha\not=0$. The 
zeroth Fourier coefficient, however, may also require the values of~$\fu$ 
and~$\futoo$ for $\alpha=0$.

\subsubsection{The functions~$\fu$ and~$\futoo$ for 
$\alpha=0$}\label{sec:IandKpart2}

If the eigenvalue~$\lambda$ of~$A$ is~$1$ (and only in this case), then we will need 
in addition the values of the functions~$\fu$ and~$\futoo$ with vanishing 
second argument for the Fourier expansion of the considered $A$-twisted Laplace 
eigenfunction, namely for its zeroth Fourier coefficient. As in the untwisted 
case, the zeroth Fourier coefficient is differently structured than all others, 
and attention needs to be paid to exceptional spectral parameters.  

In the untwisted case, $s=1/2$ is the only exceptional spectral parameter for 
the zeroth term of the Fourier expansion. In our situation, with the 
dimension~$d$ of the vector space~$V$ allowed to be any finite natural number, 
the set of exceptional values of the spectral parameter becomes
\begin{align} \label{s-specialvalues} 
E(d) & \coloneqq  \left\{ \frac12\pm j \setmid j=0, 1,\ldots,  \left\lfloor 
\frac{d-1}{2} \right \rfloor \right\}
\\
& \ = \left\{ \frac12 - \left\lfloor\frac{d-1}{2}\right\rfloor, \frac12 - 
\left( 
\left\lfloor\frac{d-1}{2}\right\rfloor + 1 \right),\ldots \right. \nonumber 
\\ & \hphantom{\frac12 - \left( \left\lfloor\frac{d-1}{2}\right\rfloor + 1 
\right)} \left.
\ldots, -\frac12, \frac12, \frac32,\ldots, \frac12 + \left\lfloor \frac{d-1}{2} 
\right \rfloor\right\}\,.
\nonumber
\end{align} 
The definitions of~$\fu$ and~$\futoo$ for vanishing second argument (thus, 
$\alpha=0$) require three series of combinatorial coefficients, which we now 
discuss. Each series consists of a function that is defined on a certain subset 
of~$\N_0\times\C$. 

The \textbf{first series} is given by a function, $\cone$, defined on the subset 
of~$\N_0\times\C$ that consists of the pairs~$(k,s)$ that satisfy
\begin{equation}\label{eq:cone_restr}
 s\notin\left\{ \frac12-j\setmid 
j=1,\ldots,\left\lfloor\frac{k}2\right\rfloor\right\}\,.
\end{equation}
We set 
\begin{equation}\label{eq:def_cone_start}
 \cone(0,s) \coloneqq \cone(1,s) \coloneqq 1
\end{equation}
and, for $k\geq 2$, 
\begin{equation}\label{def_cof_conc}
\cone(k,s)  \coloneqq 
(-1)^{\left\lfloor \frac{k}{2} \right\rfloor} \dfrac{k!}{2^{\left\lfloor 
\frac{k}{2} \right\rfloor} \left\lfloor \frac{k}{2} \right\rfloor! 
\prod\limits_{j=1}^{ \left\lfloor \frac{k}{2} \right\rfloor} (2j+2s-1)}\,.
\end{equation}
Using the convention that the empty product equals~$1$, the definitions 
in~\eqref{eq:def_cone_start} may be subsumed into~\eqref{def_cof_conc}. 
Moreover, by taking advantage of the Gamma function, we can express the 
definitions in~\eqref{eq:def_cone_start} and~\eqref{def_cof_conc} as
\begin{equation}
 \cone(k,s) = \left(-\frac14\right)^{\left\lfloor\frac{k}2\right\rfloor} 
\frac{\Gamma(k+1)\Gamma\left(s+\frac12\right)}{\Gamma\left(\left\lfloor\frac{k}
2\right\rfloor + 1\right)\Gamma\left(s+\left\lfloor\frac{k}2\right\rfloor + 
\frac12\right)}
\end{equation}
for all pairs~$(k,s)\in\N_0\times\C$ obeying~\eqref{eq:cone_restr}. 

The \textbf{second series} is given by a function, $\ctwo$, defined on the set
\[
 \N_0 \times \left( E(d) \cap \bigl(-\infty, \tfrac12\bigr]\right)\,.
\]
We set 
\begin{equation}\label{eq:ctwo_start1}
\ctwo\Bigl(0, \frac12\Bigr)\coloneqq \ctwo\Bigl(1, \frac12\Bigr) \coloneqq 1\,. 
\end{equation}
For $s \in E(d)\cap \bigl(-\infty,-\frac12\bigr]$ we set
\begin{equation}\label{eq:ctwo_start2}
\ctwo(0, s) \coloneqq  (2s) \cone(-2s-1,s)\,, \qquad \ctwo(1, s) \coloneqq  
(2s-2) \cone(-2s,s)\,.  
\end{equation}
For $k \geq 2$ and any $s \in E(d)\cap\bigl(-\infty,\frac12\bigr]$, we define 
\begin{align}
\ctwo (k, s) & \coloneqq  -\frac14 \frac{ 
(1-2s+k)(-2s+k)}{\left\lfloor\frac{k}{2}\right\rfloor \left(\frac 1 2 -s 
+\left\lfloor\frac{k}{2}\right\rfloor\right)}\, \ctwo(k-2, s) \label{defctwo}
\\[5mm]
& \ = 
\begin{cases}
\left(-\frac14\right)^{\left\lfloor\frac{k}{2}\right\rfloor} 
\dfrac{\prod\limits_{j=1}^k (1-2s+j)}{\left\lfloor\frac{k}{2}\right\rfloor! 
\prod\limits_{j=1}^{\left\lfloor\frac{k}{2}\right\rfloor} (1/2-s+j)}\, 
\ctwo(0,s) & \quad\text{for $k$ even,} \nonumber
\\[10mm]
\left(-\frac14\right)^{\left\lfloor\frac{k}{2}\right\rfloor} 
\dfrac{\prod\limits_{j=2}^k (1-2s+j)}{\left\lfloor\frac{k}{2}\right\rfloor! 
\prod\limits_{j=1}^{\left\lfloor\frac{k}{2}\right\rfloor} (1/2-s+j)}\, 
\ctwo(1,s) & \quad\text{for $k$ odd.}
\end{cases}\nonumber
\end{align}
Using~\eqref{def_cof_conc}, we may write~\eqref{defctwo} also as
\begin{equation}
\ctwo(k,s) = 
\begin{cases}
\displaystyle
\binom{k+1-2s}{k} \cone(k,1-s)\ctwo(0,s) & \text{for $k$ even,}
\\[5mm]
\displaystyle
\binom{k+1-2s}{k} \cone(k,1-s)\ctwo(1,s) & \text{for $k$ odd.}
\end{cases}
\end{equation}

\begin{remark}  
In~\eqref{eq:ctwo_start2} we use the function~$\cone$ at the value $(k,s) = 
(-2s-1,s)$ with $s\in E(d)\cap \bigl(-\infty,-\tfrac12\bigr]$. We now check 
that $\cone$ is indeed defined at this point. From $s\leq -\tfrac12$ it follows 
that $k\geq 0$. Further, since $s\in E(d)$, we have $k\in\N_0$ and 
\[
 \left\lfloor \frac{k}{2} \right\rfloor = -s -\frac12\,,
\]
and hence 
\[
 s < \frac12 - \left\lfloor \frac{k}{2} \right\rfloor = s + 1\,.
\]
Therefore, $(k,s)$ is an admissible pair for which $\cone$ can be applied (see 
also \eqref{eq:cone_restr}). Analogously, one checks that $\cone(-2s,s)$ is 
well-defined.

We further remark that for $s=\tfrac12$, neither the point~$(-2s-1,s) = (-2,s)$ 
nor the point~$(-2s,s)=(-1,s)$ are in the domain of~$\cone$. For this reason, 
\eqref{eq:ctwo_start1} cannot be subsumed into~\eqref{eq:ctwo_start2}.
\end{remark}

The \textbf{third series} is given by a function, $\cthree$, defined on  
\[
 \N_0 \times \left( E(d) \cap \bigl(-\infty,\tfrac12\bigr]\right)\,.
\]
We set 
\begin{equation}\label{eq:cthree_start1}
\cthree(0,s) \coloneqq  \cthree(1,s) \coloneqq 0
\end{equation}
and, for $k\in\{2,3\}$, 
\begin{equation}\label{eq:cthree_start2}
\cthree(k,s) \coloneqq  - \frac{\frac 1 2 - s +2\lfloor \frac{k}{2}\rfloor 
}{2\lfloor\frac{k}{2}\rfloor\,\left(\frac 1 2 - s 
+\lfloor\frac{k}{2}\rfloor\right)}\,\ctwo(k,s)\,.
\end{equation}
For $k \geq 4$ we set
\begin{align}\label{eq:def_cthree}
\cthree(k, s) & \coloneqq  - \frac{\frac 1 2 - s +2\lfloor \frac{k}{2}\rfloor 
}{2\lfloor\frac{k}{2}\rfloor\,\left(\frac 1 2 - s 
+\lfloor\frac{k}{2}\rfloor\right)}\,\ctwo(k,s) 
\\ \nonumber
& \qquad + \sum_{j=1}^{\lfloor\frac{k}{2}\rfloor-1} 
(-1)^{\lfloor\frac{k}{2}\rfloor+j+1}\,\frac{\frac 1 2 - s +2j}{2j\,\left( \frac 
1 2 - s +j \right)}\,\frac{1}{4^{\lfloor\frac{k}{2}\rfloor-j}}
\\
&\qquad\qquad \hphantom{\coloneqq}\ \times \frac{ 
\prod\limits_{\ell=2j+1+\delta_\odd(k)}^{k} (1-2s+\ell)}{ 
\prod\limits_{\ell=j+1}^{\lfloor\frac{k}{2}\rfloor} \ell\,\left( \frac 1 2 - s 
+\ell\right)}\, \ctwo(2j+ \delta_\odd(k), s)\,, \nonumber
\end{align}
where
\begin{equation}\label{eq:def_deltaodd}
\delta_\odd(k) \coloneqq \begin{cases} 1 & \text{for $k$ odd,} \\ 0 & \text{for 
$k$ even.} \end{cases}
\end{equation}
Using the convention that empty sums equal~$0$, we may 
subsume~\eqref{eq:cthree_start2} in~\eqref{eq:def_cthree}.

With these preparations we can now define the functions~$\fu$ and~$\futoo$ for arguments with vanishing second component.
We recall that the function~$\cone$  is not defined on all of~$\N_0\times\C$. See~\eqref{eq:cone_restr}.
To simplify the statements of the functions~$\fu$ and~$\futoo$ in 
Definition~\ref{defazero}, we extend the function~$\cone$---somewhat 
arbitrarily---to the pairs $(k,s)\in\N_0\times\C$ with $s$ being contained in 
the ``forbidden'' set in~\eqref{eq:cone_restr} by setting 
\[
 \cone(k,s)\coloneqq 1\,.
\]
In our applications, we will never use the functions~$\fu$ and~$\futoo$ at 
those arguments for which the definition uses the function~$\cone$ at these 
special values.  

\begin{defi}\label{defazero} 
Let $m\in\N_0$, $y\in\R_{>0}$ and $s\in\C$. 
\begin{enumerate}[label=$\mathrm{(\roman*)}$, ref=$\mathrm{\roman*}$]
\item For
\[
 s\in E(d) \cap \left[\frac12-\frac{m}{2},\frac12\right]
\]
we set 
\begin{align}\label{def:azeroslesshalf} 
\begin{aligned}
\fu(m,0,y,s) \coloneqq  \ctwo(m-1+2s, s) & y^{s + 
2\lfloor\frac{m}{2}\rfloor}\log y 
\\
 & + \cthree(m-1+2s,s) y^{s+2\lfloor\frac{m}{2}\rfloor}\,.
\end{aligned}
\end{align}
For all other values of~$s$ we set 
\begin{equation}\label{def:I_azeronoted}
\fu(m,0,y,s) \coloneqq \cone(m,s)\, y^{s+ 2\lfloor\frac{m}{2}\rfloor}\,.
\end{equation}
\item For 
\[
 s\in E(d) \cap \left[\frac32,\frac12 + \frac{m}2 \right]
\]
we set 
\begin{align}\label{def:azerosbiggerhalf} 
\begin{aligned}
\futoo(m,0,y,s) \coloneqq  \ctwo(m +1 &-2s, 1-s) y^{1 - s + 
2\lfloor\frac{m}{2}\rfloor}\log y
\\
& + \cthree(m+1-2s, 1-s) y^{1-s+2\lfloor\frac{m}{2}\rfloor}\,.
\end{aligned}
\end{align}
For all other values of~$s$ we set 
\begin{equation}\label{def:K_azeronoted} 
\futoo (m, 0, y, s) \coloneqq  \cone(m,1-s)\, y^{1 -s + 
2\lfloor\frac{m}{2}\rfloor}\,.
\end{equation} 
\end{enumerate}
\end{defi}

\subsection{Matrices of scalar coefficients}\label{sec:coeffmatrices}

As a last component for the statement of the refined version of 
Theorem~\ref{thm:intro_single} we now define two sets of coefficient matrices. 

For $n, k\in\N_0$ we denote by~$s(n,k)$ the \emph{signed Stirling number of the first kind}, which is the number of permutations of $n$~elements with 
$k$~disjoint cycles.
The Stirling numbers are also characterized as the coefficients in the polynomial
\begin{equation}\label{eq:charstirling}
\prod_{\ell=0}^{n-1} (x-\ell) = \sum_{k=0}^n s(n,k)x^k\,,\qquad x\in\R\,,
\end{equation}
where we use again the convention that empty products equal~$1$, and empty sums equal~$0$.
In particular, we have
\[ 
s(0,0) =1\,, \quad\text{and}\quad s(n, k) = 0 \quad\text{for $n,k \in \N\,,\  
n<k$\,.}
\]
Further we use $\stirmat=(\stirelem_{i j})$ to denote the $(d \times d)$-\emph{Stirling matrix}.
We will use it as a transformation to untangle a system of differential equations, hence the notation~$U$ for this matrix.
Its entries are  
\begin{equation}\label{def:stirlingmatrix} 
\stirelem_{i j} = \frac{s(d-i,d-j)}{(d-i)!} 
\end{equation} 
for $i,j\in\{1,\ldots, d\}$.

We denote by~$\Coeff_{1;d}$ the subset of matrices~$(c_{ij})\in \Mat(d\times d; \C)$ whose entries satisfy 
\[
c_{ij} = 0 \quad\text{for $i>j$}\,,
\]
and for~$i\in\{1,\ldots, d\}$ and~$p\in\{0,\ldots,d-i\}$
\begin{align*}
c_{i,i+p} &=  c_{1,1+p} \, \frac{\prod\limits_{k=1}^{i-1} 
(d-p-k)}{\prod\limits_{k=1}^{i-1}(d-k)}\,.
\end{align*}
Finally, we let $\Coeff_{0;d}$ denote the subset of matrices~$(c_{ij})\in 
\Mat(d\times d; \C)$ whose entries satisfy 
\begin{align*}
c_{i+2k,j+2k} =  c_{ij}\, 
\frac{\prod\limits_{p=0}^{2k-1}\big(d-(j+p)\big)}{\prod\limits_{p=0}^{2k-1}
\big(d-(i+p)\big)}
\end{align*}
for all possible combinations of~$i,j,k\in\{1,\ldots, d\}$, and
\[
c_{ij} = 0 \qquad\text{for $i\in\{3,\ldots, d\},\ j\in\left\{1,\ldots, 
2\left\lfloor \frac{i-1}2\right\rfloor\right\}$\,.}
\]
We note that $\Coeff_{1;d}$ is a $d$-dimensional complex subspace of 
upper triangular matrices, the entries of which are completely determined by 
the entries of the first row. For $d=1$, the set~$\Coeff_{0;d}$ is 
$1$-dimensional, coinciding with~$\C$.
For $d\geq 2$, the set~$\Coeff_{0;d}$ is a $2d$-dimensional subspace of $\Mat(d\times d;\C)$, the entries of which are determined by the first two rows.

\subsection{Refined statement of 
Theorem~\ref{thm:intro_single}}\label{sec:refined}

With all definitions in place, we can now present one of our main results, the refined statement of Theorem~\ref{thm:intro_single}.
To that end, for each~$n\in\Z$, we set
\begin{equation}\label{eq:def_alpha_n_lambda}
 \alpha_n \coloneqq 2\pi n - \iu \log(\lambda;\omega)\,,
\end{equation}
pick $\eps_n\in\{0,1\}$ such that 
\begin{equation}\label{eq:pickeps}
 (-1)^{\eps_n}\alpha_n \notin\omega\R_{>0}\,,
\end{equation}
and set
\begin{equation}\label{eq:def_tildealpha}
 \tilde\alpha_n \coloneqq (-1)^{\eps_n}\alpha_n\,.
\end{equation}
Further we let 
$S\in\Mat(d\times d;\C)$ be the ``alternating'' diagonal matrix 
\begin{equation}\label{eq:def_S}
 S \coloneqq 
 \begin{pmatrix}
  (-1)^{d-1}
  \\
  & (-1)^{d-2} 
  \\
  & & \ddots 
  \\
  & & & -1 
  \\
  & & & & 1 
  \\
  & & & & & -1 
  \\
  & & & & & & 1
 \end{pmatrix}\,,
\end{equation}
the \emph{sign matrix}.

\begin{thm}[Single Jordan block; refined statement]\label{Fourier:main_theorem} 
Let $V$ be a $d$-dimen\-sional complex vector space, identified with~$\C^d$. Let 
$f\colon\h\to V$ be a smooth map such that 
\[
 \Delta f = s(1-s)f
\]
for some $s\in\C$, and 
\[
 f(z+1) = Jf(z) \qquad\text{for all~$z\in\h$,}
\]
where $J=J(\lambda)$ is a Jordan-like matrix with $\lambda\not=0$. Then the 
Fourier expansion of~$f$ is 
\[
 f(x+\iu y) = \sum_{n\in\Z} J^x \hat f_n(y,s) e^{2\pi \iu n x}
\]
for all $z=x+\iu y\in\h$, where, for each~$n\in\Z$, the Fourier coefficient 
function is 
\[
 \hat f_n(y,s) =  \stirmat   C_n  S^{\eps_n}\begin{pmatrix}
\fu\bigl(d-1, \tilde\alpha_n, y,s\bigr) \\
\vdots   \\
\fu\bigl(1, \tilde\alpha_n, y,s\bigr) \\
\fu\bigl(0, \tilde\alpha_n, y,s\bigr)
\end{pmatrix} + \stirmat  D_n S^{\eps_n}\begin{pmatrix}
\futoo\bigl(d-1, \tilde\alpha_n, y,s\bigr) \\
\vdots   \\
\futoo\bigl(1, \tilde\alpha_n, y,s\bigr) \\
\futoo\bigl(0, \tilde\alpha_n,y,s\bigr)
\end{pmatrix}
\] 
with
\[
C_n, D_n \in \begin{cases}
\Coeff_{0;d} \quad &\text{if $\tilde\alpha_n=\alpha_n=0$}\,, \\
\Coeff_{1;d} \quad &\text{otherwise.}
\end{cases}
\]
\end{thm}

The proof of Theorem~\ref{Fourier:main_theorem} will be discussed in 
Sections~\ref{sec:periodization}--\ref{sec:proof_nonzero}.

\section{Complexity reduction}\label{sec:reduction}

In this section we show how Theorem~\ref{thm:intro_general} follows immediately 
from Proposition~\ref{prop:intro_vanish} and Theorem~\ref{Fourier:main_theorem} 
(or Theorem~\ref{thm:intro_single}). In addition we dispatch 
Proposition~\ref{prop:intro_vanish}, which is the ``trivial case'' in which the 
Fourier expansion vanishes identically. 

Throughout this section let $V$ be a finite-dimensional complex vector space 
and let $A$ be any endomorphism of~$V$. We do not assume that~$A$ acts 
irreducibly or is invertible.

\begin{proof}[Proof of Theorem~\ref{thm:intro_general} assuming  
Proposition~\ref{prop:intro_vanish} and Theorem~\ref{Fourier:main_theorem}]
We consider a smooth map $f\colon\h\to\nobreak V$ that satisfies
\[
 \Delta f = s(1-s)f\qquad\text{and}\qquad f(z+1) = Af(z) \quad\text{for 
all~$z\in\h$}\,.
\]
We fix a basis of~$V$ with respect to which the endomorphism~$A$ is represented 
by a matrix~$J$ in Jordan-like normal form, say
\[
 J = 
 \begin{pmatrix}
  J_1 
  \\
  & \ddots 
  \\
  &  & J_p
 \end{pmatrix}
\]
such that for each $j\in\{1,\ldots, p\}$, the submatrix~$J_j$ is a Jordan-like 
block matrix of size~$(d_j\times d_j)$ with eigenvalue~$\lambda_j\in\C$. 
See~\eqref{eq:Jordan_intro}--\eqref{eq:intro_lambda}. Let 
\[
 V = V_1 \oplus \ldots \oplus V_p\qquad\text{and}\qquad f = f_1 
\oplus\ldots\oplus f_p
\]
be the associated decompositions of~$V$ and~$f$. The linearity of~$A$ and the 
property that $f$ is $A$-twisted imply that for each $j\in\{1,\ldots,p\}$, the 
map $f_j$ is $J_j$-twisted:
\[
 f_j(z+1) = J_jf_j(z) \quad\text{for all $z\in\h$}\,.
\]
For $j\in\{1,\ldots,p\}$ with $\lambda_j=0$, 
Proposition~\ref{prop:intro_vanish} 
now implies that $f_j=0$. Moreover, the linearity of~$\Delta$ and the property 
that $f$ is a $\Delta$-eigenfunction with spectral parameter~$s$ imply that 
$f_j$ is so for each $j\in\{1,\ldots, p\}$. Further, if we fix for all 
$j\in\{1,\ldots, p\}$ with $\lambda_j=0$ any (possibly non-continuous) 
extension of the map 
\[
 \N_0 \to \Mat(d_j\times d_j;\C)\,,\quad n\mapsto J_j^n\,,
\]
to a map (not necessarily a homomorphism)
\[
 \R\to\Mat(d_j\times d_j;\C)\,,\quad x\mapsto J_j^x\,,
\]
and use the extensions discussed in Lemma~\ref{J_power_lemma} for those 
$j\in\{1,\ldots, p\}$ with $\lambda_j\not=0$, then 
\[
 J^x = 
 \begin{pmatrix}
  J_1^x 
  \\
  & \ddots 
  \\
  &  & J_p^x
 \end{pmatrix}
\]
for all~$x\in\R$. Therefore, linearity yields that the Fourier expansion of~$f$ 
is the direct sum of the separate Fourier expansions of the functions~$f_j$, 
$j\in\{1,\ldots,p\}$. Now taking advantage of the Fourier expansion for the 
components~$f_j$ for $j\in\{1,\ldots,p\}$ with $\lambda_j\not=0$ as developed 
in Theorem~\ref{Fourier:main_theorem} completes the proof of 
Theorem~\ref{thm:intro_general}.
\end{proof}

As a first step to establish the conditions on which the previous proof relies 
we now prove the following statement on vanishing eigenfunctions for nilpotent 
endomorphisms, of which Proposition~\ref{prop:intro_vanish} is a special case.

\begin{prop}\label{p:lambdais0}  
Suppose that the map~$f\colon \h \to V$ satisfies
\[
 f(z+1) = A f(z)
\]
for all $z \in \h$. If $A$ is nilpotent, then $f=0$.
\end{prop}
 
\begin{proof} 
 Let $d\coloneqq\dim V$. For all $z \in \mathbb H$ we have $f(z+d)=A^d f(z)$. 
Now $A^d=0$ yields $f=0$.
\end{proof}

\section{Periodization}\label{sec:periodization}

In the course of this section and the following ones we will provide a proof 
of Theorem~\ref{Fourier:main_theorem}. Throughout these sections we suppose 
that $V$ is a complex vector space of dimension $d\in\N$ and that $A$ is an 
invertible endomorphism of~$V$ that acts irreducibly on~$V$. We fix a basis 
of~$V$ with respect to which $A$ is represented by the Jordan-like matrix 
\[
 J = J(\lambda) = 
 \begin{pmatrix}
  \lambda & \lambda
  \\
  & \ddots & \ddots 
  \\
  & & \lambda & \lambda 
  \\
  & & & \lambda
 \end{pmatrix}
\]
for some~$\lambda\in\C^\times$, which is then the eigenvalue of~$A$. Using this 
basis, we identify the vector spaces~$V$ and~$\C^d$. Further, we fix a 
complex logarithm via a choice of~$\omega\in\C$ satisfying the properties in 
Section~\ref{sec:cut}, and we let $f\colon\h\to V$ be a smooth map that is 
$J$-twist periodic (or $A$-twist periodic) with period~$1$:
\[
 f(z+1) = Jf(z) \quad\text{for all $z\in\h$}\,.
\]
Starting with the next section, we will suppose in addition that $f$ is a 
Laplace eigenfunction. For the results of the present section, this additional 
hypothesis is not needed.

At the end of Section~\ref{s:intro} we briefly outlined the strategy of the 
proof of Theorem~\ref{Fourier:main_theorem}. As explained there, the first 
major obstacle towards a Fourier expansion of~$f$ is its non-periodicity. In 
this section we will show how to overcome it by means of a periodization, which 
uses in a crucial way that the map
\[
 \R\to\Mat(d\times d;\C)\,,\quad x\mapsto J^x\,,
\]
is a group homomorphism (see Lemma~\ref{J_power_lemma}). We note that this approach is also known from the study of monodromy of ordinary differential equations, vector bundles and modular forms. In addition we will 
provide a vanishing statement (Lemma~\ref{fourier_observation}) and a 
differentiability statement (Lemma~\ref{lem:diff}) that will be important for 
investigating the fine structure of the Fourier coefficient functions.

\begin{prop}\label{prop:periodization}
The map $F\colon\h\to V$, 
\[
 F(z) \coloneqq J^{-x} f(z)\qquad (z=x+\iu y\in\h)\,,
\]
is smooth and periodic with period~$1$.
\end{prop}

\begin{proof}
To establish the periodicity of~$F$, we recall from Lemma~\ref{J_power_lemma} 
that the map 
\begin{equation}\label{eq:inprop_J}
 \R\to\Mat(d\times d;\C)\,,\quad x\mapsto J^x\,,
\end{equation}
is a group homomorphism. For any $z=x+iy\in\h$ we find
\begin{align*}
 F(z+1) & = J^{-(x+1)}f(z+1) = J^{-x-1}Jf(z) = J^{-x}f(z) = F(z)\,. 
\end{align*}
The smoothness of~$F$ follows immediately from the smoothness of~$f$ and the 
real-analyticity of the map in~\eqref{eq:inprop_J} (see 
Lemma~\ref{J_power_lemma}).
\end{proof}

Since the map~$F$ defined in Proposition~\ref{prop:periodization} is smooth and 
periodic with period~$1$, we may expand it in a Fourier series:
\begin{equation}\label{eq:FourierF}
F(z) = J^{-x} f(z) = \sum_{n \in \Z} \hat f_n(y) e^{2\pi \iu n x}
\end{equation}
for all~$z=x+\iu y\in\h$, with suitable Fourier coefficient functions~$\hat 
f_n$, $n\in\Z$. Re-arranging, we obtain a Fourier expansion for~$f$, 
\begin{equation}\label{quasi_Fourier_expansion}
f(z) = \sum_{n \in \Z} J^{x} \hat f_n(y) e^{2 \pi \iu n x} \qquad (z=x+\iu y\in 
\h)\,.
\end{equation}

In order to find a more explicit description of the Fourier coefficient 
functions~$\hat f_n$, $n\in\Z$, we will differentiate the Fourier series 
in~\eqref{quasi_Fourier_expansion} term by term. See Section~\ref{s:system}. 
This will result in summands involving the functions
\[
\R\to\C\,,\quad x\mapsto x^k e^{2 \pi \iu  n x}\,,
\]
for $k \in \N_0 \cap [0, d]$ and $n\in\Z$. The polynomial terms are caused 
by the presence of~$J^{-x}$.  Whereas the functions in~$\{e^{2\pi \iu  n 
x}\}_{n \in \Z}$ are pairwise orthogonal in~$L^2 ([0,1])$, the functions 
in~$\{x^k e^{2\pi \iu  n x} \}_{n \in \Z, k \in \N_0\cap [0,d]}$ lose this 
property.  We therefore require the following lemma in order to push the study 
of the coefficient functions further.

\begin{lemma}\label{fourier_observation} 
Let $(P_n)_{n \in \Z}$ be a sequence of complex polynomials with uniformly 
bounded degrees, i.e., $\sup_{n \in \Z} \deg P_n < \infty$. Suppose 
that for all~$x \in \R$,
\begin{equation}\label{fourier_mod}
\sum_{n \in \Z} P_n(x)e^{2 \pi \iu n x} = 0\,.
\end{equation}
Then $P_n=0$ for all~$n \in \Z$ and, in particular, $(e^{2 \pi \iu n x} 
P_n(x))|_{x=0}=0$.
\end{lemma}

\begin{proof} 
The second statement clearly follows from the first. We will prove the first 
statement by induction on $\max_{n\in\Z} \deg P_n \eqqcolon \mathfrak{D}$. We 
use the convention that the degree of the zero polynomial is $-1$. Let 
$\mathfrak{D} \leq 0$.  Then each polynomial $P_n$, $n\in\N$, is constant (also 
allowing the constant~$0$). Say, $P_n=p_n\in\C$ for $n\in\Z$. 
By~\eqref{fourier_mod} we have
\[
\sum_{n\in \Z} p_n e^{2 \pi \iu n x} = 0\qquad\text{for all~$x\in\R$}\,.
\]
Elementary Fourier analysis implies $p_n = 0$ for all~$n\in\Z$. Now we proceed 
to the induction step by assuming that the statement holds true for all 
polynomials of degree at most $\mathfrak{D}-1$ with $\mathfrak{D} > 0$. Let 
$(P_n)_{n\in \Z}$ be a sequence of polynomials of degree at most~$\mathfrak{D}$ 
which satisfies \eqref{fourier_mod}. For all $x \in \R$ we have therefore
\[ 
\sum_{n \in \Z} P_n (x) e^{2\pi \iu n x} = 0 \qquad\text{and}\qquad 
\sum_{n\in \Z} P_n (x+1) e^{2\pi \iu n (x+1)} = 0\,, 
\]
and hence (note that $e^{2\pi\iu n} = 1$)
\[ 
\sum_{n \in \Z} e^{2\pi \iu n x} \left( P_n (x+1) - P_n (x) \right) = 0\,.
\]
Since $\mathfrak D > 0$, $P_n(x+1) - P_n (x)$ is a polynomial of degree at most 
$\mathfrak D - 1$ for all~$n\in\Z$. Thus, the inductive hypothesis implies that 
$P_n(x+1) - P_n (x) = 0$ for all~$x\in\R$ and all~$n\in\Z$.  Consequently, each 
polynomial~$P_n$, $n\in\Z$, is constant.  By the base case of induction, $P_n = 
0$ for all $n  \in \Z$.  
\end{proof} 

We end this section by showing a differentiability property of the series 
in~\eqref{quasi_Fourier_expansion}.

\begin{lemma}\label{lem:diff}
The series in~\eqref{quasi_Fourier_expansion} may be differentiated infinitely 
often term by term both in~$x$ and in~$y$.  
\end{lemma}

\begin{proof}
We recall the map~$F$ from Proposition~\ref{prop:periodization}, 
\[
 F(z) = J^{-x}f(z)\qquad (z=x+\iu y\in\h)\,.
\]
To prove that the series in~\eqref{quasi_Fourier_expansion} has the claimed 
differentiability properties, it suffices to establish them for the Fourier 
series of~$F$, given in~\eqref{eq:FourierF}. For every fixed $y>0$, the map 
\[
 \R\to V\,,\quad x\to F(x+\iu y)\,,
\]
is smooth by Proposition~\ref{prop:periodization}, and the series 
in~\eqref{eq:FourierF} is its Fourier series (with the functions~$\hat f_n$ 
evaluated at the fixed value of~$y$). Therefore the derivative of this Fourier 
series (in~$x$) is given by differentiating it termwise in~$x$. 
See \cite[Theorem~2.3]{folland2009fourier}. Iteration of this argument proves 
the claimed differentiability properties in~$x$.

For differentiation in~$y$ we use the Leibniz integral rule.  For $k, 
m\in\N_0$, $n \in \Z$ and~$y>0$ we have
\begin{align*}
\hat f_n(y) & = \int_0^1 F(x+\iu y) e^{- 2 \pi \iu  n x} dx 
\intertext{and}
(2 \pi n)^{4m} \pa_y ^k \hat f_n(y) &=\int_0^1 \pa_y ^k  F(x + \iu y) \cdot (2 
\pi n)^{4m} e^{- 2 \pi \iu nx}  dx \\
& =\int_0^1 \pa_y ^k  F(x+ \iu y) \cdot \pa_x ^{4m} e^{- 2 \pi \iu nx}  dx\,.
\end{align*}
The exchange of differentiation and integration in the second equation is 
justified by smoothness of the integrand. Integrating by parts together with 
the fact that $F$ is $1$-periodic shows that the boundary terms vanish. 
Therefore we obtain
\begin{align*}
(2 \pi n)^{4m} \pa_y ^k \hat f_n(y) &= \int_0^1 \pa_x ^{4m} \pa_y ^{k} F(x+\iu 
y)  \cdot e^{-2 \pi \iu nx} dx\,.
\end{align*}
Hence for any compact set $K$ in $\h$, there exists $C_{k, K, m}>0$ such that 
for any $z=x+\iu y \in K$,
\begin{align*}
\left| \pa_y ^k \hat f_n(y) \right| \leq  C_{k, K, m}  n^{-4m}\,.
\end{align*}
Consequently, for any $z = x + \iu y \in K$,
\begin{align*}
\left| e^{2\pi \iu  n x} J^{-x} \pa_y ^k  \hat f_n(y)\right| \leq C\cdot \left| 
x^{d } \pa_y ^k \hat f_n(y) \right| \le  \tilde C_{k, K, m} |x^d| n^{-4m}
\end{align*}
for appropriate positive constants $C$ and $\tilde C_{k,K,m}>0$, the latter 
potentially depending on~$k,K,m$.
We recall that $F(z) = J^{-x}f(z)$. Hence, for sufficiently large $m$, the 
$k$-th termwise derivative of~\eqref{eq:FourierF} in~$y$ converges absolutely 
and uniformly on compact subsets of~$\h$. Applying the Leibniz integral rule 
completes the proof.
\end{proof}

\section{Differential equations for the Fourier coefficient 
functions}\label{s:system}

Starting with this section we add to the hypotheses from 
Section~\ref{sec:periodization} that $f$ be a Laplace eigenfunction. Thus, in 
total we suppose that $V$ is a complex vector space of dimension~$d\in\N$ which 
we identify throughout with~$\C^d$, and $A$ is an endomorphism of~$V$ that is 
represented by the Jordan-like block matrix~$J = J(\lambda)$ with 
$\lambda\in\C$, 
$\lambda\not=0$. Further $f\colon\h\to V$ is $J$-twist periodic (or, 
equivalently, $A$-twist periodic), that is 
\[
 f(z+1) = Jf(z) \qquad\text{for all~$z\in\h$}\,,
\]
and moreover, $f$ is a Laplace eigenfunction
\begin{equation}\label{eq:laplaceeigen_sec5}
 \Delta f = s(1-s)f\,.
\end{equation}
Unless the Laplace eigenvalue~$s(1-s)$ of~$f$ is $1/4$, there are two choices 
for the spectral parameter~$s$ of~$f$. In this section, all considerations will 
be independent of this choice as all results only involve the 
eigenvalue~$s(1-s)$, not the spectral parameter. Nevertheless, we will express 
the eigenvalue using the spectral parameter to keep the necessary notation to a 
minimum.
We fix a complex logarithm by choosing~$\omega\in\C$ with the properties as in 
Section~\ref{sec:cut} and recall from~\eqref{quasi_Fourier_expansion} that~$f$ 
admits a Fourier expansion of the form 
\begin{equation}\label{eq:fourier_sec5}
 f(z) = \sum_{n\in\Z} J^x \hat f_n(y) e^{2\pi \iu n x} \qquad (z=x+\iu 
y\in\h)\,.
\end{equation}
In this section we will show that just as in the scalar case, each of the 
Fourier coefficient functions, $\hat f_n$, satisfies a certain differential 
equation, depending on~$n\in\Z$. In contrast to the scalar case, the 
differential equation is vector-valued, so indeed a system of differential 
equations.   

\begin{lemma}\label{descent_to_a_system} 
The Fourier coefficient function $\hat f_n$, $n\in\Z$, satisfies the system of 
second order differential equations (on $\R_{>0}$)
\begin{equation}\label{d_entry_matrix}
\left( y^{2}\pa_y ^2 + s(1-s) +  \left. y^{2}  \pa_x ^2 \right|_{x = 0} (e^{2 
\pi \iu  n x } J^x)\right) \hat f_n(y)
= 0\,.
\end{equation}
\end{lemma}

\begin{proof}
To evaluate the action of the Laplacian~$\Delta$ on the function~$f$ we take 
advantage of the Fourier expansion~\eqref{eq:fourier_sec5} of~$f$ and may 
pull the Laplace operator~$\Delta$ onto each summand by Lemma~\ref{lem:diff}. 
From this and the property of~$f$ to be a Laplace eigenfunction with spectral 
parameter~$s$ (see~\eqref{eq:laplaceeigen_sec5}) we obtain 
\begin{equation}\label{eq:systemdiffeq}
\begin{aligned}
0 & = \bigl( s(1-s) - \Delta\bigr)f(z) 
\\
& = \sum_{n \in \Z} \bigl( s(1-s) e^{2\pi \iu  n x}J^x + y^{2} \partial^2_{x} 
\left(e^{2\pi \iu  n x} J^x\right) + y^{2} e^{2\pi \iu  n x} J^x 
\partial^2_{y}\bigr) \hat f_n(y)
\end{aligned}
\end{equation}
for all $z\in\h$. We consider~\eqref{eq:systemdiffeq} as a system of 
differential equations in the variable~$y$ on $\R_{>0}$ with parameter $x\in\R$, 
to which $\hat f_n$ is an ($x$-independent) joint solution for all values 
of~$x$. To evaluate this system of differential equations further, we recall 
from Sections~\ref{sec:cut} and~\ref{sec:jordanblocks} that 
\[
 J^x = J(\lambda)^x = \lambda^x J(1)^x
\]
for all~$x\in\R$, and that $\lambda^x = \exp\bigl(x\log(\lambda;\omega)\bigr)$. 
By direct calculation we find 
\begin{align*}
\partial^2_{x} (e^{2\pi \iu  n x} J^x) & =  e^{2\pi \iu  n x} \left( (2 \pi \iu 
n)^2 J^x + 4 \pi \iu  n \partial_x J^x + \partial_x^2 J^x\right) 
\\
& = \lambda^x e^{2\pi \iu  n x} \big[  (2 \pi \iu n)^2  J(1)^x  + 4 \pi \iu n 
\bigl(\log(\lambda;\omega)  J(1)^x +  \partial_{x} J(1) ^x\bigr)  
\nonumber  \\& \phantom{\lambda  e^{2\pi \iu  n x} \times \left(   \right.}
+  \log(\lambda;\omega)^2  J(1)^x +  \partial^2_{x} J(1)^x + 2 
\log(\lambda;\omega)  \partial_{x} J(1)^x \big]\,. 
\end{align*}
The system of differential equations~\eqref{eq:systemdiffeq} is therefore
\begin{equation}\label{eq:sd_first}
\begin{aligned}
0  & = \lambda^x \sum_{n \in \Z} e^{2\pi \iu  n x} \Big(  s(1-s) J(1)^x + y^2 
\Big[ (2\pi \iu n)^2 J(1)^x   
\\
&\qquad + \log(\lambda;\omega)^2 J(1)^x + \partial^2_{x} J(1)^x + 2 \log(\lambda;\omega) \partial_{x} J(1)^x  
\\
&\qquad + 
4 
\pi \iu  n \log(\lambda;\omega) J(1)^x+ 4 \pi \iu  n \pa_x J(1)^x + J(1)^x \pa_y ^2 \Big] \Big) \hat f_n(y)\,.
\end{aligned}
\end{equation}
We note that $\lambda$, $s$ and~$\omega$ are fixed, and in particular are 
independent of $x$. Further, we recall from Lemma~\ref{J_power_lemma} that each 
of the entries of~$J(1)^x$ is polynomial in~$x$. Therefore for each fixed $y 
\in 
\R_{>0}$, the right hand side of~\eqref{eq:sd_first} is of the form 
\[
\lambda^x \sum_{n \in \Z} e^{2\pi \iu  n x} P_n(x) = 0\,,
\]
where, for each~$n \in \Z$, the function~$P_n$ is a polynomial in~$x$, and its 
degree is bounded by the dimension, $d$, of the vector space~$V$. Since 
$\lambda^x$ does not vanish, we obtain for each~$n\in\Z$ that
\begin{equation}\label{eq:specialpolvanish}
0 = \left( e^{2 \pi \iu nx} P_n(x) \right)\!\big|_{x=0}
\end{equation}
by Lemma~\ref{fourier_observation}. Comparing~\eqref{eq:sd_first} 
to~\eqref{eq:systemdiffeq} and using~\eqref{eq:specialpolvanish} shows that for 
all $n \in \Z$ and $y \in \R_{>0}$ we have 
\[ 
\Big[s(1-s)  \left. e^{2 \pi \iu  n x}  J^x \right|_{x=0} + y^2  \Big( \left. 
\partial^2_{x}\right|_{x=0}  (e^{2 \pi \iu  n x} J^x) +  \left.e^{2 \pi \iu  n 
x} J^x \right|_{x=0} \partial^2_{y}  \Big)\Big]\hat f_n(y) =0\,.
\]  
Now $\left. e^{2 \pi \iu  n x} J^x \right|_{x=0} = \id$, the identity matrix, 
which completes the proof.  
\end{proof} 

The system of the differential equations for the Fourier coefficients in 
Lemma~\ref{descent_to_a_system} involves an upper triangular matrix.  
Consequently, this system is quite large. In the remainder of this section we 
will deduce an equivalent system of differential equations by means of a base 
change such that the upper triangular matrix is transformed into a three-band 
matrix. 
To obtain this simpler system we require the following two technical lemmas 
regarding the Stirling matrix~$\stirmat$, which was defined 
in~\eqref{def:stirlingmatrix}.

\begin{lemma}\label{propSt1}
Let $\stirmat$ be the Stirling matrix from~\eqref{def:stirlingmatrix}. For each $x\in\R$ we have
\[
\stirmat \begin{pmatrix} x^{d-1} \\[1mm] x^{d-2} \\ \vdots \\ x \\[1mm] 1 
\end{pmatrix} 
= \begin{pmatrix} \binom{x}{d-1} \\[1mm] \binom{x}{d-2} \\ \vdots \\ 
\binom{x}{1} \\[1mm] \binom{x}{0} \end{pmatrix}\,.
\]
\end{lemma} 

\begin{proof} 
Using the very definition of the Stirling matrix~$\stirmat = (\stirelem_{ij})$, we see that the entry in the~$i^{th}$~row of the product on the left hand side of the claimed formula is 
\[
\sum_{j=1}^d \stirelem_{ij}x^{d-j} = \sum_{j=i} ^d \frac{s(d-i, d-j)}{(d-i)!} x^{d-j} = \sum_{k=0} ^{d-i} \frac{s(d-i,k)x^k}{(d-i)!}\,.
\]
By the characterization of the Stirling numbers in~\eqref{eq:charstirling} and the very definition of the generalized binomial coefficients we have
\[  
\sum_{k=0} ^{d-i} \frac{ s(d-i, k) x^k}{(d-i)!} = \frac{1}{(d-i)!} 
\prod_{\ell=0} ^{d-i-1} (x-\ell) = \binom{x}{d-i}\,. 
\]
This completes the proof.
\end{proof} 

A simplification of the system of differential equations from 
Lemma~\ref{descent_to_a_system} will be obtained by conjugating the 
matrix~$J^x$ with the Stirling matrix~$\stirmat$. We provide an explicit expression of the matrix~$\stirmat^{-1} J^x \stirmat$ in the following lemma. For each~$x\in\R$ we let $C(x)=(c_{ij}(x))$ be the $(d\times d)$-matrix with entries 
\begin{equation}\label{eq:def_Cmatrix}
c_{ij} (x) \coloneqq 
\begin{cases}
{d-i \choose j-i} x^{j-i} \quad &\text{if $i\leq j$\,,}
\\
0 & \text{if $i>j$\,.}
\end{cases}
\end{equation}

\begin{lemma}\label{Cmatrix} 
For each~$x\in\R$, we have $\stirmat^{-1}J^x \stirmat = \lambda^x C(x)$.
\end{lemma}

\begin{proof} Since $J^x=\lambda^x J(1)^x$, it suffices to show that 
$\stirmat^{-1}J(1)^x \stirmat = C(x)$ for all~$x\in\R$. To that end we let $x\in\R$ and set 
\[
M\coloneqq M(x)\coloneqq J(1)^x\stirmat\quad\text{and}\quad N \coloneqq N(x)\coloneqq \stirmat C(x)\,.
\]
As usual we denote the entries of all matrices by corresponding small letters, and we omit their dependence on~$x$ to simplify notation. Thus, $M=(m_{ij})$, $N=(n_{ij})$ and $C=C(x) = (c_{ij})$. Since $\stirmat$, $J(1)$ and~$C$ are upper triangular matrices, so are $M$ and~$N$. To show that $M$ and $N$ coincide on the upper triangle, we let $i,j\in\{1,\ldots, d\}$ with $i\leq j$. Then 
\begin{align*}
n_{ij} & = \sum_{k=1}^d \stirelem_{ik}c_{kj} = \sum_{k=i}^j \stirelem_{ik}c_{kj}
\\
& = \sum_{k=i}^j \frac{s(d-i,d-k)}{(d-i)!} \binom{d-k}{j-k} x^{j-k}
\\
& = \sum_{q=0}^{j-i} x^q \frac{s(d-i,d-j+q)}{(d-i)!} \binom{d-j+q}{q}\,,
\end{align*}
and
\begin{align*}
m_{ij} & = \sum_{k=i}^j \frac{s(d-k,d-j)}{(d-k)!}\binom{x}{k-i} 
\\
& \stackrel{(*)}{=} \sum_{k=i}^j \frac{s(d-k,d-j)}{(d-k)!} 
\frac{1}{(k-i)!}\sum_{\ell=0}^{k-i} s(k-i,\ell)x^\ell
\\
& = \sum_{q=0}^{j-i} x^q \sum_{k=i+q}^j \frac{s(d-k,d-j) s(k-i,q)}{(d-k)!(k-i)!}
\\
& = \sum_{q=0}^{j-i} x^q \sum_{\ell = d-j}^{d-i-q} \frac{s(\ell, 
d-j)s(d-i-\ell,q)}{\ell! (d-i-\ell)!}\,,
\end{align*}
where we took advantage of Lemma~\ref{propSt1} for the step~$(*)$. It remains 
to show that for all $i,j\nobreak\in\nobreak\{1,\ldots, d\}$ with $i\leq j$ and all $q\in\{0,\ldots, j-i\}$ we have
\begin{equation}\label{preredform} 
\frac{s(d-i,d-j+q)}{(d-i)!} \binom{d-j+q}{q} =  \sum_{\ell=d-j} ^{d-i-q} 
\frac{s(\ell, d-j) s(d-i-\ell, q)}{\ell! (d-i-\ell)!}\,.
\end{equation}  
If we apply the substitutions $a = d-i$ and $b = d-j$, then \eqref{preredform} 
becomes
\[
\frac{s(a,b+q)}{a!} \binom{b+q}{q} = \sum_{\ell=b}^{a-q} 
\frac{s(\ell,b)s(a-\ell,q)}{\ell!(a-\ell)!}\,,
\]
which is valid by~\cite[ 24.1.3 II.A]{AS2008}. This completes the proof.
\end{proof}

With these preparations we can now show that the system of differential 
equations from Lemma~\ref{descent_to_a_system} can be reformulated as a system of differential equations in which the interactions of the single differential equations are determined by a certain three-band matrix and hence are simpler than in the system from Lemma~\ref{descent_to_a_system}. For $\alpha\in\C$ we let $H(\alpha) = (h_{ij}(\alpha))$ denote the $(d\times d)$-matrix with entries
\begin{equation}\label{eq:def_H}
 h_{ij} (\alpha) \coloneqq -2\alpha \binom{d-i}{1} \delta_0(j-i-1) - 2 
\binom{d-i}{2} \delta_0(j-i-2)
\end{equation}
for $i,j\in\{1,\ldots, d\}$, where $\delta_0$ denotes the Kronecker delta 
function~\eqref{eq:kronecker}. 
The matrix~$H(\alpha)$ roughly looks like
\begin{align*}
 & H(\alpha) = \left( h_{ij}(\alpha) \right)_{i,j=1}^d
 \\[1mm]
 & = 
 {\small 
 \begin{pmatrix}
  0 & -2\alpha(d-1) & -(d-1)(d-2) & 0 & 0 &  \ldots & 0
  \\
  0 & 0 & -2\alpha(d-2) & -(d-2)(d-3) & 0 & \ldots & 0
  \\
  \vdots & \vdots & \vdots  &  & \vdots &  & \vdots
  \\
  0 & 0  & 0  & & 0  & -2\alpha\cdot 2 & -2
  \\
  0 & 0 & 0 & \ldots  &  & 0 & -2\alpha 
  \\
  0 & 0 & 0 & \ldots  & & & 0
 \end{pmatrix}\,.
 }
\end{align*}
As we will see in Proposition~\ref{simplified_system}, the required three-band matrix is a linear combination of the identity matrix with~$H(\alpha)$ for a specific value of~$\alpha$. 

We recall that $\stirmat$ denotes the Stirling matrix, defined 
in~\eqref{def:stirlingmatrix}, and that $\log(\;\cdot\;;\omega)$ is our choice from Section~\ref{sec:cut} for a complex logarithm that is holomorphic in the eigenvalue~$\lambda$ of the twisting endomorphism~$A$ (which is here 
represented by the Jordan-like matrix~$J$). Further we recall 
from~\eqref{eq:def_alpha_n_lambda} that 
\[
 \alpha_n = 2 \pi n - \iu \log(\lambda;\omega)
\]
for all~$n\in\Z$.

\begin{prop}\label{simplified_system}
For each $n\in\Z$, the Fourier coefficient function~$\hat f_n$ satisfies the 
system of second order differential equations
\begin{equation}\label{simplified_system_eq}
\bigr(y^2 \partial^2_y + s(1-s)  \bigl) \stirmat^{-1}  \hat f_n(y)  =  y^2  
\bigr(\alpha_n^2 + H(\iu \alpha_n)\bigl) \stirmat^{-1}   \hat f_n(y)
\end{equation}
on~$\R_{>0}$. Even stronger, this system is equivalent to the 
system~\eqref{d_entry_matrix} in Lemma~\ref{descent_to_a_system}.
\end{prop}

\begin{proof}
Multiplying the system of differential equations~\eqref{d_entry_matrix} with 
$\stirmat^{-1}$ from the left, we get 
\begin{align*}
\bigl( y^2 \partial^2_{y} + s(1-s) \bigr) \stirmat^{-1} \hat f_n (y)  
& = - y^2  \stirmat^{-1} \bigl( \left.\partial^2_{x} \right|_{x = 0} (e^{2 \pi \iu  n x } J^x)  \bigr) \hat f_n(y)  \\
& = - y^2  \bigl( \left.\partial^2_{x} \right|_{x = 0} ( e^{2 \pi \iu  n x } \stirmat^{-1} J^x \stirmat)  \bigr) \stirmat^{-1}  \hat f_n (y)\,.
\end{align*}
With the matrix~$C(x)$ from~\eqref{eq:def_Cmatrix} and the identity $\lambda^x C(x) = \stirmat^{-1}J^x\stirmat$ shown in Lemma~\ref{Cmatrix} we obtain
\begin{align*}
& \left. \partial^2_{x} \right|_{x = 0}    (e^{2 \pi \iu  n x} \stirmat^{-1} J^x \stirmat) 
\\
& \qquad = \left. \partial^2_{x} \right|_{x = 0} (e^{2 \pi \iu  n x} \lambda^x C(x))
\\
& \qquad =  \left. \left(  \partial^2_{x} (e^{\iu \alpha_n x})\cdot  C(x) + 2 
\pa_x e^{\iu \alpha_n x}\cdot \partial_{x} C(x)  + e^{\iu \alpha_n x} \pa_x ^2  C(x) \right)\right|_{x=0}
\\
& \qquad = \left( -\alpha_n^2 \delta_0(i-j) + 2 \iu \alpha_n 
\binom{d-i}{1}\delta_0(j-i-1) \right.
\\
& \hphantom{ \qquad = \Big( -\alpha_n^2  \delta_0(i-j) + 2 \iu \alpha_n }  \left. +\ 2\binom{d-i}{2}\delta_0(j-i-2)\right)_{i,j=1}^d
\\
& \qquad = -\alpha_n^2\id - H(\iu \alpha_n)\,.
\end{align*}
Thus, the system~\eqref{d_entry_matrix} is equivalent to the system~\eqref{simplified_system_eq}.
\end{proof}

\section{Solution of the system of differential equations in Proposition~\ref{simplified_system} if $n=0$ and $\alpha_0=0$}
\label{sec:proof_zero}

We continue to suppose that~$V$ is a complex vector space of dimension~$d\in\N$, and that~$A$ is a linear isomorphism of~$V$ that acts irreducibly.
We fix a basis with respect to which~$A$ is represented by the Jordan-like block matrix $J=J(\lambda)$ with $\lambda\in\C$, $\lambda\not=0$, being the eigenvalue of~$A$, and we identify $V$ with~$\C^d$.
We fix a complex logarithm~$\log(\;\cdot\;;\omega)$ by choosing~$\omega\in\C$ with the properties as in Section~\ref{sec:cut}.
We let $f\colon\h\to V\cong\C^d$ be a $J$-twist periodic Laplace eigenfunction with spectral parameter~$s$ and consider its Fourier expansion
\[
 f(z) = \sum_{n\in\Z} J^x \hat f_n(y) e^{2\pi\iu nx}\qquad (z=x+\iu y\in\h)\,.
\]
In Proposition~\ref{simplified_system} we have shown that the Fourier 
coefficient function~$\hat f_n$, for any $n\in\Z$, satisfies a certain system of second order differential equations.
We will see that the solutions of this system of differential equations are based on the functions~$\fu$ and~$\futoo$, defined in 
Sections~\ref{sec:IandKpart1} and~\ref{sec:IandKpart2}.
The second variable, $\alpha$, of these functions encodes a certain combination of the considered $n$ and the considered eigenvalue~$\lambda$ of~$J$.
As $\fu$ and~$\futoo$ are qualitatively different depending on whether $\alpha$ equals~$0$ or not, the discussion of the solution space of the differential equations naturally splits into these cases. 

In this section, we will provide (see Proposition~\ref{prop:basis_sol_zero} 
below) the space of solutions of the system of differential equations for the 
case $n=0$ and $\lambda=1$, i.e., $\alpha_n=0$ in the notation of 
Proposition~\ref{simplified_system}. We recall the Stirling matrix~$\stirmat$ 
from~\eqref{def:stirlingmatrix}, the two-band matrix~$H(0)$ 
from~\eqref{eq:def_H} and set
\[
 w(y) \coloneqq \stirmat^{-1}\hat f_0(y)\,.
\]
Then the system of differential equations becomes
\begin{equation}\label{eq:masterdiff_zero}
\left( y^2\partial_y^2 + s(1-s)\right) w(y) = y^2 
H(0)w(y)\,,\quad y\in\R_{>0}\,.
\end{equation}
We will first provide two particular solutions of~\eqref{eq:masterdiff_zero} and then construct from these all solutions.
We note that this latter step is more involved than the standard argumentation with linear combinations as the space of solutions of~\eqref{eq:masterdiff_zero} is~$2d$-dimensional, but we start 
with only two particular solutions.
Proposition~\ref{prop:basis_sol_zero} immediately establishes  Theorem~\ref{Fourier:main_theorem} in the case~$n=0$ and~$\lambda=1$. 
We consider $s\in\C$ to be fixed throughout and recall from~\eqref{eq:def_H} 
that 
\begin{align*}
 H(0) & = \left( -2\binom{d-i}{2}\delta_0(j-i-2) \right)_{i,j=1}^d 
 \\
 & = \Bigl( -(d-i)(d-i-1)\delta_0(j-i-2) \Bigr)_{i,j=1}^d
 \\
 & = 
 \begin{pmatrix}
  0 & 0 & -(d-1)(d-2) & 0 & 0 &  \ldots & 0
  \\
  0 & 0 & 0 & -(d-2)(d-3) & 0 & \ldots & 0
  \\
  \vdots & \vdots & \vdots  &  & \vdots &  & \vdots
  \\
  0 & 0  & 0  & & 0  & 0 & -2
  \\
  0 & 0 & 0 & \ldots  &  & 0 & 0 
  \\
  0 & 0 & 0 & \ldots  & & & 0
 \end{pmatrix}\,.
\end{align*}

\subsection{Particular solutions}
In this section we provide two solutions of the differential 
equation~\eqref{eq:masterdiff_zero}. 

\begin{prop}\label{prop:zero_nonexc}
Let $\fuPsi\in\{\fu,\futoo\}$. Then the map $w\colon\R_{>0}\to V$, 
\[
 w \coloneqq 
 \begin{pmatrix}
  \fuPsi(d-1,0,\cdot,s)
  \\
  \vdots
  \\
  \fuPsi(1,0,\cdot,s)
  \\
  \fuPsi(0,0,\cdot,s)
 \end{pmatrix}
\]
is a solution of the second order differential equation
\[
 \left( y^2\partial_y^2 + s(1-s) \right)w(y) = y^2 H(0) w(y)\,,\quad 
y\in\R_{>0}\,.
\]
\end{prop}

Preparatory for the proof of Proposition~\ref{prop:zero_nonexc} we present a few recursive identities between the combinatorial coefficients of the functions~$\fu$ and~$\futoo$.
We recall the set 
\[
 E(d) = \left\{ \frac12\pm j\setmid 
j=0,1,\ldots,\left\lfloor\frac{d-1}{2}\right\rfloor \right\}
\]
of exceptional spectral parameters from~\eqref{s-specialvalues} and the 
coefficient function~$\cone$ from~\eqref{def_cof_conc}.
We will require the following recursive identity for~$\cone$.

\begin{lemma}\label{lem:recur_cone}
Let $m\in\N$, $m\geq 2$, and $s\in\C$, $s\notin E(d)\cap 
\left[\tfrac12-\tfrac{m}{2},-\tfrac12\right]$.
Then we have
\[
2 \left\lfloor\frac{m}{2}\right\rfloor \left( 2s-1 + 
2\left\lfloor\frac{m}{2}\right\rfloor \right) \cone(m,s) = m(1-m)\cone(m-2,s)\,.
\]
\end{lemma}

\begin{proof}
Starting at the definition of~$\cone$ in~\eqref{def_cof_conc}, we find
\begin{align*}
\cone(m,s) & = 
(-1)^{\left\lfloor\frac{m}{2}\right\rfloor}\,\frac{m!}{2^{\left\lfloor\frac{m}{2
}\right\rfloor} \left\lfloor\frac{m}{2}\right\rfloor!\, 
\prod\limits_{j=1}^{\left\lfloor\frac{m}{2}\right\rfloor} (2j+2s-1)}
\\
& = 
(-1)^{\left\lfloor\frac{m-2}{2}\right\rfloor}\,\frac{(m-2)!}{2^{
\left\lfloor\frac{m-2}{2}\right\rfloor} 
\left\lfloor\frac{m-2}{2}\right\rfloor!\,\prod\limits_{j=1}^{\left\lfloor\frac{
m-2}{2}\right\rfloor}(2j+2s-1)}
\\
& \qquad\times 
(-1)\,\frac{m(m-1)}{2\,\left\lfloor\frac{m}{2}\right\rfloor\,\left( 
2\,\left\lfloor\frac{m}{2}\right\rfloor + 2s - 1\right)}
\\
& = \frac{m(1-m)}{2\,\left\lfloor\frac{m}{2}\right\rfloor\,\left( 
2\,\left\lfloor\frac{m}{2}\right\rfloor + 2s - 1\right)}\cone(m-2,s)\,.
\end{align*}
This completes the proof.
\end{proof}

For the next set of recursive identities we recall the definitions of~$\ctwo$ 
and~$\cthree$ from~\eqref{eq:ctwo_start1}--\eqref{defctwo} 
and~\eqref{eq:cthree_start1}--\eqref{eq:def_cthree}. 

\begin{lemma}\label{lem:recur_ctwo_cthree}
Let $m\in\N$, $m\geq2$, and $s\in E(d) \cap \left[\tfrac32 - \tfrac{m}{2}, 
\tfrac12 \right]$. Then we have
\[
m(1-m)\,\ctwo(m-3+2s,s) = 
2\,\left\lfloor\frac{m}{2}\right\rfloor\,\left( 
2\left\lfloor\frac{m}{2}\right\rfloor + 2s - 1\right)\,\ctwo(m-1+2s,s)\,,
\]
and 
\begin{align*}
m(1-m)\,\cthree(m-3&+2s,2) 
\\
& = 
\left( 4\,\left\lfloor\frac{m}{2}\right\rfloor + 2s - 1\right)\,\ctwo(m-1+2s,s) 
\\
& \qquad\quad + 2\,\left\lfloor\frac{m}{2}\right\rfloor\,\left( 
2\,\left\lfloor\frac{m}{2}\right\rfloor + 2s - 1\right)\,\cthree(m-1+2s,s)\,.
\end{align*}
\end{lemma}

\begin{proof}
We note that 
\[
 2 \leq m-1+2s \leq m\,.
\]
Thus, using~\eqref{defctwo}, we find
\begin{align*}
\ctwo(m-1+2s,s) & = -\frac14\,\frac{ m(m-1) }{ 
\left\lfloor\frac{m-1+2s}{2}\right\rfloor \left( \frac12 - s + 
\left\lfloor\frac{m-1+2s}{2}\right\rfloor\right)}\,\ctwo(m-3+2s,s)
\\
& = -\frac14\,\frac{m(m-1)}{\left( \left\lfloor\frac{m}{2}\right\rfloor + s - 
\frac12\right)\,\left\lfloor\frac{m}{2}\right\rfloor}\,\ctwo(m-3+2s,s)
\\
& = \frac{m(1-m)}{2\,\left\lfloor\frac{m}{2}\right\rfloor\,\left( 
2\left\lfloor\frac{m}{2}\right\rfloor + 2s -1 \right)}\,\ctwo(m-3+2s,s)\,.
\end{align*}
This proves the first of the claimed identities. For the second identity we 
start at the definition~\eqref{eq:def_cthree} of~$\cthree$ and obtain
\begin{align*}
\cthree(m&-1+2s,s) 
\\
& = - \frac{\frac12 - s + 2\,\left\lfloor\frac{m-1+2s}{2}\right\rfloor}{ 
2\,\left\lfloor\frac{m-1+2s}{2}\right\rfloor\, \left( \frac12 - s + 
\left\lfloor\frac{m-1+2s}{2}\right\rfloor\right)}\,\ctwo(m-1+2s,2) 
\\
& \qquad + \sum_{j=1}^{\lfloor\frac{m-3+2s}{2}\rfloor} 
-(-1)^{\lfloor\frac{m-3+2s}{2}\rfloor+j+1}\,\frac{\frac 1 2 - s +2j}{2j\,\left( 
\frac 
1 2 - s +j \right)}\,\frac{1}{4^{\lfloor\frac{m-3+2s}{2}\rfloor-j}}\,\frac14
\\
&\qquad \hphantom{\coloneqq}\ \times \frac{ 
\prod\limits_{\ell=2j+1+\delta_\odd(m-1+2s)}^{m-1+2s} (1-2s+\ell)}{ 
\prod\limits_{\ell=j+1}^{\lfloor\frac{m-1+2s}{2}\rfloor} \ell\,\left( \frac 1 2 
- s 
+\ell\right)}\, \ctwo(2j+ \delta_\odd(m-1+2s), s)\,,
\end{align*}
with~$\delta_\odd$ as defined in~\eqref{eq:def_deltaodd}. We note that 
\[
\delta_\odd(m-1+2s) = \delta_\odd(m-3+2s)\,,
\]
and that, for any $j\in\N$ with $j\leq 
\left\lfloor\frac{m-3+2s}{2}\right\rfloor$, we have
\begin{align*}
\frac{ \prod\limits_{\ell=2j+1+\delta_\odd(m-1+2s)}^{m-1+2s} (1-2s+\ell)}{ 
\prod\limits_{\ell=j+1}^{\lfloor\frac{m-1+2s}{2}\rfloor} \ell\,\left( \frac 1 2 
- s +\ell\right)} 
 & = \frac{m(m-1)}{\left\lfloor\frac{m-1+2s}{2}\right\rfloor\,\left(\frac12 - s 
+  \left\lfloor\frac{m-1+2s}{2}\right\rfloor\right)}
 \\
 & \qquad \times \frac{ \prod\limits_{\ell=2j+1+\delta_\odd(m-3+2s)}^{m-3+2s} 
(1-2s+\ell)}{ 
\prod\limits_{\ell=j+1}^{\lfloor\frac{m-3+2s}{2}\rfloor} \ell\,\left( \frac 1 2 
- s +\ell\right)}\,.
\end{align*}
Using these identities in the equality above, we obtain
\begin{align*}
\cthree(m&-1+2s,s) 
\\
& = - \frac{4\,\left\lfloor\frac{m}{2}\right\rfloor + 2s - 1}{ 
2\,\left\lfloor\frac{m}{2}\right\rfloor\, \left( 
2\,\left\lfloor\frac{m}{2}\right\rfloor + 2s - 1\right)}\,\ctwo(m-1+2s,2) 
\\
& \qquad + \frac{m(1-m)}{2\,\left\lfloor\frac{m}{2}\right\rfloor\, \left( 
2\,\left\lfloor\frac{m}{2}\right\rfloor + 2s - 1\right) 
}\sum_{j=1}^{\lfloor\frac{m-3+2s}{2}\rfloor} 
(-1)^{\lfloor\frac{m-3+2s}{2}\rfloor+j+1}
\\
&\qquad \hphantom{\coloneqq}\ \times \frac{\frac 1 2 - s +2j}{2j\,\left( 
\frac 
1 2 - s +j \right)}\,\frac{1}{4^{\lfloor\frac{m-3+2s}{2}\rfloor-j}}
\\
&\qquad \hphantom{\coloneqq}\ \times \frac{ 
\prod\limits_{\ell=2j+1+\delta_\odd(m-3+2s)}^{m-3+2s} (1-2s+\ell)}{ 
\prod\limits_{\ell=j+1}^{\lfloor\frac{m-3+2s}{2}\rfloor} \ell\,\left( \frac 1 2 
- s 
+\ell\right)}\, \ctwo(2j+ \delta_\odd(m-3+2s), s)\,.
\end{align*}
The summand for $j=\left\lfloor\frac{m-3+2s}{2}\right\rfloor$ in the ``large'' 
sum is 
\[
 (-1)\, \frac{\frac12 - s + 2\,\left\lfloor\frac{m-3+2s}{2}\right\rfloor}{ 
2\,\left\lfloor\frac{m-3+2s}{2}\right\rfloor\, \left( \frac12 - s + 
\left\lfloor\frac{m-3+2s}{2}\right\rfloor\right)}\, \ctwo(m-3+2s,s)\,,
\]
since, as one easily checks,
\[
 2\,\left\lfloor\frac{m-3+2s}{2}\right\rfloor + 1 + \delta_\odd(m-3+2s) = 
m-2+2s\,. 
\]
It now follows that 
\begin{align*}
 \cthree(m-1+2s,s) & = - \frac{4\,\left\lfloor\frac{m}{2}\right\rfloor + 2s - 
1}{ 2\,\left\lfloor\frac{m}{2}\right\rfloor\, \left( 
2\,\left\lfloor\frac{m}{2}\right\rfloor + 2s - 1\right)}\,\ctwo(m-1+2s,2) 
\\
& \quad + \frac{m(1-m)}{2\,\left\lfloor\frac{m}{2}\right\rfloor\, \left( 
2\,\left\lfloor\frac{m}{2}\right\rfloor + 2s - 1\right) }\,\cthree(m-3+2s,s)\,,
\end{align*}
which establishes the second of the claimed identities.
\end{proof}

With these preparations we can now present a proof of 
Proposition~\ref{prop:zero_nonexc}.

\begin{proof}[Proof of Proposition~\ref{prop:zero_nonexc}]
Using Definition~\ref{defazero},  we note that for all $s\in\C$, 
$s\not=\frac12$, and all $m\in\N_0$, $y\in\R_{>0}$, we have
\[
 \fu(m,0,y,s) = \futoo(m,0,y,1-s)\,.
\]
Therefore it suffices to prove the claimed statement for the two cases that 
$\fuPsi=\fu$ and $s\in\C$ arbitrary as well as that $\fuPsi=\futoo$ and 
$s=\frac12$. Equivalently, it suffices to establish the statement for the cases
\begin{enumerate}[label=$\mathrm{(\alph*)}$, ref=$\mathrm{\alph*}$]
 \item\label{eq:Psi_1} $s\in\C$, $s\notin E(d)\cap 
\left[\frac12-\frac{d-1}{2},-\frac12\right]$ and for all $m\in\{0,\ldots, 
d-1\}$,
 \[
  \fuPsi(m,0,y,s) = \cone(m,s)\,y^{s+2\,\lfloor\frac{m}{2}\rfloor}\,,
 \]
\item\label{eq:Psi_2} $s=\frac12-\frac{k}{2}$ for a (unique) $k\in\{0,\ldots, 
d-1\}\cap 2\N_0$ (thus, $s\in E(d)$) and 
\begin{align*}
 \fuPsi(m,0,y,s) =
 \begin{cases}
\ctwo(m-1+2s,s)\,y^{s+2\,\lfloor\frac{m}{2}\rfloor}\,\log y &\ \text{if $k\leq 
m\leq d-1$}\,,
\\
\quad + \cthree(m-1+2s,s)\,y^{s+2\,\lfloor\frac{m}{2}\rfloor} 
\\[3mm]
\cone(m,s)\,y^{s+2\,\lfloor\frac{m}{2}\rfloor} &\ \text{if $0\leq m < k$}\,.
 \end{cases}
\end{align*}
\end{enumerate}
In both cases, establishing the claimed statement is equivalent to proving that 
\begin{equation}\label{eq:ODE_single}
 \left(y^2\,\partial_y^2 + s(1-s)\right) \fuPsi(m,0,y,s) = 
m(1-m)\,y^2\,\fuPsi(m-2,0,y,s)
\end{equation}
for all $m\in\{0,\ldots, d-1\}$.

We consider the case~\eqref{eq:Psi_1}. For each $m\in\{0,\ldots, d-1\}$, a 
straightforward calculation shows 
\begin{align*}
 \left(y^2\,\partial_y^2 + 
s(1-s)\right)\cone(m,s)\, & y^{s+2\,\lfloor\frac{m}{2}\rfloor} 
\\
& = 
2\,\left\lfloor\frac{m}{2}\right\rfloor\left(2s-1+2\,\left\lfloor\frac{m}{2}
\right\rfloor\right)\,\cone(m,s)\,y^{s+2\,\lfloor\frac{m}{2}\rfloor}\,.
\end{align*}
For $m\in\{0,1\}$, the right hand side of this equation obviously vanishes. For 
$m\geq 2$ it equals 
\[
 m(1-m)\cone(m-2,s)\,y^{s+2\,\lfloor\frac{m}{2}\rfloor}\,,
\]
by Lemma~\ref{lem:recur_cone}. This shows~\eqref{eq:ODE_single} for 
case~\eqref{eq:Psi_1}. 

We now discuss the case~\eqref{eq:Psi_2}. For $m\in\{0,\ldots, k-1\}$, the 
definition of $\fuPsi$ is as in~\eqref{eq:Psi_1}, and also the sought-for 
equality in~\eqref{eq:ODE_single} reads as in~\eqref{eq:Psi_1} and hence follows 
as above. For $m=k+p$ with $p\in\{0,1\}$, we have 
\[
 m - 1 + 2s = p
\]
and 
\[
 2s-1-\left\lfloor\frac{m}{2}\right\rfloor = 0\,.
\]
It follows that 
\begin{align*}
 \big(y^2\,\partial_y^2 & + s(1-s)\big) \fuPsi(m,0,y,s)
 \\
 & = \left(y^2\,\partial_y^2 + 
s(1-s)\right)\ctwo(p,s)\,y^{s+2\,\lfloor\frac{m}{2}\rfloor}\,\log y
 \\
 & = \ctwo(p,s)\,\Big( 2\,\left\lfloor\frac{m}{2}\right\rfloor\,\left(2s-1 - 
2\,\left\lfloor\frac{m}{2}\right\rfloor\right)\,y^{s+2\lfloor\frac{m}{2}\rfloor}
\log y
 \\
 & \qquad\hphantom{\ctwo(p,s)\,\Big( 
2\,\left\lfloor\frac{m}{2}\right\rfloor\,y^{s+2\lfloor\frac{m}{2}\rfloor}} + 
\left( 2s - 1 +  
4\,\left\lfloor\frac{m}{2}\right\rfloor\right)\,y^{s+2\lfloor\frac{m}{2}\rfloor}
\Big)
 \\
 & = \ctwo(p,s)\left( 2s - 1 +  
4\,\left\lfloor\frac{m}{2}\right\rfloor\right)\,y^{s+2\lfloor\frac{m}{2}\rfloor}
 \\
 & = m(1-m)\cone(m-2,s)\,y^{s+2\lfloor\frac{m}{2}\rfloor}
 \\
 & = m(1-m)\,y^2\,\fuPsi(m-2,0,y,s)\,,
\end{align*}
having used~\eqref{eq:ctwo_start2} in the penultimate equality. Finally, for 
$m\geq k+2$, we obtain 
\begin{align*}
 \big(y^2\,\partial_y^2 & + s(1-s)\big) \fuPsi(m,0,y,s)
 \\
 & = \ctwo(m-1+2s,s)\,2\,\left\lfloor\frac{m}{2}\right\rfloor\left(2s-1+ 
2\,\left\lfloor\frac{m}{2}\right\rfloor\right)\, 
y^{s+2\,\lfloor\frac{m}{2}\rfloor}\,\log y 
 \\
 & \qquad + \ctwo(m-1+2s,s)\left( 
2s-1+4\,\left\lfloor\frac{m}{2}\right\rfloor\right)\,y^{s+2\,\lfloor\frac{m}{2}
\rfloor}
 \\
 & \qquad + 
\cthree(m-1+2s,s)\,2\,\left\lfloor\frac{m}{2}\right\rfloor\left(2s-1+ 
2\,\left\lfloor\frac{m}{2}\right\rfloor\right) 
\,y^{s+2\,\lfloor\frac{m}{2}\rfloor}\,.
\end{align*}
An application of Lemma~\ref{lem:recur_ctwo_cthree} finishes the proof.
\end{proof}

\subsection{Space of all solutions}\label{sec:all_zero}

In this section we construct from the particular solutions in the previous 
section the $2d$-dimensional space of all solutions of the differential 
equation~\eqref{eq:masterdiff_zero}. See Proposition~\ref{prop:basis_sol_zero}. 
This establishes  Theorem~\ref{Fourier:main_theorem} for the case~$n=0$ and 
$\lambda=1$ or, equivalently, $\alpha_n=0$. We start by showing an independence 
result (Lemma~\ref{lem:IK_zero_indep}) which will be crucial for the proof of 
Proposition~\ref{prop:basis_sol_zero}.

\begin{lemma}\label{lem:IK_zero_indep}
For all~$m\in\N_0$ and all~$s\in\C$, the two functions~$\fu(m,0,\cdot,s)$ and 
$\futoo(m,0,\cdot,s)$ are linearly independent (over~$\C$). 
\end{lemma}

\begin{proof}
We need to distinguish three cases, depending on the value of~$s$. If 
\[
 s\in E(d)\cap \left[ \frac12-\frac{m}2, \frac12\right]\,,
\]
then $\fu(m,0,y,s)$ is a certain linear combination with nonzero coefficients of 
the functions $y^{s+2\left\lfloor\frac{m}{2}\right\rfloor}\log y$ and 
$y^{s+2\left\lfloor\frac{m}{2}\right\rfloor}$ and $\futoo(m,0,y,s)$ is a nonzero 
 scaling of~$y^{1-s+2\left\lfloor\frac{m}{2}\right\rfloor}$. See 
Section~\ref{sec:IandKpart2}. Therefore, in this case, the 
functions~$\fu(m,0,\cdot,s)$ and~$\futoo(m,0,\cdot,s)$ are linearly independent 
as can be concluded from their different growth behaviors. If 
\[
 s\in E(d)\cap \left[ \frac32, \frac12+\frac{m}2 \right]\,,
\]
then we have an analogous situation (with the roles of~$\fu$ and~$\futoo$ 
interchanged). Thus, also in this case, $\fu(m,0,\cdot,s)$ 
and~$\futoo(m,0,\cdot, s)$ are linearly independent. For all other values 
of~$s$, the function~$\fu(m,0,\cdot,s)$ is a nonzero scaling 
of~$y^{s+2\left\lfloor\frac{m}{2}\right\rfloor}$ and $\futoo(m,0,\cdot,s)$ is a 
nonzero scaling of~$y^{1-s+2\left\lfloor\frac{m}{2}\right\rfloor}$. Comparing 
the growth behaviors of these functions, we deduce that they are linearly 
dependent if and only if the two exponents are equal, thus, if $s=1/2$. However, 
 $1/2$ is not among the values $s$ may assume. Thus, also for this range of 
values for~$s$, the functions~$\fu(m,0,\cdot,s)$ and~$\futoo(m,0,\cdot,s)$ are 
linearly independent.
\end{proof}

\begin{prop}\label{prop:basis_sol_zero}
Let $s\in\C$ and define the maps $\PI,\PK\colon\R_{>0}\to V$ by
\[
\PI \coloneqq  
\begin{pmatrix}
\fu(d-1,0,\cdot,s) 
\\
\vdots 
\\
\fu(1,0,\cdot,s) 
\\
\fu(0,0,\cdot,s)
\end{pmatrix}
\qquad\text{and}\qquad
\PK \coloneqq   
\begin{pmatrix}
\futoo(d-1,0,\cdot,s) 
\\
\vdots 
\\
\futoo(1,0,\cdot,s) 
\\
\futoo(0,0,\cdot,s)
\end{pmatrix}\,.
\]
Then the $2d$-dimensional space of solutions of the second order differential 
equation
\begin{equation}\label{eq:solODE_zero}
 \bigl( y^2\partial_y^2 + s(1-s)\bigr)w(y) = y^2H(0)w(y)\,,\quad y\in\R_{>0}\,,
\end{equation}
is 
\begin{equation*}\label{eq:solspace_zero}
 L_0 \coloneqq \left\{ C\PI + D\PK \ : \ C,D \in \Coeff_{0;d} \right\}\,,
\end{equation*}
with $\Coeff_{0;d}$ as defined in Section~\ref{sec:coeffmatrices}.
\end{prop}

\begin{proof}
The maps~$\PI$ and~$\PK$ are solutions of~\eqref{eq:solODE_zero} by 
Proposition~\ref{prop:zero_nonexc}. Further, each matrix in~$\Coeff_{0;d}$ 
commutes with~$H(0)$ (and indeed, $\Coeff_{0;d}$ is the full centralizer 
of~$H(0)$) as can be shown by a straightforward, but tedious calculation.  
Therefore, each element in~$L_0$ is indeed a solution 
of~\eqref{eq:solODE_zero}. 
In what follows we will show that these are all solutions. As $L_0$ is a space 
of solutions of a $d$-dimensional system of second order linear differential 
equations, the (complex) vector space dimension of~$L_0$ is at most~$2d$. We 
will now find $2d$ linearly independent elements in~$L_0$, which then 
immediately implies that $L_0$ is $2d$-dimensional and hence the full space of 
solutions of~\eqref{eq:solODE_zero}. To that end we consider the $d$ elements 
$C_1,\ldots, C_d\in\Coeff_{0;d}$, where $C_j = (c^{(j)}_{mn})$ for 
$j\in\{1,\ldots, d\}$ is the matrix with 
\[
 c^{(j)}_{jd} = 1
\]
and as many other entries as possible equal to~$0$. The definition 
of~$\Coeff_{0;d}$ in Section~\ref{sec:coeffmatrices} shows that in particular 
\begin{equation}\label{eq:entryvan1}
 c^{(j)}_{md} = 0 \quad\text{for $m>j$}
\end{equation}
as well as
\begin{equation}\label{eq:entryvan2}
 c^{(j)}_{jn} = 0 \quad\text{for $n\not=d$\,.}
\end{equation}
We now show that the $2d$ elements 
\begin{equation}\label{eq:elem_nonzero}
 C_1\PI\,,\ \ldots\,,\ C_d\PI\,,\ C_1\PK\,,\ \ldots\,,\ C_d\PK
\end{equation}
of~$L_0$ are linearly independent (over~$\C$) by showing that any vanishing 
linear combination 
\begin{equation}\label{eq:lincomb_zero}
 \sum_{j=1}^d \mu_j C_j\PI + \nu_j C_j\PK = 0
\end{equation}
with $\mu_j,\nu_j\in\C$ is necessarily trivial (i.e., $\mu_j=0=\nu_j$ for all 
$j\in\{1,\ldots, d\}$). The bottom vector entry of the linear 
combination~\eqref{eq:lincomb_zero} is 
\[
 \mu_d \fu(0,0,\cdot,s) + \nu_d \futoo(0,0,\cdot,s) = 0\,.
\]
None of the other summands contribute to the bottom entry as follows immediately 
from~\eqref{eq:entryvan1} and~\eqref{eq:entryvan2}. Since $\fu(0,0,\cdot,s)$ 
and~$\futoo(0,0,\cdot,s)$ are linearly independent by 
Lemma~\ref{lem:IK_zero_indep}, we obtain 
\[
 \mu_d = 0 = \nu_d\,.
\]
Proceeding inductively implies the triviality of~\eqref{eq:lincomb_zero} and 
completes the proof. 
\end{proof}

\section{Asymptotic behavior of the Fourier coefficient 
functions}\label{sec:growthIK}

In Section~\ref{sec:proof_zero} we have observed that the growth properties of 
the functions~$\fu$ and~$\futoo$ for $\alpha=0$ are important for showing linear 
independence and ultimately for establishing Theorem~\ref{Fourier:main_theorem} 
in the special case of~$\alpha=\alpha_n=0$ (in the notation of 
Theorem~\ref{Fourier:main_theorem}). In Section~\ref{sec:proof_nonzero} we will 
see that also for $\alpha\not=0$ the growth properties of~$\fu$ and~$\futoo$ are 
crucial for an efficient proof of Theorem~\ref{Fourier:main_theorem}, this time 
in the case that $\alpha=\alpha_n\not=0$. 

In this section we prove growth asymptotics for the functions~$\fu$ and~$\futoo$ 
for non-vanishing second variable~$\alpha$, thereby providing a complete and 
refined version of Theorem~\ref{thm:intro_growth}, split into 
Propositions~\ref{prop:asymp_fu} and~\ref{prop:asymp_futoo}. The ``principal 
sector'' in the statement of Theorem~\ref{thm:intro_growth} refers to the case 
that $\arg\alpha\in\Omega_\omega$ in these propositions. Further, in both 
propositions, $\sqrt{2\pi\alpha}$ is evaluated with respect to the principal 
logarithm, thus 
\[
 \frac{1}{\sqrt{2\pi\alpha}} = \exp\left( - \frac 1 2 \log(2\pi \alpha; -1) 
\right)\,.
\]

\begin{prop}\label{prop:asymp_fu}  
Let $s\in\C$, $\alpha \in \C \setminus \omega \R_{\geq 0}$, and $m \in \N_{0}$. 
Depending on the sectorial location of~$\alpha$ we have the following asymptotic 
behavior of~$\fu(m,\alpha, y, s)$ as~$y \to \infty$: 
\begin{align*}
& \fu(m, \alpha, y, s) \sim
\\
& \quad 
\begin{cases}
\frac{y^m}{\iu^m \sqrt{2\pi\alpha}}\,e_+(m,\alpha,y,s)\,, 
& \arg \alpha \in \left( -\frac\pi2,  \pi \right]\cap\Omega_\omega\,, 
\\[2mm]
\frac{y^m}{\iu^m \sqrt{2\pi\alpha}}\, e_-(m,\alpha,y,s)\,,
&
\arg \alpha \in \left( - \pi, -\frac\pi2 \right]\cap\Omega_\omega\,,
\\[2mm]
\frac{y^m}{\iu^m \sqrt{2\pi\alpha}} e^{-2\pi \iu (s-\frac12) s(\omega)} 
\, e_+(m,\alpha,y,s)\,, & \arg \alpha 
\in \left( - \frac\pi2, \pi \right]\setminus\Omega_\omega\,,
\\[2mm]
\frac{y^m}{\iu^m \sqrt{2\pi\alpha}} e^{-2\pi \iu (s-\frac12) s(\omega)}\, e_-(m,\alpha,y,s)\,, 
&
\arg \alpha \in \left( -  \pi, - \frac\pi2 \right]\setminus\Omega_\omega\,,
\end{cases}
\end{align*}
where 
\begin{align*}
e_+(m,\alpha,y,s) & \coloneqq  e^{\alpha y}+ (-1)^m e^{\iu \pi s}e^{-\alpha y}\,,
\\
e_-(m,\alpha,y,s) & \coloneqq e^{\alpha y}+ (-1)^m e^{-\iu \pi s}e^{-\alpha y}
\\
\text{and}\qquad\qquad s(\omega) &\coloneqq \sgn(\arg \omega)\,.
\end{align*}
\end{prop}

\begin{proof}
We recall from~\eqref{eq:def_fu_general} that 
\[ 
\fu (m, \alpha, y, s) = \frac{\exp(\frac 1 2 \log(y; -1))}{\iu^m} 
\partial_\alpha ^m I_{s-\frac12} (\alpha y; \omega)
\] 
and note that
\[ 
\partial_\alpha ^m I_{s-\frac12} (\alpha y; \omega) = \left . y^m \partial_z^m 
I_{s-\frac12} (z; \omega) \right|_{z = \alpha y}\,. 
\] 
(We remark that here $y$ is not the imaginary part of~$z$.) We further have the 
recurrence relation
\begin{equation} \label{eq:inductive_relation_I} 
\partial_z^m I_\eta (z;\omega) = \frac{1}{2^m} \sum_{\ell = 0} ^m {m \choose 
\ell} I_{\eta -m+ 2\ell} (z; \omega)
\end{equation} 
for any~$\eta\in\C$, as can be shown by a straightforward induction argument, 
starting with 
\[
\partial_z I_\eta(z; -1) = \frac 1 2 \left( I_{\eta-1}(z; -1) + I_{\eta+1}(z; 
-1) \right)
\]
(see \cite[8.486.2]{Gradshteyn} or \cite[\S 3.71 (2), p. 79]{Watson44}) and 
using the relation between the functions~$I_\eta(\cdot;-1)$ and~$I_\eta(\cdot;\omega)$ 
from~\eqref{I_prop_def}. We therefore obtain
\begin{equation}\label{sum_fu_asymptotics} 
\fu (m, \alpha, y, s) = \frac{y^{m} \exp(\frac 1 2 \log(y; -1))}{\iu ^m} 
\frac{1}{2^m} \sum_{\ell = 0} ^m {m \choose \ell} I_{s-\frac12 -m+ 2 \ell} 
(\alpha y; \omega)\,. 
\end{equation} 
Combining the classical results on the asymptotic behavior of the modified 
Bessel function~$I$ of the first kind (see, e.g., \cite[8.451.5]{Gradshteyn} and 
\cite[10.30.4, 10.30.5]{DLMF}) with~\eqref{I_prop_def} provides us with the 
following asymptotics for $I_\eta(\alpha y;\omega)$ as $y\to\infty$:
\begin{align*}
I_\eta(\alpha y;\omega)\sim
\begin{cases}
\frac{1}{\sqrt{2\pi\alpha y}}\, e_+\bigl(0,\alpha, y, \eta+\tfrac12\bigr)\,,
& 
\arg\alpha \in\left( -\frac{\pi}{2},\pi\right]\cap\Omega_\omega\,,
\\[2mm]
\frac{1}{\sqrt{2\pi\alpha y}}\, e_-\bigl(0,\alpha, y, \eta+\tfrac12\bigr)\,,
&
\arg\alpha \in\left(-\pi, -\frac{\pi}{2}\right]\cap\Omega_\omega\,,
\\[2mm]
\frac{1}{\sqrt{2\pi\alpha y}} e^{-2\pi\iu y s(\omega)}\, e_+\bigl(0,\alpha, y, \eta+\tfrac12\bigr)\,,
&
\arg\alpha \in\left( -\frac{\pi}{2},\pi\right]\setminus\Omega_\omega\,,
\\[2mm]
\frac{1}{\sqrt{2\pi\alpha y}} e^{-2\pi\iu y s(\omega)}\, e_-\bigl(0,\alpha, y, \eta+\tfrac12\bigr)\,,
&
\arg\alpha \in\left(-\pi, -\frac{\pi}{2}\right]\setminus\Omega_\omega\,,
\end{cases}
\end{align*}
where $\sqrt{2\pi\alpha y}$ is evaluated with respect to the principal 
logarithm. Using these asymptotics in~\eqref{sum_fu_asymptotics} completes the 
proof. 
\end{proof}

\begin{prop}\label{prop:asymp_futoo} 
Let $s\in\C$, $\alpha \in \C \setminus \omega \R_{\geq 0}$, and $m \in \N_{0}$. 
Depending on the sectorial location of~$\alpha$ we have the following asymptotic 
behavior of~$\futoo(m,\alpha, y, s)$ as~$y \to \infty$: 
\begin{enumerate}[label=$\mathrm{(\roman*)}$, ref=$\mathrm{\roman*}$]
\item If $\arg \alpha  \in \Omega_\omega$, then 
\[ 
\futoo (m, \alpha, y, s)  \sim \sqrt{ \frac \pi {2\alpha}} \iu^m\, y^m 
e^{-\alpha y}\,.
\] 
\item If $\arg \alpha \in \left(-\frac{\pi}{2}, 
\pi\right]\setminus\Omega_\omega$, then 
\begin{align*}
\futoo (m, \alpha, y, s)  \sim 
&  - e^{-2\pi \iu s(\omega) s} \sqrt{ \frac \pi {2\alpha}}\iu^m\,y^m 
e^{-\alpha y}
\\
& + \frac{\sqrt{2\pi}}{\iu^{m-1} \sqrt{\alpha}}\, s(\omega) \sin(\pi s) \, 
y^m\left( e^{\alpha y} - (-1)^m e^{- \iu \pi s} e^{-\alpha y} \right)\,.
\end{align*} 
\item If $\arg \alpha  \in \left(-\pi,-\frac\pi2\right]\setminus\Omega_\omega$, 
then 
\begin{align*}
\futoo (m, \alpha, y, s)  \sim 
& - e^{-2\pi \iu s(\omega) s} \sqrt{ \frac \pi {2\alpha}}\iu^m\, y^m 
e^{-\alpha y}
\\
& + \frac{\sqrt{2\pi}}{\iu^{m-1} \sqrt{\alpha}}\, s(\omega) \sin(\pi 
s)\,y^m \left( e^{\alpha y} - (-1)^m e^{\iu \pi s} e^{-\alpha y} \right)\,.
\end{align*}
\end{enumerate}
As in Proposition~\ref{prop:asymp_fu}, we set $s(\omega) = \sgn(\arg\omega)$. 
\end{prop}

\begin{proof}
Analogously as in the proof of Proposition~\ref{prop:asymp_fu} we recall 
from~\eqref{eq:def_futoo_general} that
\[ 
\futoo (m, \alpha, y, s) = \frac{\exp(\frac 1 2 \log(y; -1))}{\iu^m} 
\partial_\alpha ^m K_{s-\frac12} (\alpha y; \omega)
\] 
and note that
\[ 
\partial_\alpha ^m K_{s-\frac12} (\alpha y; \omega) = \left. y^m \partial_z^m 
K_{s-\frac12} (z; \omega) \right|_{z = \alpha y}\,.
\] 
(We caution the reader that here $y$ is not the imaginary part of~$z$.) We 
further have the recurrence relation
\begin{equation} \label{eq:inductive_derivative_K} 
\partial_z^m K_\eta (z;\omega) = \frac{(-1)^m}{2^m} \sum_{\ell=0} ^m {m \choose 
\ell} K_{\eta -m + 2 \ell} (z;\omega)\qquad\text{for all~$\eta\in\C$}\,.
\end{equation} 
This follows from a straightforward induction argument using the relation 
between $K_\eta(\;\cdot\;;-1)$ and~$K_\eta(\;\cdot\;;\omega)$ 
from~\eqref{K_prop_def1}--\eqref{K_prop_def3}, the recurrence relation 
\[ 
\partial_z K_\eta (z; -1) =  - \frac 1 2 \left( K_{\eta-1} (z;-1) + K_{\eta+1} 
(z;-1) \right)
\] 
as given in~\cite[8.486.11]{Gradshteyn} (see also \cite[3.71(2)]{Watson44}) for 
$\arg z \in \Omega_\omega$, together with the recurrence relation 
\[
\partial_z I_\eta(z; -1) = \frac 1 2 \left( I_{\eta-1}(z; -1) + I_{\eta+1}(z; 
-1) \right)
\]
(see \cite[8.486.2]{Gradshteyn} or \cite[\S 3.71 (2), p. 79]{Watson44}) if $\arg 
z\not\in\Omega_\omega$. We hence obtain 
\begin{equation} \label{sum_futoo_asymptotics} 
\futoo (m, \alpha, y, s) = \frac{i^m}{2^m}\, y ^{m}\sqrt{y}
\sum_{\ell = 0} ^m {m \choose \ell} K_{s-\frac 1 2 -m + 2 \ell} (\alpha y; 
\omega)\,,
\end{equation} 
where (here) $\sqrt{y} = \exp\bigl(\tfrac 1 2 \log(y; -1)\bigr)$. 
Combining the classical results regarding the asymptotic behavior of the modified 
Bessel function~$K$ of the second type (see, e.g., \cite[8.451.6]{Gradshteyn} 
and also \cite[p. 202, \S 7.23(1)]{Watson44} and \cite[10.25.3]{DLMF}) 
with~\eqref{K_prop_def1}--\eqref{K_prop_def3} provides us with the following 
asymptotics for $K_\eta(\alpha y;\omega)$ as $y\to\infty$:
\[
K_\eta(\alpha y;\omega) \sim 
 \sqrt\frac{\pi}{2}\,\frac{e^{-\alpha y}}{\sqrt{\alpha y}}\qquad\text{if 
$\arg\alpha\in\Omega_\omega$\,,}
\] 
and, if $\arg\alpha\in \left(-\frac\pi2,\pi\right]\setminus\Omega_\omega$, 
\begin{align*}
K_\eta(\alpha y;\omega) \sim \sqrt\frac\pi2 & 
e^{-2\pi\iu\eta s(\omega)}\,\frac{e^{-\alpha y}}{\sqrt{\alpha y}}
\\
& + 2\pi\iu s(\omega)\cos(\pi\eta)\,\frac{1}{\sqrt{2\pi\alpha y}}\left( 
e^{\alpha y} + e^{\left(-\eta+\frac12\right)\iu\pi} e^{-\alpha y}\right)\,,
\end{align*}
and, if $\arg\alpha\in \left(-\pi,-\frac\pi2\right]\setminus\Omega_\omega$, 
\begin{align*}
K_\eta(\alpha y;\omega) \sim \sqrt\frac\pi2 & 
e^{-2\pi\iu\eta s(\omega)}\,\frac{e^{-\alpha y}}{\sqrt{\alpha y}}
\\
& + 2\pi\iu s(\omega)\cos(\pi\eta)\,\frac{1}{\sqrt{2\pi\alpha y}}\left( 
e^{\alpha y} + e^{-\left(-\eta+\frac12\right)\iu\pi} e^{-\alpha y}\right)\,.
\end{align*}
In each of these asymptotics, $\sqrt{\alpha y}$ is evaluated with respect to the 
principal logarithm. Combing these asymptotics 
with~\eqref{sum_futoo_asymptotics} completes the proof. 
\end{proof}

\section{Solution of the system of differential equations in Proposition~\ref{simplified_system} if $n\not=0$ or $\alpha_0\not=0$}
\label{sec:proof_nonzero}

In this section we return to the situation that~$V$ is a complex vector 
space of dimension~$d\in\N$ and that~$A$ is a linear isomorphism of~$V$ that 
acts irreducibly. As before, we fix a basis with respect to which~$A$ is 
represented by the Jordan-like block matrix $J=J(\lambda)$ with $\lambda\in\C$, 
$\lambda\not=0$, being the eigenvalue of~$A$, and we identify $V$ with~$\C^d$. 
We again fix a complex logarithm~$\log(\;\cdot\;;\omega)$ by 
choosing~$\omega\in\C$ with the properties as in Section~\ref{sec:cut}. We let 
$f\colon\h\to V\cong\C^d$ be a $J$-twist periodic Laplace eigenfunction with 
spectral parameter~$s$ and consider its Fourier expansion
\[
 f(z) = \sum_{n\in\Z} J^x \hat f_n(y) e^{2\pi\iu nx}\qquad (z=x+\iu y\in\h)\,.
\]
In Proposition~\ref{simplified_system} we have shown that the Fourier 
coefficient function~$\hat f_n$, $n\in\Z$, is a solution of the differential 
equation 
\begin{equation}\label{eq:diff_f_part}
 \left( y^2\partial_y^2 + s(1-s) - \alpha_n^2y^2\right) \stirmat^{-1}\hat f_n(y) = y^2 
H(\iu\alpha_n)\stirmat^{-1}\hat f_n(y)\,,
\end{equation}
where $y\in\R_{>0}$, 
\[
 \alpha_n = 2\pi n -\iu \log(\lambda;\omega)\,,
\]
the matrix~$\stirmat$ is the Stirling matrix from~\eqref{def:stirlingmatrix}, 
and~$H(\iu\alpha_n)$ is the two-band matrix defined in~\eqref{eq:def_H}. In 
Section~\ref{sec:proof_zero} we provided the space of solutions 
of~\eqref{eq:diff_f_part} in the case that~$\alpha_n=0$. In this section we will solve the differential equation~\eqref{eq:diff_f_part} for~$\alpha_n\not=0$. 

As in Section~\ref{sec:proof_zero}, we will see that the solutions 
of~\eqref{eq:diff_f_part} are based on the functions~$\fu$ and~$\futoo$, defined in Section~\ref{sec:IandKpart1}.
The appearance of~$\alpha_n^2$ on the left hand side of~\eqref{eq:diff_f_part} will \emph{a priori} lead to the wish to 
evaluate~$\fu$ and~$\futoo$ with~$\alpha_n$ in the second argument.
However, $\alpha_n$ might not be in $\C\setminus\omega\R_{>0}$, which is the domain for the second argument of~$\fu$ and~$\futoo$.
One way to overcome this issue is to reconsider the choice of~$\omega$.
For $\lambda=1$, however, choosing for~$\omega$ a value not in~$\R_{<0}$ would render Theorem~\ref{thm:intro_single} (Theorem~\ref{Fourier:main_theorem}) incompatible with the classical result. See~\eqref{eq:untw_exp_intro} and~\eqref{eq:untw_coeff_intro}.
We refer to Section~\ref{sec:motivation} for more discussion.
Another way to overcome this issue is to take $-\alpha_n$ as the 
second argument of~$\fu$ and~$\futoo$.
Such a sign change might seem to be rather innocent at first glance, but it influences the solution.

We will use here this second option and will first consider a more conceptual 
variant of the differential equation in~\eqref{eq:diff_f_part}.
Indeed, we will provide solutions for the (vector-valued) differential equation
\begin{equation}\label{eq:diff_concept}
 \left( y^2\partial_y^2 + s(1-s) - \alpha^2y^2\right) w(y) = y^2 
H\bigl((-1)^\eps\iu\alpha\bigr)w(y)\,,\quad y\in\R_{>0}\,,
\end{equation}
for any~$\alpha\in\C\setminus\omega\R_{>0}$, $\eps\in\{0,1\}$ and~$s\in\C$. As before, we will ultimately use $w(y) = \stirmat^{-1}\hat f_n(y)$.

After establishing the full space of solutions of~\eqref{eq:diff_concept}, we 
will show how this result and the one from Section~\ref{sec:proof_zero} combine to a proof of Theorem~\ref{Fourier:main_theorem} and briefly indicate how we may retrieve the classical, untwisted result from this theorem.

\subsection{Particular solutions}

In this section we will present two solutions of the differential 
equation~\eqref{eq:diff_concept} for any~$\alpha\in\C\setminus\omega\R_{>0}$, 
$\alpha\not=0$ and any $\eps\in\{0,1\}$.
We consider $s\in\C$ to be fixed throughout. 

For the solution of~\eqref{eq:diff_concept} we will take advantage of the fact that it is a system of differential equations that is only sparsely entangled as a result of our base change from Section~\ref{s:system}.
Indeed, the differential operator on the left hand side of~\eqref{eq:diff_concept} acts coordinate-wise, and the matrix~$H\bigl( (-1)^\eps \iu \alpha\bigr)$ on the right hand side of~\eqref{eq:diff_concept} is non-vanishing only on the first two upper diagonal bands.
Therefore the last row of~\eqref{eq:diff_concept} depends only on the last row, $w_d$, of $w=(w_i)_{i=1}^d$, and can be solved independently.
The penultimate row of~\eqref{eq:diff_concept} depends only on the last two entries of~$w$, hence on~$w_{d-1}$ and~$w_d$, and can be solved after having found the solution of the differential equation in the last row. 
In total, we may solve~\eqref{eq:diff_concept} in this ``bottom-up'' procedure. 
The final result is stated as Proposition~\ref{prop:sol_simp_all}.

We start with a series of lemmas in which we discuss some crucial relations 
between the functions~$\fu$ and~$\futoo$ for different values of their first 
argument.
In hindsight these results will be understood as the building blocks 
of the intermediate solutions of the separated rows of~\eqref{eq:diff_concept}. 
We recall from~\eqref{eq:def_fu_general} and~\eqref{eq:def_futoo_general} that 
\[
 \fu(0, \alpha, y, s) = y^{\frac12} I_{s-\frac{1}{2}} (\alpha y; \omega)\,, 
\quad \futoo(0, \alpha, y, s) = y^{\frac12} K_{s-\frac{1}{2}} (\alpha y; \omega)
\]
for all~$\alpha\in\C\setminus\omega\R_{\geq0}$ and $y\in\R_{>0}$.

\begin{lemma}\label{zeroth} 
Let  $\alpha \in \C\setminus\omega\R_{\geq 0}$. Then the two functions
$\fu(0, \alpha, \cdot, s)$ and $\futoo(0, \alpha, \cdot, s)$ form a fundamental 
system of solutions of the second order differential equation
\begin{equation}\label{eq:zeroth000}
(y^2 \pa_y^2+ s(1-s)- y^2 \alpha^2)u(y) = 0\,, \quad y \in \R_{>0}\,.
\end{equation}
\end{lemma}

\begin{proof}
As it is well-known, the two functions $\R_{>0}\to \C$, 
\[
 y\mapsto y^{\frac12}I_{s-\frac12}(y)\qquad\text{and}\qquad y\mapsto 
y^{\frac12}K_{s-\frac12}(y)
\]
form a fundamental system of solutions of the differential equation
\begin{equation}\label{eq:classode}
 (y^2 \pa_y^2 + s(1-s) - y^2)u(y) = 0
\end{equation}
on~$\R_{>0}$. (A proof can be deduced, e.g., 
from~\cite[Section~3.7]{Watson44}.) 
Using the relations between the modified Bessel functions~$I_{s-\frac12}$ 
and~$K_{s-\frac12}$ (using the standard complex logarithm) and their 
variants~$I_{s-\frac12}(\,\cdot\,;\omega)$ 
and~$K_{s-\frac12}(\,\cdot\,;\omega)$ 
(using the adapted choice of a complex logarithm) as discussed in 
Section~\ref{sec:Bessel} and the fact that also $y^{\frac12}I_{\frac12-s}(y)$ 
solves~\eqref{eq:classode}, one easily shows that the 
functions~$y^{\frac12}I_{s-\frac12}(\alpha y;\omega)$ 
and~$y^{\frac12}K_{s-\frac12}(\alpha y;\omega)$ are linearly independent 
solutions of~\eqref{eq:zeroth000}. This completes the proof.
\end{proof}

For the proofs of the following two lemmas we note that  
\begin{equation}\label{Phi_connection}
\iu^m \fu(m, \alpha, y, s) =  \pa_\alpha^m \fu(0, \alpha, y, s)\,, \quad \iu^m 
\futoo(m, \alpha, y, s) =  \pa_\alpha^m \futoo(0, \alpha, y, s)  
\end{equation}
for all $\alpha \in \C\setminus\omega\R_{\geq0}$ and $m \in \N_0$, as can be 
seen directly from the definitions in~\eqref{eq:def_fu_general} 
and~\eqref{eq:def_futoo_general}. We further recall that for all~$m\in\N_0$ 
and~$s\in\C$ the maps $\fu(m,\cdot,\cdot,s)$ and~$\futoo(m,\cdot,\cdot,s)$ are 
smooth as functions of two variables.

\begin{lemma}\label{oneth}
Let $\alpha \in \C\setminus\omega\R_{\geq 0}$ and $\fuPsi\in\{\fu,\futoo\}$. 
Then 
$\fuPsi(1, \alpha, \cdot , s)$ satisfies 
\[
(y^2\pa_y ^2+ s(1-s)- y^2 \alpha^2)  \fuPsi(1, \alpha, y, s) = - 2 \iu \alpha 
y^2 \fuPsi(0, \alpha, y, s)\,, \quad y \in \R_{>0}\,.
\]
\end{lemma}

\begin{proof}
Taking advantage of Lemma~\ref{zeroth} and of~\eqref{Phi_connection}, we find
\begin{align*}
0 & =  \pa_\alpha \left[ (y^2 \pa_y ^2 + s(1-s)- y^2 \alpha^2)  \fuPsi(0, 
\alpha, y, s) \right] \\
&   = (y^2 \pa_y ^2 + s(1-s)- y^2 \alpha^2) (\pa_\alpha \fuPsi(0, \alpha, y, 
s)) 
- 2 \alpha y^2   \fuPsi(0, \alpha, y, s) \\ 
& = (y^2 \pa_y ^2 + s(1-s)- y^2 \alpha^2)  \iu \fuPsi(1, \alpha, y, s) - 2 
\alpha y^2  \fuPsi(0, \alpha, y, s) 
\end{align*}
for all~$y\in\R_{>0}$. Re-arranging this equality completes the proof.
\end{proof}

While Lemma~\ref{oneth} discusses a relation between the functions~$\fuPsi(0,\alpha,\cdot,s)$ 
and~$\fuPsi(1,\alpha,\cdot,s)$ for $\fuPsi\in\nobreak\{\fu,\futoo\}$, the next 
lemma considers the situation for the first argument of~$\fuPsi$ being at 
least~$2$.

\begin{lemma}\label{lem:genth}
Let $\alpha \in \C\setminus\omega\R_{\geq0}$, $m\in\N$ with $m\geq 2$ and 
$\fuPsi\in\{\fu,\futoo\}$. Then the function~$\fuPsi(m, \alpha, \cdot, s)$ 
satisfies the 
identity
\begin{align*}
(y^2 \pa_y ^2  + s(1 &-s) - \alpha^2 y^2 ) \fuPsi(m, \alpha, y, s)  \\& = 
-y^2\bigl( 2 m \iu \alpha  \fuPsi(m-1,\alpha,y,s) + m(m-1)  
\fuPsi(m-2,\alpha,y,s)\bigr)
\end{align*}
for all $y\in\R_{>0}$.
\end{lemma}

\begin{proof}
Let $y\in\R_{>0}$. Using the Leibniz rule and~\eqref{Phi_connection} we find
\begin{align*}
\pa_\alpha^m &\bigl( \alpha^2 \fuPsi(0, \alpha, y,s)  \bigr) 
\\
& = \sum_{k=0}^m \binom{m}{k} \bigl(\pa_\alpha ^k 
\alpha^2\bigr)\bigl(\pa_\alpha 
^{m-k} \fuPsi(0, \alpha, y, s) \bigr)
\\
& = \alpha^2  \pa_\alpha ^m \fuPsi(0, \alpha, y, s) + 2 m \alpha  \pa_\alpha 
^{m-1} \fuPsi(0, \alpha, y, s) 
\\
& \hphantom{\alpha^2  \pa_\alpha ^m \fuPsi(0, \alpha, y, s) + 2 m \alpha}
+ m(m-1) \pa_\alpha ^{m-2}  \fuPsi(0, \alpha, y, s)
\\
& = \iu^m \alpha^2 \fuPsi(m, \alpha, y, s) + \iu^{m-1} 2m\alpha  
\fuPsi(m-1,\alpha,y, s)
\\
& \hphantom{\fuPsi(m, \alpha, y, s)+ \iu^{m-1} 2m\alpha  } - \iu^{m} m(m-1) \fuPsi(m-2,\alpha,y, s)\,.
\end{align*}
Combining this equality and Lemma~\ref{zeroth}, we obtain 
\begin{align*}
0  & =  \pa_\alpha ^m \left[ (y^2 \partial_y^2 + s(1-s)- y^2 \alpha^2)  
\fuPsi(0, \alpha, y, s) \right]
\\
& = (y^2\pa_y ^2+ s(1-s)- \alpha^2 y^2 )  \iu^m \fuPsi(m, \alpha, y, s)  
\\
& \qquad - y^2 \left(2 m \alpha \iu^{m-1}  \fuPsi(m-1,\alpha,y, s) - 
m(m-1)\iu^{m}  \fuPsi(m-2,\alpha,y, s) \right)\,.
\end{align*}
Re-arranging (and dividing by $\iu^m$) yields the claimed identity. 
\end{proof}

We recall the sign matrix $S\in\Mat(d\times d;\C)$ from~\eqref{eq:def_S}:
\[
 S \coloneqq 
 \begin{pmatrix}
  (-1)^{d-1}
  \\
  & (-1)^{d-2} 
  \\
  & & \ddots 
  \\
  & & & -1 
  \\
  & & & & 1 
  \\
  & & & & & -1 
  \\
  & & & & & & 1
 \end{pmatrix}\,.
\]

\begin{prop}\label{prop:sol_simp_all}
Let $\alpha \in \C \setminus\omega\R_{\geq 0}$, $\eps\in\{0,1\}$ and 
$\fuPsi\in\{\fu,\futoo\}$.  Then the map $w\colon\R_{>0}\to V$, 
\[
w \coloneqq   S^\eps
\begin{pmatrix}
\fuPsi(d-1,\alpha,\cdot,s) 
\\
\vdots 
\\
\fuPsi(1,\alpha,\cdot,s) 
\\
\fuPsi(0,\alpha,\cdot,s)
\end{pmatrix}
\]
is a solution of the second order differential equation
\begin{equation}\label{eq:sysODE_w}
 \bigl( y^2\partial_y^2 + s(1-s) - \alpha^2y^2\bigr)w(y) = y^2H\bigl((-1)^\eps 
\iu\alpha\bigr)w(y)\,,\quad y\in\R_{>0}\,.
\end{equation}
\end{prop}

Before we provide the proof of this proposition, we briefly elaborate on 
its statement.

\begin{remark} Using the notation of Proposition~\ref{prop:sol_simp_all} we 
note that  
\[
 S \begin{pmatrix}
\fuPsi(d-1,\alpha,\cdot,s) 
\\
\vdots 
\\
\fuPsi(3,\alpha,\cdot,s) 
\\
\fuPsi(2,\alpha,\cdot,s) 
\\
\fuPsi(1,\alpha,\cdot,s) 
\\
\fuPsi(0,\alpha,\cdot,s)
\end{pmatrix}
=  
\begin{pmatrix}
(-1)^{d-1}\fuPsi(d-1,\alpha,\cdot,s) 
\\
\vdots 
\\
-\fuPsi(3,\alpha,\cdot,s) 
\\
\hphantom{-}\fuPsi(2,\alpha,\cdot,s) 
\\
-\fuPsi(1,\alpha,\cdot,s) 
\\
\hphantom{-}\fuPsi(0,\alpha,\cdot,s)
\end{pmatrix}\,.
\]
Thus, Proposition~\ref{prop:sol_simp_all} states that 
\[
\begin{pmatrix}
\fuPsi(d-1,\alpha,\cdot,s) 
\\
\vdots 
\\
\fuPsi(3,\alpha,\cdot,s) 
\\
\fuPsi(2,\alpha,\cdot,s) 
\\
\fuPsi(1,\alpha,\cdot,s) 
\\
\fuPsi(0,\alpha,\cdot,s)
\end{pmatrix}
\]
is a solution of 
\[
 \bigl( y^2\partial_y^2 + s(1-s) - \alpha^2y^2\bigr)w(y) = 
y^2H\bigl(\iu\alpha\bigr)w(y)\,,\quad y\in\R_{>0}\,,
\]
and 
\[
\begin{pmatrix}
(-1)^{d-1}\fuPsi(d-1,\alpha,\cdot,s) 
\\
\vdots 
\\
-\fuPsi(3,\alpha,\cdot,s) 
\\
\hphantom{-}\fuPsi(2,\alpha,\cdot,s) 
\\
-\fuPsi(1,\alpha,\cdot,s) 
\\
\hphantom{-}\fuPsi(0,\alpha,\cdot,s)
\end{pmatrix}
\]
is a solution of 
\[
 \bigl( y^2\partial_y^2 + s(1-s) - \alpha^2y^2\bigr)w(y) = 
y^2H\bigl(-\iu\alpha\bigr)w(y)\,,\quad y\in\R_{>0}\,.
\]
\end{remark}

\begin{proof}[Proof of Proposition~\ref{prop:sol_simp_all}]
Let $w = (w_j)_{j=1}^d$. Then 
\begin{equation}\label{eq:formula_w}
 w_{d-m}\colon\R_{>0} \to \C\,,\quad w_{d-m}(y) = (-1)^{\eps 
m}\fuPsi(m,\alpha,y,s)
\end{equation}
for each~$m\in\{0,\ldots, d-1\}$. To show that $w$ is a solution of the system 
of differential equations~\eqref{eq:sysODE_w}, we start to check this for the 
last row (i.e., for~$w_{d}$ or, equivalently, for~$m=0$) and then proceed 
iteratively up to the first row (i.e., to~$w_1$ or, equivalently, to~$m=d-1$). 

The last row of~\eqref{eq:sysODE_w}, thus the equation for~$w_d$, reads
\begin{equation}\label{eq:lastrow}
 \left( y^2\partial_y^2 + s(1-s) - \alpha^2y^2\right)w_d(y) = 0
\end{equation}
for all~$y\in\R_{>0}$. Since $w_d = \fuPsi(0,\alpha,\cdot,s)$, 
Lemma~\ref{zeroth} shows that the equality~\eqref{eq:lastrow} is satisfied. 

The penultimate row of~\eqref{eq:sysODE_w}, thus the equation for~$w_{d-1}$, is
\begin{equation}\label{eq:rowm1}
 \left( y^2\partial_y^2 + s(1-s) - \alpha^2y^2\right)w_{d-1}(y) = 2 
(-1)^{\eps+1}\iu 
\alpha y^2 w_d(y)
\end{equation}
for all~$y\in\R_{>0}$. Since $w_{d-1} = (-1)^\eps\fuPsi(1,\alpha,\cdot,s)$ 
and~$w_d=\fuPsi(0,\alpha,\cdot, s)$, Lemma~\ref{oneth} yields 
that~\eqref{eq:rowm1} is valid.

For row~$d-m$ of~\eqref{eq:sysODE_w} with $m\in\{2,\ldots, d-1\}$ we need to 
establish 
\begin{align}\label{eq:rowgen}
\begin{aligned}
\big( y^2\partial_y^2 &  + s(1 - s)  - \alpha^2y^2\big)w_{d-m}(y) 
\\
& = 2 (-1)^{\eps+1}\iu \alpha y^2 m w_{d-(m-1)}(y) - m(m-1)y^2w_{d-(m-2)}(y)
\end{aligned}
\end{align}
for all~$y\in\R_{>0}$. From Lemma~\ref{lem:genth} it follows that 
\begin{align*}
\big( y^2\partial_y^2 & + s(1  -s)  - \alpha^2y^2\big)(-1)^{\eps 
m}\fuPsi(m,\alpha,y,s)
\\
& = 2 (-1)^{\eps m  + 1}\iu \alpha y^2 m \fuPsi(m-1,\alpha,y,s) 
\\
& \qquad\qquad - (-1)^{\eps m} 
m(m-1) y^2 \fuPsi(m-2,\alpha, y,s)
\\
& = 2 (-1)^{\eps + 1} \iu \alpha y^2 m (-1)^{\eps(m-1)}\fuPsi(m-1,\alpha, y,s) 
\\
& \qquad\qquad - m(m-1)y^2(-1)^{\eps(m-2)}\fuPsi(m-2,\alpha,y,s)
\end{align*}
for all~$y\in\R_{>0}$. Using the formulas in~\eqref{eq:formula_w} for~$w$ 
completes the proof that \eqref{eq:rowgen} is satisfied.
\end{proof}

\subsection{Space of all solutions}

With these preparations and the asymptotics for~$\fu$ and~$\futoo$ from 
Section~\ref{sec:growthIK} we are now able to present the full space of 
solutions of~\eqref{eq:diff_concept}.

\begin{prop}\label{prop:basis_sol_nonzero}
Let $\alpha \in \C \setminus\omega\R_{\geq 0}$, $\eps\in\{0,1\}$, $s\in\C$ and 
define the maps $\PI,\PK\colon\R_{>0}\to V$ by
\[
\PI \coloneqq   S^\eps
\begin{pmatrix}
\fu(d-1,\alpha,\cdot,s) 
\\
\vdots 
\\
\fu(1,\alpha,\cdot,s) 
\\
\fu(0,\alpha,\cdot,s)
\end{pmatrix}
\qquad\text{and}\qquad
\PK \coloneqq   S^\eps
\begin{pmatrix}
\futoo(d-1,\alpha,\cdot,s) 
\\
\vdots 
\\
\futoo(1,\alpha,\cdot,s) 
\\
\futoo(0,\alpha,\cdot,s)
\end{pmatrix}\,.
\]
Then the $2d$-dimensional space of solutions of the second order differential 
equation
\begin{equation}\label{eq:solODE_all}
 \bigl( y^2\partial_y^2 + s(1-s) - \alpha^2y^2\bigr)w(y) = y^2H\bigl((-1)^\eps 
\iu\alpha\bigr)w(y)\,,\quad y\in\R_{>0}\,,
\end{equation}
is 
\begin{equation*}\label{eq:solspace_nonzero}
 L_1 \coloneqq \left\{ C\PI + D\PK \ : \ C,D \in \Coeff_{1;d} \right\}\,,
\end{equation*}
with $\Coeff_{1;d}$ defined as in Section~\ref{sec:coeffmatrices}.
\end{prop}

\begin{proof}
By Proposition~\ref{prop:sol_simp_all}, the maps~$\PI$ and~$\PK$ are solutions 
of~\eqref{eq:solODE_all}. Further, a straightforward calculation shows that each 
 matrix in~$\Coeff_{1;d}$ commutes with~$H\bigl((-1)^\eps \iu\alpha\bigr)$. 
Therefore, each element in~$L_1$ is indeed a solution of~\eqref{eq:solODE_all}. 
(The vector space~$\Coeff_{1;d}$ is the full centralizer of~$H\bigl((-1)^\eps 
\iu\alpha\bigr)$, a fact that we will not need here.) It remains to show that 
these are all solutions. As $L_1$ is a space of solutions of a $d$-dimensional 
system of second order linear differential equations, the (complex) vector space 
dimension of~$L_1$ is at most~$2d$. In what follows we will detect $2d$ linearly 
independent elements in~$L_1$, which then immediately yields that $L_1$ is 
$2d$-dimensional and hence the full space of solutions 
of~\eqref{eq:solODE_all}. 
To that end we consider the $d$ elements $B_1,\ldots, B_d\in\Coeff_{1;d}$, where 
$B_j = (b^{(j)}_{mn})$ for $j\in\{1,\ldots, d\}$ is determined by
\[
 b^{(j)}_{1n}\coloneqq
 \begin{cases}
  1 & \text{for $n=j$}\,,
  \\
  0 & \text{for $n\not=j$}\,.
 \end{cases}
\]
We recall from Section~\ref{sec:coeffmatrices} that indeed the full matrix~$B_j$ 
is determined by this prescription. We now show that the $2d$ elements 
\begin{equation}\label{eq:elem_nonzero2}
 B_1\PI\,,\ \ldots\,,\ B_d\PI\,,\ B_1\PK\,,\ \ldots\,,\ B_d\PK
\end{equation}
of~$L_1$ are linearly independent (over~$\C$). The topmost entry of the linear 
equation
\[
 \sum_{j=1}^d \mu_j B_j\PI + \nu_j B_j\PK = 0\,,
\]
where $\mu_j,\nu_j\in\C$, reads
\[
 \sum_{j=1}^d \mu_j\fu(d-j,\alpha,\cdot, s) + \nu_j\futoo(d-j,\alpha,\cdot, s) = 
0\,.
\]
The asymptotics for $\fu$ and~$\futoo$ from Propositions~\ref{prop:asymp_fu} 
and~\ref{prop:asymp_futoo} now imply that 
\[
\mu_j=0=\nu_j
\]
for all $j\in\{1,\ldots, d\}$. This completes the proof.
\end{proof}

\subsection{Endgame}

Even though we have already indicated the proof of 
Theorem~\ref{Fourier:main_theorem} in the previous sections, we provide here the 
brief, explicit discussion for the convenience of the readers. We further 
explain why Theorem~\ref{Fourier:main_theorem} recovers the classical result in 
the untwisted setting (see Section~\ref{sec:motivation}). 

\begin{proof}[Proof of Theorem~\ref{Fourier:main_theorem}]
Proposition~\ref{simplified_system} shows that, for each~$n\in\Z$, the Fourier 
coefficient function~$\hat f_n$ satisfies the differential equation
\[
\bigr(y^2 \partial^2_y + s(1-s) - y^2\alpha_n^2 \bigl) T^{-1}  \hat f_n(y)  =  
y^2 H(\iu \alpha_n) T^{-1}   \hat f_n(y)\,,\quad y\in\R_{>0}\,,
\]
with 
\[
 \alpha_n = 2\pi n - \iu \log(\lambda;\omega)\,.
\]
If $\alpha_n=0$, then Proposition~\ref{prop:basis_sol_zero} states that $\hat 
f_n$ is of the form claimed in Theorem~\ref{Fourier:main_theorem}. For 
$\alpha_n\not=0$, we rewrite the differential equation as 
\[
 \bigr(y^2 \partial^2_y + s(1-s) - y^2\tilde\alpha_n^2 \bigl) T^{-1}  \hat 
f_n(y)  =  y^2 H((-1)^{\eps_n} \iu \tilde\alpha_n) T^{-1}   \hat f_n(y)\,,\quad 
y\in\R_{>0}\,,
\]
with (see~\eqref{eq:def_tildealpha})
\[
 \tilde\alpha_n = (-1)^{\eps_n}\alpha_n\,.
\]
This has the advantage that $\tilde\alpha_n\in\C\setminus\omega\R_{>0}$ by the 
choice of~$\eps_n$ (see~\eqref{eq:pickeps}). Then 
Proposition~\ref{prop:basis_sol_nonzero} shows that $\hat f_n$ is of the claimed 
form. 
\end{proof}

\bibliography{fourier_bib} 
\bibliographystyle{amsplain}

\setlength{\parindent}{0pt}

\end{document}